\documentclass{article}
\usepackage[T1]{fontenc}

\usepackage{amsthm, amsmath, amssymb}
\usepackage{Baskervaldx, mathpazo}

\usepackage[style=alphabetic, backend=biber]{biblatex}
\usepackage{enumerate}
\usepackage[top=3cm,bottom=3cm,inner=3cm,outer=3cm]{geometry}
\usepackage{tikz-cd}
\usepackage{eucal}
\usepackage{quiver}
\newcommand{\mcal}[1]{\mathcal{#1}}
\newtheorem{thm}{Theorem}[section]
\newtheorem{cor}[thm]{Corollary}
\newtheorem{prop}[thm]{Proposition}
\newtheorem{lem}[thm]{Lemma}

\newtheorem{ques}[thm]{Question}

\newtheorem*{thm*}{Theorem}
\newtheorem*{cor*}{Corollary}

\theoremstyle{definition}
\newtheorem{defn}[thm]{Definition}
\newtheorem*{defn*}{Definition}

\newtheorem{exam}[thm]{Example}
\newtheorem{rmk}[thm]{Remark}
\theoremstyle{remark}

\newcommand{\Set}{\mathcal{S}et}
\newcommand{\Spaces}{\mathcal{S}}

\newcommand{\Spectra}{\mathrm{Sp}}
\DeclareMathOperator{\Grp}{\mathcal{G}rp}
\DeclareMathOperator{\Grpd}{\mathcal{G}pd}
\DeclareMathOperator{\Alg}{Alg}

\newcommand{\BiMod}[2]{{}_{#1}\mathrm{BiMod}_{#2}}
\DeclareMathOperator{\Fun}{Fun}

\DeclareMathOperator{\Catnoinf}{Cat}
\newcommand{\Nrml}{\mathrm{Nml}}
\newcommand{\Ass}{\mathcal{A}ssoc}

\DeclareMathOperator{\Mon}{Mon}

\DeclareMathOperator{\RMod}{RMod}
\DeclareMathOperator{\RMon}{RMon}

\newcommand{\im}{\mathrm{im}}
\newcommand{\Conn}{\mathrm{Conn}}
\newcommand{\Trunc}{\mathrm{Trunc}}

\DeclareFontFamily{U}{min}{}
\DeclareFontShape{U}{min}{m}{n}{<-> udmj30}{}

\DeclareMathOperator{\colim}{colim}
\DeclareMathOperator{\fib}{fib}
\DeclareMathOperator{\cof}{cof}
\newcommand{\op}{\mathrm{op}}

\newcommand{\cech}{\check{C}}
\newcommand{\Barc}{\mathrm{B}_\bullet}
\newcommand{\Barcplus}{\mathrm{B}_\bullet^+}
\newcommand{\Act}{\mathrm{Act}_\bullet}
\newcommand{\Actplus}{\mathrm{Act}^+_\bullet}

\newcommand{\Cut}{\mathrm{Cut}}

\newcommand{\CutR}{\mathrm{Cut}^R}
\newcommand{\Res}{\mathrm{Res}}
\newcommand{\Ind}{\mathrm{Ind}}
\newcommand{\BB}{\mathrm{B}}
\newcommand{\ncl}{\mathrm{ncl}}

\newcommand{\CC}{\mathcal{C}}
\newcommand{\DD}{\mathcal{D}}
\newcommand{\EE}{\mathbb{E}}

\newcommand{\ZZ}{\mathbb{Z}}
\newcommand{\XX}{\mathfrak{X}}
\newcommand{\YY}{\mathfrak{Y}}
\newcommand{\KK}{\mathcal{K}}

\newcommand{\PP}{\mathcal{P}}
\newcommand{\sph}{\mathbb{S}}
\renewcommand{\epsilon}{\varepsilon}
\usepackage{hyperref}
\hypersetup{
	colorlinks = true, 
	urlcolor     = blue, 
	linkcolor    = black, 
	citecolor   = red 
}
\title{Higher Groups and Higher Normality}
\author{Jonathan Beardsley and Landon Fox}
\begin{document}
	\maketitle
	\begin{abstract}
		In this paper we continue Prasma's homotopical group theory program by considering homotopy normal maps in arbitrary $\infty$-topoi. We show that maps of group objects equipped with normality data, in Prasma's sense, are algebras for a ``normal closure'' monad in a way which generalizes the standard loops-suspension monad. We generalize a result of Prasma by showing that monoidal functors of $\infty$-topoi preserve normal maps or, equivalently, that monoidal functors of $\infty$-topoi preserve the property of ``being a fiber'' for morphisms between connected objects. We also formulate Noether's Isomorphism Theorems in this setting, prove the first two of them, and provide counterexamples to the third.

		Accomplishing these goals requires us to spend substantial time synthesizing existing work of Lurie so that we may rigorously talk about group objects in $\infty$-topoi in the ``usual way.'' One nice result of this labor is the formulation and proof of an Orbit-Stabilizer Theorem for group actions in $\infty$-topoi.
	\end{abstract}

	\tableofcontents
	\section{Introduction}

	A general theme of homotopical algebra is to take structures typically carried by sets, e.g.~group, ring or module structures, and replace the sets with homotopy types, or so-called $\infty$-groupoids. One formalism for this approach is to work with $\infty$-categories as developed by Joyal, Lurie, Rezk and many others. We heavily use Lurie's work on the topic as recorded in \cite{htt,ha}. For the sake of the involved reader who may be going back and forth between our document and those references, we also follow Lurie's notational conventions whenever possible.

	An orthogonal generalization of ``set-based mathematics'' is to replace the category of sets with a topos. Recall that a (Grothendieck) topos is, roughly, a category of \textit{sheaves} of sets. This latter perspective has been fruitful, for instance, in algebraic geometry, topology, and mathematical logic. The joint generalization of these two ideas is to do mathematics in $\infty$-topoi as defined in \cite{htt}, which behave something like $\infty$-categories of (presheaves) of $\infty$-groupoids. The $\infty$-category of $\infty$-groupoids is of course an $\infty$-topos, as it is sheaves over the point. However, one can also build $\infty$-topoi from Quillen model categories of simplicial (pre)sheaves of the sort that arise in, say, algebraic geometry and motivic homotopy theory.

	In all of these settings, one is often interested in objects that are \textit{acted on} by whatever the relevant analogue of a group is. For instance, in homotopy theory we are often interested in actions of loop spaces, e.g.~$S^1$, on other spaces. In algebraic geometry, especially in Galois theory, we can frequently encode useful mathematical structure via actions of group schemes or group stacks. Classically, understanding groups themselves aids in the understanding of their representations and the ways in which they can act on sets, spaces, or schemes. However, a general \textit{theory of higher groups} has not yet been fully developed. Nonetheless, important steps have been made in understanding higher group theory, primarily by Prasma and collaborators \cite{prasmahmtpynorm, prasma-segalgroups, prasmaSchlank-sylow}. Higher group theory, both in $\infty$-groupoids and in other $\infty$-topoi, has also been investigated by authors interested in mathematical physics, especially from the perspective of higher \textit{principal $G$-bundles} and topological field theories e.g.~in \cite{nikolausSchreiberStevenson-principal1,nikolausSchreiberStevenson-principal2, gripaios-generalized-symmetries, bunk-smooth_string}. See \cite{bunk-inftybundles_survey} for an overview of these sorts of applications.

	The work of Prasma's which is most relevant to us is \cite{prasmahmtpynorm}, which introduces the idea of a \textit{normality datum} for a map of loop spaces. That work is in turn based on earlier work of Farjoun and Segev \cite{farjounSegev-homotopynormal} and augmented by concomitant work of Hess and Farjoun \cite{hessfarjoun-conormal}. The crucial observation there is that in higher group theory one should replace the \textit{property} of being a normal subgroup with \textit{data}. This data is, essentially, the data of a group structure on the quotient by the ``subgroup.'' We will see that in this setting, unlike in discrete group theory, there may be more than one such group structure, necessitating the use of data instead of property. In other words, a ``subgroup'' can be normal in more than one way. Another significant departure from the classical theory is that ``sub-objects'' are not particularly well behaved in homotopical mathematics and ``normal subgroups'' are replaced by ``normal maps.''

	In this paper, we take Prasma's definition of normality seriously but generalize it to an arbitrary $\infty$-topos and explore its ramifications in higher group theory. Some of these ramifications include: formulating and proving a higher Orbit-Stabilizer Theorem (Theorem \ref{thm:orbit stabilizer}); proving that group maps equipped with normality data are the algebras for the ``normal closure'' monad (Theorem \ref{thm:normal closure functor}); formulating and proving Noether's First and Second Isomorphism Theorems (Theorems \ref{thm:first iso theorem} and \ref{thm:secondisotheorem}); and providing counterexamples to (the natural generalization of) Noether's Third Isomorphism Theorem (Section \ref{sec:third iso}).

	We warn the reader that group theory in this setting should not be thought of as a ``topological'' generalization of group theory. For instance, even the group objects in the $\infty$-topos of ``spaces'' do not have an underlying topological space at all, but rather an underlying homotopy type. Instead, we should think of higher group theory as a kind of \textit{deformation} of discrete group theory (in the same way that the sphere spectrum is a deformation of the integers). One piece of evidence for this intuition is the fact that our constructions all become the standard group-theoretic constructions after applying $\pi_0$.

	\subsubsection*{Acknowledgements}

	The authors thank Jake Bian, Bastiaan Cnossen, Aras Ergus, Sonja Farr, Kiran Luecke, Eric Peterson, Maxime Ramzi, Lorenzo Riva, and Chris Rogers for many helpful conversations regarding the material of this paper. Kiran Luecke in particular is responsible for many of the counterexamples. The first author would also like to extend a special thanks to Matan Prasma for a number of delightful conversations about higher group theory that occurred during the 2015 Young Topologists' Meeting at EPFL. Though it has been almost 10 years since then, those ideas have never fully disappeared. While working on this paper, the first author was partially supported by a Simons Foundation collaboration grant, Award ID \#853272. Much of the work of this paper is based on the second author's master's thesis at the University of Nevada, Reno.

	\section{Background on  $\infty$-topoi}

	We briefly remind the reader of some of the central concepts in the theory of $\infty$-topoi as developed in \cite{htt}. Throughout, $\XX$ and $\YY$ will always denote $\infty$-topoi. Recall that for any $\infty$-category $\CC$ we can form the $\infty$-category of presheaves (valued in $\infty$-groupoids $\Spaces$) which we will denote $\PP(\CC)$. An $\infty$-topos is defined to be an $\infty$-category $\XX$ which admits a fully faithful right adjoint functor $\XX\hookrightarrow \PP(\CC)$ whose left adjoint $L\colon \PP(\CC)\to\XX$ is \textit{accessible} and \textit{left exact}. We will avoid discussing accessibility in this paper as it will not be immediately relevant but we remind the reader that being left exact means that $L$ preserves finite limits (the reader should be reminded of the left exact functors of abstract algebra that preserve kernels, which are finite limits).

	Central to the theory of $\infty$-categories is the notion of \textit{effective epimorphisms}, which should be thought of as covers. These will occur frequently throughout this work so we recall them now.

	\begin{defn}[{\cite[Corollary 6.2.3.5]{htt}}]\label{defn: cech nerve and effective epis}
		Let $f\colon X\to Y$ be a morphism in $\XX$, thought of as an object of $\XX^{\Delta^{1}}$. Noticing that there is an inclusion $\Delta^{1}\simeq\{[0]\to[-1]\}\subset\Delta_+^{\op}$ of $\Delta^1$ into the augmented simplex category, we call the right Kan extension of $f$ along this inclusion the \textit{augmented \v Cech nerve of $f$}. The (unaugmented) \v Cech nerve of $f$ is then taken to be the restriction of the augmented \v Cech nerve to the subcategory $\Delta^{\op}\subset\Delta^{\op}_+$. We say that $f$ is an effective epimorphism of $\XX$ if the colimit of the \v Cech nerve of $f$ is equivalent to $f$.
	\end{defn}

	It may be useful for the reader to have the following (standard) schematic depiction of the \v Cech nerve of a morphism $f\colon X\to Y$:
	\[\begin{tikzcd}
		Y & X & {X\times_{Y}X} & {X\times_Y X\times_Y X} & \cdots
		\arrow["f"', from=1-2, to=1-1]
		\arrow[shift right, from=1-3, to=1-2]
		\arrow[shift left, from=1-3, to=1-2]
		\arrow[from=1-4, to=1-3]
		\arrow[shift left=2, from=1-4, to=1-3]
		\arrow[shift right=2, from=1-4, to=1-3]
		\arrow[shift right, from=1-5, to=1-4]
		\arrow[shift left, from=1-5, to=1-4]
		\arrow[shift right=3, from=1-5, to=1-4]
		\arrow[shift left=3, from=1-5, to=1-4]
	\end{tikzcd}\]

	We will later need a restricted version of the \v Cech nerve whose codomain will be the group objects of interest in this paper. We will rarely, in this paper, take the \v Cech nerve of arbitrary morphisms in $\XX$. For that reason we reserve the symbol $\cech$, which often denotes the general \v Cech nerve functor in the literature, for later use.

	An important characterization of $\infty$-topoi is given in \cite[Theorem 6.1.0.6]{htt} as $\infty$-categories satisfying a version of Giraud's axioms. This characterization implies, among other things, that $\infty$-topoi admit all small limits and colimits and we will regularly use this fact without comment. We will also need the fact that every groupoid object in an $\infty$-topos is effective, but because this potentially requires some explanation, we will discuss it further when necessary.

	Just as classical topoi generalize the theory of sets, $\infty$-topoi generalize the theory of $\infty$-groupoids, or spaces. It is not surprising then that they can be equipped with an internal homotopy theory including things like homotopy groups and Postnikov towers. Some of the work in this paper is ``obviously true'' in $\Spaces$, the $\infty$-topos of $\infty$-groupoids, but requires some work to re-prove in full generality. To aid us, and the reader, in this generalization process, we recall in this section some of the ways in which an $\infty$-topos can be made to behave like $\Spaces$. The material of this section can all be found in \cite[Section 6.5.1]{htt}.

	To begin we must get clear on what we mean by $n$-connected, $n$-connect\textit{ive} and $n$-truncated, for both objects and morphisms. This terminology is not entirely standardized in the literature. Because so much of this work uses \cite{htt} and \cite{ha} as its foundational references, we follow the conventions therein.

	We will say that an $\infty$-groupoid $X$ is $n$-\textit{connective} if, for every point $x\in X$, the homotopy sets $\pi_i(X,x)$ are trivial for $0\leq i\leq n-1$. For intuition regarding this terminology, we suggest the reader keep the case of spectra or chain complexes in mind, where ``connective'' refers to those objects which have trivial homotopy groups in negative degrees. Therefore, assuming ``$0$-connective'' should correspond to ``connective,'' every non-empty $\infty$-groupoid should be $0$-connective. This then requires that $n$-connective $\infty$-groupoids can have non-trivial homotopy in degree $n$.

	We will then say that a \textit{morphism} $f\colon X\to Y$ of $\infty$-groupoids is $n$-\textit{connective} if its fibers are $n$-connective. In traditional algebraic topology, a morphism of spaces is called \textit{$n$-connected} if it has fibers which have trivial homotopy groups in degrees $0\leq i\leq n-1$. Therefore an $n$-connective morphism in our sense is the same as an $n$-connected morphism in the traditional sense. It is only that the terminology differs in indexing for individual spaces. One advantage of our  terminology (which is actually Lurie's, of course) is that the connectivity of the fiber of a map, and the map itself, have the same index.

	In the other direction, we say that an $\infty$-groupoid $X$ is $n$-truncated if, for every $x\in X$, the homotopy sets $\pi_i(X,x)$ are trivial for all $i>n$. We think of an $n$-truncated $\infty$-groupoid as having its homotopy groups \textit{cut down to} degree $n$, i.e.~they still include degree $n$. Then, similarly to above, we say that a morphism of $\infty$-groupoids $X\to Y$ is $n$-truncated if each of its fibers is an $n$-truncated $\infty$-groupoid.

	Note that the definitions of $n$-connective and $n$-truncated morphisms can also be formulated in terms of the long exact sequence in homotopy. Suppose that a morphism $f\colon X\to Y$ is $n$-connective. Then its fiber has trivial homotopy sets up through degree $n-1$ and thus the long exact sequence implies that $\pi_i(f)$ is an isomorphism for $0\leq i\leq n-1$ and that $\pi_n(f)$ is a surjection. Similarly, if $f$ is $n$-truncated then $\pi_i(f)$ is an isomorphism for $i>n+1$ and an injection for $i=n+1$. In other words, if $f$ is $n$-truncated then $X$ and $Y$ are ``the same'' \textit{after} degree $n$.

	We now follow \cite{htt} in making the corresponding definitions for an arbitrary $\infty$-topos. They are necessarily more complicated in this setting because objects of $\infty$-topoi are more like \textit{(pre)sheaves of $\infty$-groupoids} and therefore things like homotopy groups are substantially more complicated. An important difference is that connectivity and truncatedness of morphisms in an $\infty$-topos are \textit{not} determined, in general, by connectivity or truncatedness of the fibers of that morphism.

	\begin{defn}
		Let $\XX$ be an $\infty$-topos and let $-1\leq n$. Then we define the category of $n$-truncated objects $\XX^{\leq n}$ to be the full subcategory of $\XX$ spanned by objects with the property that $\XX(Z,X)$ is an $n$-truncated $\infty$-groupoid for every $Z\in \XX$. We write $\tau_{\leq n}\colon \XX\to \XX^{\leq n}$ for the left adjoint to the inclusion $\XX^{\leq n}\hookrightarrow\XX$, as described in \cite[Proposition 5.5.6.18]{htt}. For $n=-2$ we make the convention that the only object of $\XX^{\leq -2}$ is the terminal object and $\tau_{\leq -2}$ is the terminal functor. We call a morphism $f\colon X\to Y$ of $\XX$ $n$-truncated if for every $Z\in\XX$ the induced morphism of $\infty$-groupoids $\XX(Z,X)\xrightarrow{f_\ast}\XX(Z,Y)$ is $n$-truncated.
	\end{defn}

	\begin{rmk}\label{rmk:truncation is symm monoidal}
		Later, we will use the fact that $\infty$-topoi are symmetric monoidal with respect to the Cartesian product (cf.~\cite[Section 2.4.1]{ha}) and that the truncation functors are symmetric monoidal with respect to this structure, by virtue of \cite[Lemma 6.5.1.2]{htt}.
	\end{rmk}

	For the following definition it is necessary to recall that if $\XX$ is an $\infty$-topos and $X\in\XX$ then the slice category $\XX_{/X}$ is again an $\infty$-topos (cf.~\cite[Proposition 6.3.5.1]{htt}).

	\begin{defn}\label{defn:homotopygroups}
		Let $\XX$ be an $\infty$-topos and $X\in \XX$. Choosing a base point $p\colon\ast\to S^k$ of the $k$-sphere in $\Spaces$ induces, by cotensoring, a morphism $p^\ast\colon X^{S^k}\to X$ given by precomposition with $p$. We write $\pi_k(X)$ for $\tau_{\leq 0}(p^\ast)\in\XX_{/X}$. Note that if $f\colon X\to Y$ is a morphism of $\XX$ then it is an object of $\XX_{/Y}$ which is again an $\infty$-topos. Therefore it makes sense to define $\pi_k(f)\in(\XX_{/Y})_{/f}\simeq \XX_{/f}\simeq \XX_{/X}$.

		If $\eta\colon \ast\to X$ is a base point for $X$ then pulling back induces a functor $\eta^\ast\colon\XX_{/X}\to\XX_{/\ast}\simeq \XX$ and we could write $\eta^\ast\pi_k(X)$ for the resulting object of $\tau_{\leq 0}\XX$. Typically, we will have already assumed $X$ is a pointed object and in that case, so long as it will not create confusion, we will simply write $\pi_k(X)$.
	\end{defn}

	\begin{rmk}
		When $\XX=\Spaces$ the straightening functor of \cite[Theorem 3.2.0.1]{htt} implies that $\pi_k(X)$ determines a functor $X\to\Spaces$ which factors through $\Set$. At a point $x\colon\ast\to X$ the value of this functor is the set of homotopy classes of morphisms $S^k\to X$ which take the chosen base point of $S^k$ to the point $x\in X$. In the case of $X$ pointed, pulling back along the base point gives a functor $\ast\to \Spaces$ that factors through $\Set$. In other words, the choice of base point for $X$ causes $\pi_k(X)$ to simply be a set. It is straightforward to verify \cite[Remark 6.5.1.6]{htt} which asserts that $\pi_k(X)$, for $X\in\Spaces_\ast$, recovers the usual $k^{th}$ homotopy group of $X$.
	\end{rmk}

	\begin{rmk}
		Note that for $f\colon X\to Y$ in $\XX$, the ``homotopy sets of $f$'' are not morphisms of $\tau_{\leq 0}\XX$ but, rather, objects of $\tau_{\leq 0}\XX_{/X}$. These can be thought of as sets ``indexed by $X$.'' In the case that $\XX=\Spaces$, the homotopy set $\pi_k(f)$ at each $x\in X$ is precisely the $k^{th}$ homotopy set of the fiber of $f$ over $x$. This justifies the following definition.
	\end{rmk}

	\begin{defn}
		Let $f\colon X\to Y$ be a morphism of an $\infty$-topos $\XX$ and let $1\leq n$. Then we say that $f$ is $n$-connective if it is an effective epimorphism in $\XX$ and $\pi_i(f)\simeq\ast$ for all $0\leq i<n$. A morphism is called $0$-connective exactly when it is an effective epimorphism. By convention, every morphism of $\XX$ is considered to be $(-1)$-connective. We say that an object $X\in\XX$ is $n$-connective if the terminal morphism $X\to \ast$ is $n$-connective. Write $\XX^{\geq n}$ for the full subcategory of $n$-connective objects of $\XX$. For $n=1$ we will simply refer to objects of $\XX^{\geq 1}$ as \textit{connected} rather than $1$-connective. Notice that if an object $X\in\XX$ is $0$-connective then it must have $\pi_0(X)\neq\varnothing$, i.e.~it must be ``inhabited'' in a certain sense. This is, of course, not an issue when we restrict ourselves to pointed objects.
	\end{defn}

	\begin{defn}
		Let $\XX$ be an $\infty$-topos. Then we define the category of pointed objects of $\XX$, denoted $\XX_\ast$, to be the slice category $\XX^{\backslash\ast}$. We write $\XX_\ast^{\geq n}$ for the full subcategory of $\XX_\ast$ spanned by objects which are $n$-connective after application of the forgetful functor $\XX_\ast\to\XX$.
	\end{defn}

	\begin{defn}
		Let $\XX$ be an $\infty$-topos and write $K$ for the full subcategory of the augmented simplex category $\Delta_+^{\op}$ spanned by $[-1]$, $[0]$ and $[1]$, sometimes called $(\Delta^{\op}_+)_{\leq 1}$. Let $i\colon\Delta^1\to K$ be the inclusion of the subcategory $[-1]\leftarrow [0]$ and let $j\colon \Delta^1\to K$ be the inclusion of the subcategory $[0]\to[1]$ (as the degeneracy map). Then we define a pair of adjoint functors as follows:
		\begin{enumerate}
			\item Thinking of $\XX_\ast$ as a full subcategory of $\XX^{\Delta^1}$, we define the loops functor $\Omega_{\XX}\colon \XX_\ast\to\XX_\ast$ to be the composite of right Kan extension along $i$ followed by restriction along $j$.
			\item By the general yoga of Kan extensions this functor has a left adjoint $\Sigma_{\XX}\colon \XX_\ast\to\XX_\ast$ given by left Kan extension along $i$ followed by restriction along $j$.
		\end{enumerate}
		The ambient $\infty$-topos will always be clear from context so we will drop the subscript from the suspension and loops functors and simply write $\Sigma$ and $\Omega$.
	\end{defn}

	\begin{prop}\label{prop:loops suspension calculations}
		If $X\in\XX_\ast$ then $\Sigma X$ is equivalent to the pushout of the span $\ast\leftarrow X\to\ast$ and $\Omega X$ is equivalent to the pullback of the cospan $\ast\to X\leftarrow\ast$, both computed in $\XX_\ast$.
	\end{prop}

	\begin{proof}
		The case of $\Omega$ is well known (to those who know it well). The pointwise formula for right Kan extensions computes the value of $Ran_i(\ast\to X)$ at $[1]\in\Delta_+^{\op}$ to be the limit over the comma category $(\Delta_+^{\op})_{\leq 0}^{\backslash [1]}$. This last category is isomorphic to the cospan category $\{[0]\to[-1]\leftarrow[0]\}$. Because, by assumption, $[0]$ is sent to $\ast$, this gives the result.

		Now we prove the suspension case. Let $S=\{a\to b\leftarrow c\}$ denote the cospan category. In \cite[Lemma 1.2.4.17]{ha}, Lurie proves that the functor $\sigma\colon S\to\Delta_{\leq 1}$, given by $a,c\mapsto [0]$ and $b\mapsto[1]$, is ``right cofinal'' in the sense of the definition given on \cite[p. 13]{ha}. This is equivalent to saying that $\sigma^{\op}\colon S^{\op}\to\Delta_{\leq 1}^{\op}$ is cofinal in the sense of \cite[Section 4.1.1]{htt}. We warn the reader that this terminology is the opposite of the terminology used in \cite[Tag  02NQ]{kerodon}.

		Now consider an object $X\in\XX_\ast$ thought of as a functor $\{[0]\to[1]\}\to\XX_\ast$ with $[0]\mapsto\ast$ and $[1]\mapsto X$. Left Kan extending along the inclusion $\{[0]\to[1]\}\hookrightarrow \Delta^{\op}_{\leq 1}$ necessarily gives a diagram in which both of the induced morphisms $X\to\ast$ are the trivial ones (there is nothing else they can be).  Because Kan extensions paste, the left Kan extension of the diagram $\ast\to X$ along the inclusion into $(\Delta_+^{\op})_{\leq 1}$ is equivalent to first left Kan extending to $\Delta_{\leq 1}^{\op}$ and then left Kan extending again along the inclusion $\Delta^{\op}_{\leq 1}\hookrightarrow (\Delta_+^{\op})_{\leq 1}$. But $(\Delta_+^{\op})_{\leq 1}$ is precisely the cocone $(\Delta_{\leq 1}^{\op})^{\triangleright}$. Therefore the desired left Kan extension can be computed by simply taking the colimit of the previously obtained diagram $\Delta_{\leq 1}^{\op}\to\XX_\ast$. We now use that $S^{\op}\to\Delta_{\leq 1}^{\op}$ is cofinal, along with \cite[Proposition 4.1.1.8]{htt}, to deduce that this colimit is equivalent to the desired pushout.
	\end{proof}

	We will repeatedly use the following result of \cite{htt} so we reproduce it here.

	\begin{lem}[{\cite[Proposition 6.5.1.20]{htt}}]\label{lem:section is -1 connected}
		Let $\XX$ be an $\infty$-topos and assume $n\geq 0$. Let $f\colon X\to Y$ be a morphism of $\XX$ which admits a section $s\colon Y\to X$:
		\[Y\xrightarrow{s} X\xrightarrow{f} Y\]
		Then $f$ is $n$-connective if and only if $s$ is $(n-1)$-connective.
	\end{lem}

	\begin{lem}
		If~$\XX$ is an $\infty$-topos and $X\in\XX_\ast^{\geq n}$ then $\Sigma X\in\XX_\ast^{\geq n+1}$. In particular, $\Sigma X$ is connective for every $X\in\XX_\ast$.
	\end{lem}

	\begin{proof}
		Suppose that $X$ is $n$-connective, i.e.~the terminal morphism $X\to\ast$ is $n$-connective. By \cite[Corollary 6.5.1.17]{htt} we have that $n$-connective morphisms are stable under pushout, so the right vertical morphism in the following commutative diagram is also $n$-connective.
		\[\begin{tikzcd}
			X & \ast \\
			\ast & {\Sigma X}
			\arrow[from=1-1, to=1-2]
			\arrow[from=1-1, to=2-1]
			\arrow[from=1-2, to=2-2]
			\arrow[from=2-1, to=2-2]
			\arrow["\lrcorner"{anchor=center, pos=0.125, rotate=180}, draw=none, from=2-2, to=1-1]
		\end{tikzcd}\]
		That morphism fits into a retract $\ast\to\Sigma X\to \ast$ and therefore, by Lemma \ref{lem:section is -1 connected}, the morphism $\Sigma X\to \ast$ is $(n+1)$-connective.
	\end{proof}

	\begin{lem}\label{lem:omegasigma adjunction}
		Let $\Omega^{\geq 1}$ denote the restriction of $\Omega$ to $\XX_\ast^{\geq 1}$. The reduced suspension functor $\Sigma\colon \XX_\ast\to\XX_\ast$ is left adjoint to $\Omega$ and corestricts to a left adjoint of~~$\Omega^{\geq 1}$.
	\end{lem}

	\begin{proof}
		The fact that $\Sigma$ is left adjoint to $\Omega$ follows from definitions. Because $\Sigma$ factors through the full subcategory $\XX^{\geq 1}_\ast$, the unit morphism $id_{\XX_\ast}\Rightarrow \Omega\Sigma$ defining the $\Omega\vdash\Sigma$ adjunction restricts to a unit morphism $id_{\XX^{\geq 1}_{\ast}}\Rightarrow \Omega^{\geq 1}\Sigma$.
	\end{proof}

	\section{Group Objects in $\infty$-topoi}

	The concept of ``group object in an $\infty$-category'' is defined in \cite[Example 7.2.2.2]{htt}. These are the objects that we will study for the remainder of the paper. The goal of this section is to carefully lay out the necessary definitions and to prove that things like delooping and taking quotients behave ``as expected.''

	It is important to note that in \cite{ha} Lurie introduces ``grouplike associative monoids'' and ``grouplike associative algebras'' as well. All of these deserve to be called ``group objects,'' and Lurie shows them to be equivalent for a Cartesian monoidal $\infty$-category, but we will primarily work with the first notion in this paper. Nonetheless, in Section \ref{sec:base change} we review the equivalences between each of these models and prove that they are compatible with quotients.

	\subsection{Simplicial Models for $\infty$-Groups and Actions}

	\begin{defn}
		Let $\XX$ be an $\infty$-topos. We say that a simplicial object $Z\colon \Delta^\op\to \XX$ is a \textit{category object} if the canonical Segal maps $Z_n\to Z_1\times_{Z_0}\cdots\times_{Z_0}Z_1$ are equivalences for all $n$ (cf.~\cite[Remark 2.4.4]{gepnerhaugsengenriched}). We say that a category object $Z$ is a \textit{monoid object} if $Z_0$ is terminal in $\XX$ (cf.~\cite[Definition 4.1.2.5]{ha}). We will write $\Catnoinf(\XX)$ for the full subcategory of $\XX^{\Delta^\op}$ spanned by category objects and $\Mon(\XX)$ for the full subcategory of $\XX^{\Delta^\op}$ spanned by monoid objects.
	\end{defn}

	\begin{defn}\label{defn:groupoid objects}
		Let $C\colon\Delta^\op\to\XX$ be a category object of $\XX$. We say that $C$ is a \textit{groupoid object} of $\XX$ if it satisfies the following condition (cf.~\cite[Proposition 6.1.2.6, Definition 6.1.2.7]{htt}):
		\begin{itemize}
			\item Let $k,m,n\geq 0$ with $m\leq n$ and $k\leq n$. Suppose additionally that there is a pushout diagram in $\Delta$
			\[\begin{tikzcd}
				{[0]} & {[m]} \\
				{[k]} & {[n]}
				\arrow[from=1-1, to=1-2]
				\arrow[from=1-1, to=2-1]
				\arrow[from=1-2, to=2-2]
				\arrow[from=2-1, to=2-2]
				\arrow["\lrcorner"{anchor=center, pos=0.125, rotate=180}, draw=none, from=2-2, to=1-1]
			\end{tikzcd}\] i.e.~that the union of the images of $[m]$ and $[k]$ is all of $[n]$ and that their intersection is only one element. Then the image of this diagram under $Z$:
			\[
			\begin{tikzcd}
				C_n\ar[r]\ar[d] & C_m\ar[d]\\
				C_k\ar[r] & C_0
			\end{tikzcd}
			\] is a pullback diagram in $\XX$. We write $\Grpd(\XX)$ for the full subcategory of $\XX^{\Delta^\op}$ spanned by groupoid objects.
		\end{itemize}
	\end{defn}

	\begin{rmk}
		Note that asking for $C$ to be a category object in Definition \ref{defn:groupoid objects} is superfluous. Indeed, the condition given there implies that $C_n\simeq C_{n-1}\times_{C_0}C_1$ for every $n$ from which the ``Segal condition'' inductively follows.
	\end{rmk}

	\begin{rmk}
		The reader is encouraged to check that a groupoid object in $\Set$, in terms of Definition \ref{defn:groupoid objects}, is precisely a (discrete) category in which every morphism is invertible. This follows from the fact that we have required the Segal condition for all partitions of $[n]$ which overlap at one element, as opposed to \textit{ordered} partitions alone.
	\end{rmk}

	\begin{defn}\label{defn:group objects}
		Given a simplicial object $Z\colon \Delta^\op\to \XX$, we say that $Z$ is a \textit{group object} if it is both a groupoid object and a monoid object (cf.~\cite[Definition 7.2.2.1]{htt}). We write $\Grp(\XX)$ for the full subcategory of $\XX^{\Delta^\op}$ spanned by group objects.
	\end{defn}

	\begin{rmk}
		The above definitions are homotopically coherent ways of saying ``a monoid is a category with one object,'' ``a groupoid is a category in which every morphism is invertible,'' and ``a group is a groupoid with one object.''
	\end{rmk}

	\begin{defn}\label{defn:group forgetful functor}
		Define the (pointed) forgetful functor $s_0^\ast\colon\Grp(\XX)\to\XX_\ast$ to be restriction along the zeroth degeneracy $\{[0]\xrightarrow{s_0}[1]\}\hookrightarrow \Delta^{\op}$. Define the (unpointed) forgetful functor $U_{\Grp}\colon \Grp(\XX)\to\XX$ as the further restriction along $[1]\hookrightarrow\Delta^{\op}$. If $G\in\Grp(\XX)$, we will sometimes write $G_1$ for $U_{\Grp}(G)$.
	\end{defn}

	\begin{defn}
		Let $\XX$ be an $\infty$-topos. Following \cite[Definition 4.2.2.2]{ha} we define a \textit{right action object} of $\XX$ to be a functor $A\colon\Delta^\op\times\Delta^1\to \XX$ satisfying the following two properties:
		\begin{enumerate}
			\item The restriction of $A$ along $\Delta^\op\times\{1\}\hookrightarrow \Delta^\op\times\Delta^1$ is a monoid object of $\XX$.
			\item The iterated face map $A([n],0)\to A(\{0\},0)\simeq A([0],0)$ and the map $A([n],0)\to A([n],1)$ exhibit $A([n],0)$ as the product $A([0],0)\times A([n],1)$.
		\end{enumerate}

		We will write $\RMon(\XX)$ for the full subcategory of $\Fun(\Delta^{\op}\times \Delta^1,\XX)$ spanned by right action objects. If $A$ is a right action object with $A(\Delta^{\op},1)=M$ and $A([0],0)=X$, we will sometimes write $\Act(X,M)$ instead of $A$ to indicate that $A$ is encoding an action of a monoid $M$ on an object $X$.
	\end{defn}

	\begin{rmk}
		A right action of a monoid $M\simeq A(\Delta^{op},1)$ on $X\simeq A([0],0)$ can be depicted schematically as
		\[\begin{tikzcd}
			\vdots & \vdots \\
			{X\times M_1\times M_1} & {M_1\times M_1} \\
			{X\times M_1} & M_1 \\
			X & \ast
			\arrow[from=4-1, to=4-2]
			\arrow[shift right, from=3-1, to=4-1]
			\arrow[shift left, from=3-2, to=4-2]
			\arrow[shift left, from=3-1, to=4-1]
			\arrow[shift right, from=3-2, to=4-2]
			\arrow[from=2-1, to=3-1]
			\arrow[from=2-2, to=3-2]
			\arrow[shift left=2, from=2-1, to=3-1]
			\arrow[shift right=2, from=2-1, to=3-1]
			\arrow[shift left=2, from=2-2, to=3-2]
			\arrow[shift right=2, from=2-2, to=3-2]
			\arrow[from=3-1, to=3-2]
			\arrow[from=2-1, to=2-2]
			\arrow[shift right, from=1-1, to=2-1]
			\arrow[shift left, from=1-1, to=2-1]
			\arrow[shift right=3, from=1-1, to=2-1]
			\arrow[shift left=3, from=1-1, to=2-1]
			\arrow[shift right, from=1-2, to=2-2]
			\arrow[shift left, from=1-2, to=2-2]
			\arrow[shift left=3, from=1-2, to=2-2]
			\arrow[shift right=3, from=1-2, to=2-2]
		\end{tikzcd}\]
	\end{rmk}

	\begin{defn}
		Restricting along the inclusion $\Delta^\op\times\{1\}\hookrightarrow \Delta^{\op}\times\Delta^1$ induces a forgetful functor $\RMon(\XX)\to\Mon(\XX)$. We write $\RMon_M(\XX)$ for the pullback of this functor along $\{M\}\hookrightarrow\Mon(\XX)$:
		\[\begin{tikzcd}
			{\RMon_M(\XX)} & {\RMon(\XX)} \\
			{\{M\}} & {\Mon(\XX)}
			\arrow[from=1-1, to=1-2]
			\arrow[from=1-1, to=2-1]
			\arrow["\lrcorner"{anchor=center, pos=0.125}, draw=none, from=1-1, to=2-2]
			\arrow[from=1-2, to=2-2]
			\arrow[from=2-1, to=2-2]
		\end{tikzcd}\]
	\end{defn}

	\begin{defn}\label{defn:barconstruction}
		Let $M\in\Mon(\XX)$ and let $\Act(X,M)\in \RMon_M(\XX)$. Then we define the \textit{bar construction for the action of $M$ on $X$} to be the simplicial object $\Delta^{op}\to \XX$ obtained from $X$ by precomposing with the inclusion $\Delta^{\op}\times\{0\}\hookrightarrow \Delta^{\op}\times\Delta^1$. When the action object $\Act(X,M)$ is specified by context we write $\Barc(X,M)$ for this simplicial object.
	\end{defn}

	\begin{rmk}
		Let $M\in\Mon(\XX)$ be a monoid object, hence an object of the functor category $\XX^{\Delta^{\op}}$. By considering the identity morphism of $M$ in $(\XX^{\Delta^\op})^{\Delta^1}$ as an object of $\XX^{\Delta^\op\times\Delta^1}$, we obtain a right action object displaying an action of $M$ on $\ast$. Note that $\Barc(\ast,M)\simeq M$. In other words, this method of thinking about monoids in an $\infty$-category has the bar construction built in--a monoid \textit{is} its bar construction.
	\end{rmk}

	\subsection{Augmented Simplicial Objects and Quotients}

	\begin{defn}\label{defn:the simplicial functors}
		Let $\XX$ be an $\infty$-topos with $M\in\Mon(\XX)$ and $\Act(X,M)\in\RMon_M(\XX)$. We define several functors relating $\RMon(\XX)$, $\Mon(\XX)$ and $\XX$, as well as extensions of these categories to ``augmented'' versions. These are essentially all ``the same picture'' described in several different ways, and we encourage the reader to understand them as such. Write $i_+\colon \Delta^{\op}\to\Delta_+^{\op}$ for the inclusion of the unaugmented simplex category into the augmented simplex category and $i_+\times\Delta^1\colon \Delta^{\op}\times\Delta^1\hookrightarrow \Delta_+^{\op}\times\Delta^1$ for its obvious extension to $\Delta^{\op}\times\Delta^1$.
		\begin{enumerate}
			\item Write $\BB\colon \Mon(\XX)\to \XX$ for the geometric realization functor. Given $M\in\Mon(\XX)$ we will write $\BB M$ for its image in $\XX$. This should be thought of as $\ast/M$, i.e.~the quotient of the point by the trivial action of the monoid $M$.

			\item Write $\BB_\bullet^+\colon \Mon(\XX)\to \XX^{\Delta_+^{\op}}$ for the functor which left Kan extends $M$ along $i_+$. Note that $\BB_\bullet^+M$ recovers $\BB M$ when restricted to $[-1]\in \Delta_+^{\op}$. Therefore, $\BB_\bullet^+(M)$ displays both the action groupoid of $M$ on the point and its resultant quotient $\BB M$.

			\item Write $\BB^{+}\colon \Mon(\XX)\to \XX_\ast$ for the functor which restricts $\BB_\bullet^+$ to the terminal morphism $\{[0]\to[-1]\}\in\Delta^{\op}_+$. In other words, $\BB^+$ is the functor that takes the quotient $\ast/M$ and remembers the quotient map $\ast\to\ast/M$.

			\item Write $-/-\colon \RMon(\XX)\to \XX$ for the functor that takes a right action object $\Act(X,M)$ to the geometric realization of $\Barc(X,M)$. Given $\Act(X,M)$ we will write $X/M$ for its image under this functor. Note that this is simply geometric realization applied to the restriction of $\Act(X,M)$ to $\Delta^{\op}\times\{0\}$.

			\item Write $\Actplus\colon \RMon(\XX)\to\XX^{\Delta_+^{\op}\times\Delta^1}$ for the left Kan extension of a right action object $\Act(X,M)$ along $i_+\times\Delta^1$. Write $\Actplus(X,M)$ for the image of $\Act(X,M)$ under this functor. This is the analogue of $\BB_\bullet^+$ for an entire action object instead of just a monoid. The functor $\Actplus(X,M)$ displays ``all possible information'' associated to the right action of $M$ on $X$ by taking the quotient in both coordinates simultaneously.

			\item Write $\Barcplus(-,-)\colon \RMon(\XX)\to \XX^{\Delta^{\op}_{+}}$ for the restriction of $\Actplus$ along $\Delta_+^{\op}\times\{0\}\hookrightarrow \Delta^{\op}_+\times\Delta^1$. We will write $\Barcplus(X,M)$ for the image of $\Act(X,M)\in\RMon(\XX)$ under this functor.  This is displaying the action groupoid of $M$ on $X$ as well as the quotient $X/M$.
		\end{enumerate}
		We will write $\Mon^+(\XX)$ and $\RMon^+(\XX)$ for the essential images of $\BB^+_\bullet$ and $\mathrm{Act}^+_\bullet$ respectively, and $\RMon^+_M(\XX)$ for obvious restriction of $\RMon^+(\XX)$ to a fixed monoid.
	\end{defn}

	\begin{rmk}The values of $\Actplus(X,M)$ at $([-1],0)$ and $([-1],1)$ are respectively the geometric realizations of $\Barc(X,M)$ and $\Barc(\ast,M)\simeq M$. This can be visualized as the following diagram:
		\[\begin{tikzcd}
			\vdots & \vdots \\
			{X\times M\times M} & {M\times M} \\
			{X\times M} & M \\
			X & \ast \\
			{X/M} & {\ast/M}
			\arrow[shift right, from=1-1, to=2-1]
			\arrow[shift left, from=1-1, to=2-1]
			\arrow[shift right=3, from=1-1, to=2-1]
			\arrow[shift left=3, from=1-1, to=2-1]
			\arrow[shift right, from=1-2, to=2-2]
			\arrow[shift left, from=1-2, to=2-2]
			\arrow[shift left=3, from=1-2, to=2-2]
			\arrow[shift right=3, from=1-2, to=2-2]
			\arrow[from=2-1, to=2-2]
			\arrow[from=2-1, to=3-1]
			\arrow[shift left=2, from=2-1, to=3-1]
			\arrow[shift right=2, from=2-1, to=3-1]
			\arrow[from=2-2, to=3-2]
			\arrow[shift left=2, from=2-2, to=3-2]
			\arrow[shift right=2, from=2-2, to=3-2]
			\arrow[from=3-1, to=3-2]
			\arrow[shift right, from=3-1, to=4-1]
			\arrow[shift left, from=3-1, to=4-1]
			\arrow[shift left, from=3-2, to=4-2]
			\arrow[shift right, from=3-2, to=4-2]
			\arrow[from=4-1, to=4-2]
			\arrow[from=4-1, to=5-1]
			\arrow[from=4-2, to=5-2]
			\arrow[from=5-1, to=5-2]
		\end{tikzcd}\]
		In the case that $M$ is a group object, we will see that the bottom square will be a pullback, recovering the idea of a principal $G$-bundle as described, for instance, in \cite{nikolausSchreiberStevenson-principal1}.
	\end{rmk}

	\subsection{Action Groupoids}\label{sec: action groupoids}

	We now specialize to the case of group actions, rather than monoid actions. This section is primarily an exposition of \cite[Remark 5.2.6.28]{ha}. First we need the following definitions.

	\begin{defn}[{\cite[Definition 6.1.3.1]{htt}}]
		Let $F,G\colon \CC\to\DD$ be functors of $\infty$-categories. We say that a natural transformation $\eta\colon F\Rightarrow G$ is \textit{Cartesian} if, for every morphism $\phi\colon c\to d$ in $\CC$, the following is a pullback diagram in $\DD$:
		\[\begin{tikzcd}
			{F(c)} & {G(c)} \\
			{F(d)} & {G(d)}
			\arrow["{\eta_c}", from=1-1, to=1-2]
			\arrow["{F(\phi)}"', from=1-1, to=2-1]
			\arrow["\lrcorner"{anchor=center, pos=0.125}, draw=none, from=1-1, to=2-2]
			\arrow["{G(\phi)}", from=1-2, to=2-2]
			\arrow["{\eta_d}"', from=2-1, to=2-2]
		\end{tikzcd}\]
	\end{defn}

	\begin{defn} Write $\RMon_{gp}(\XX)$ for the pullback of the restriction functor $\RMon(\XX)\to\Mon(\XX)$ along the inclusion $\Grp(\XX)\hookrightarrow\Mon(\XX)$.
		All of the categories and functors of Definition \ref{defn:the simplicial functors} can be restricted along the inclusions $\RMon_{gp}(\XX)\hookrightarrow \RMon(\XX)$ and $\Grp(\XX)\hookrightarrow\Mon(\XX)$, and they continue to have the same interpretations as Kan extensions and restrictions. Write $\Grp^+(\XX)$ and $\RMon_{gp}^+(\XX)$ for the essential images of these restrictions.
	\end{defn}

	\begin{lem}[{\cite[Remark 5.2.6.28]{ha}}]\label{lem:restrict to K}
		Let $\KK$ denote the full subcategory of $\Delta_+^{\op}\times\Delta^1$ spanned by the objects $([-1],0)$, $([-1], 1)$ and $([0],1)$. Let $\mathcal{E}$ denote the full subcategory of~ $\XX^{\Delta^{op}_+\times\Delta^1}$ spanned by functors $F$ with the following properties:
		\begin{enumerate}
			\item The functor $F$ is a right Kan extension of its restriction to $\mathcal{K}\subset \Delta_+^{\op}\times\Delta^1$.
			\item The object $F([0],1)$ is terminal.
			\item The object $F([-1],1)$ is connected (equivalently, $F([0],1)\to F([-1],1)$ is an effective epimorphism).
		\end{enumerate}
		Then restriction along the inclusion of~~$\mathcal{K}$ induces a fully faithful functor $\mathcal{E}\to \XX^{\Delta^1}\times_{\XX^{\{1\}}}\XX^{\Delta^1}$ whose essential image is $\XX^{\Delta^1}\times_{\XX^{\{1\}}}\XX^{\geq 1}_\ast$.
	\end{lem}

	\begin{proof}
		Note that by (the opposite of) \cite[Proposition 4.3.2.15]{htt}, functors out of $\Delta^{\op}_+\times\Delta^1$ satisfying only the first condition above are equivalent to functors out of $\mathcal{K}$. In the statement of that proposition, we should set $\CC=\Delta_+^{\op}\times\Delta^1$, $\DD=\XX$, $\CC^0=\mathcal{K}$ and $\DD'=\ast$. The proposition may be interpreted as saying that functors out of $\mathcal{K}$ which admit right Kan extensions to all of $\Delta_+^{\op}\times\Delta^1$ are equivalent to functors which \textit{are} right Kan extensions.  The remainder of the conditions simply identify a full subcategory whose essential image is the given full subcategory of $\XX^{\Delta^1}\times_{\XX^{\{1\}}}\XX^{\Delta^1}$.
	\end{proof}

	\begin{rmk}
		The above lemma shows that functors in $\mathcal{E}$ can be thought of as cospans $X\to Y\leftarrow \ast$ in $\XX$ with $Y$ connected.
	\end{rmk}

	\begin{lem}[{\cite[Remark 5.2.6.28]{ha}}]\label{lem:koszul duality diagram for actions}
		The $\infty$-category $\mathcal{E}$ of Lemma \ref{lem:restrict to K} is equal to $\RMon_{gp}^+(\XX)$ (as full subcategories of $\XX^{\Delta^{\op}_+\times\Delta^1}$) . Moreover, it fits into the following commutative diagram in which the horizontal morphisms are equivalences:
		\[\begin{tikzcd}
			{\RMon_{gp}(\XX)} && {\RMon^+_{gp}(\XX)=\mathcal{E}} & {\XX^{\Delta^1}\times_{\XX^{\{1\}}}\XX^{\geq 1}_\ast} \\
			{\Grp(\XX)} && {\Grp^+(\XX)} & {\XX_\ast^{\geq 1}}
			\arrow["{\mathrm{Act}^+_\bullet}", from=1-1, to=1-3]
			\arrow[from=1-1, to=2-1]
			\arrow["{res_{\mathcal{K}}}", from=1-3, to=1-4]
			\arrow[from=1-3, to=2-3]
			\arrow[from=1-4, to=2-4]
			\arrow["{\BB^+_\bullet}"', from=2-1, to=2-3]
			\arrow["{res_{[0]\to[-1]}}"', from=2-3, to=2-4]
		\end{tikzcd}\]
	\end{lem}

	\begin{proof}
		The fact that $\RMon_{gp}^+(\XX)= \mathcal{E}$ follows from the penultimate sentence of \cite[Remark 5.2.6.28]{ha}, in which $\mathcal{E}$ is called $\mathcal{Y}$, which is earlier (in the same remark) referred to as $\mathcal{D}$. It is implied there that these two categories are full subcategories of $\XX^{\Delta^{\op}_+\times\Delta^1}$ defined by equivalent conditions (where we remind the reader that $\RMon_{gp}^+$ is defined to be the essential image of $\Actplus$). In other words, $\mathcal{E}$ is precisely the full subcategory of $\XX^{\Delta_+^{\op}\times\Delta^1}$ spanned by functors which are left Kan extensions of their restrictions to $\Delta^{\op}\times\Delta^1$ and whose restriction to $\Delta^{\op}\times\{1\}$ is a group object, which are the defining properties of $\RMon_{gp}^+(\XX)$.

		Given that $\RMon^+_{gp}(\XX)$ is defined to be the image of $\Actplus$, it suffices to show that $\Actplus$ is fully faithful. Recall that, in general, left Kan extending against a fully faithful functor  defines a fully faithful functor between functor categories. Hence the functor $\mathrm{Act}_\bullet^+$ is fully faithful because it is Kan extension along the fully faithful functor $i_+\colon \Delta^{\op}\times\Delta^1\hookrightarrow \Delta^{\op}_+\times\Delta^1$. It is straightforward to check, from definitions, that the restriction of $\mathrm{Act}_\bullet^+$ to $\Grp(\XX)$, i.e.~the functor $\BB_\bullet^+$, is an equivalence and that both squares above commute. The fact that the two right horizontal morphisms are equivalences follows from Lemma \ref{lem:restrict to K}.
	\end{proof}

	\begin{cor}\label{cor:G modules are BG comodules}
		By taking fibers over a fixed $G\in\Grp(\XX)$ in the diagram of Lemma \ref{lem:koszul duality diagram for actions}, we obtain equivalences \[\RMon_G(\XX)\simeq \RMon_G^+(\XX)\simeq \XX_{/\BB G}\]
	\end{cor}

	Corollary \ref{cor:G modules are BG comodules} allows us to produce right action objects without having to explicitly construct simplicial objects, as in the following examples. We are implicitly using that the fiber of $\XX^{\Delta^1}\times_{\XX^\{1\}}\XX_\ast^{\geq 1}$ over a fixed pointed connected space $\BB G$ is equivalent to $\XX_{/\BB G}$ via the functor that projects a diagram $Y\to \BB G\leftarrow \ast$ to $Y\to \BB G$.

	\begin{exam}\label{exam:translation action}
		If $G\in\Grp(\XX)$ then applying $\BB^+$ gives an object $\ast\to\BB G$ of $\XX_{/\BB G}$ whose fiber is $G_1$. Passing to $\RMon_G(\XX)$ gives a right action object of $G$ on $G_1$ that we refer to as the \textit{translation action} of $G$ on itself.
	\end{exam}

	\begin{exam}
		The identity morphism $\BB G\to \BB G$ has trivial fiber and therefore the associated right action object is the \textit{trivial action} of $G$ on the terminal object of $\XX$. Note that because $\BB G$ is always pointed there is a morphism $X\to\ast\to\BB G$ for any $X\in\XX$. The associated right action object is the trivial action of $G$ on $X$.
	\end{exam}

	\begin{exam}\label{exam:conjugation action}
		Choosing a base point $p\colon\ast\to S^1$, we obtain an evaluation morphism $\BB G^{S^1}\to\BB G$ whose fiber is equivalent to $\Omega\BB G\simeq G_1$. We refer to the resulting right action object as the \textit{conjugation action} of $G$ on itself.
	\end{exam}

	\begin{rmk}
		Because $\Barcplus(-,-)\colon \RMon(\XX)\to\RMon^+(\XX)$ is defined as a left Kan extension, it has a restriction as its right adjoint. Similarly, $res_{\mathcal{K}}$ has a right Kan extension as its right adjoint. This gives explicit formulae for the inverse equivalences of the horizontal arrows in Lemma \ref{lem:koszul duality diagram for actions}.
	\end{rmk}

	\begin{defn}\label{defn:act corner and cech corner}
		We will write $\mathrm{Act}^\lrcorner\colon \RMon_{gp}(\XX)\to\XX^{\Delta^1}\times_{\XX^{\{1\}}}\XX_\ast^{\geq 1}$ for the equivalence obtained as the composite of $res_\mathcal{K}$ with $\Actplus$ (the notation is chosen to suggest the bottom right corner of the $\Delta^1\times\Delta^{\op}$ shaped diagram). For fixed $G\in\Grp(\XX)$, write $\mathrm{Act}^\lrcorner_G$ for the restriction of $\mathrm{Act}^\lrcorner$ to the fiber of $\RMon_{gp}(\XX)$ over $G$. Write $\cech^{\lrcorner}$ and $\cech^{\lrcorner}_{G}$ for inverses of $\mathrm{Act}^{\lrcorner}$ and $\mathrm{Act}^{\lrcorner}_G$.

		Note that the composite $res_{[0]\to[-1]}\circ\BB_\bullet^+$ is precisely the functor $\BB^+$ of Definition \ref{defn:the simplicial functors} restricted to $\Grp(\XX)$. Write $\cech$ for an inverse of $res_{[0]\to[-1]}\circ\BB_\bullet^+$. Going forward we will use $\BB^+$ exclusively as a functor with domain $\Grp(\XX)$ (rather than $\Mon(\XX)$).
	\end{defn}

	\begin{rmk}\label{rmk:pullback square at bottom of action}
		The arguments of \cite[Remark 5.2.6.28]{ha} imply that every object of the $\infty$-category \[\RMon_{gp}^+(\XX)\subseteq \Fun(\Delta^{\op}_+\times\Delta^1,\XX)\simeq \Fun(\Delta^1,\Fun(\Delta^{op}_+,\XX))\] corresponds to a Cartesian natural transformation. In other words, if we think of $\Actplus(X,G)$ as a natural transformation $\Delta^1\to \XX^{\Delta_+^{\op}}$ then it has the property that every naturality square is a pullback. In particular there is always a fiber sequence $X\to X/G\to \BB G$ coming from $\Actplus(X,G)$ restricted to $(\Delta_+^{\leq 0})^{\op}\times\Delta^1$.
	\end{rmk}

	\begin{cor}\label{cor:omega gives loop space action}
		Given $f\colon X\to Y\in\XX_{/Y}$, for $Y$ a pointed and connected object of $\XX$, the action object $\cech^{\lrcorner}(f)$ displays a group structure on $\Omega Y$ and an action of it on $\fib(f)$.  Moreover, $Y\simeq \BB\Omega Y$.
	\end{cor}

	\begin{proof}
		Remark \ref{rmk:pullback square at bottom of action} implies that $\cech^{\lrcorner}(f)$ describes an action of a group object on $\fib(f)$. Note that $\cech^{\lrcorner}$ factors through $\RMon_{gp}^+(\XX)$ by first taking the right Kan extension along the inclusion $\KK\hookrightarrow \Delta_+^{\op}\times\Delta^1$. It then takes the resulting augmented object to its restriction to $\Delta^{\op}\times\{1\}$. From \cite[Remark 5.2.6.28]{ha} (specifically item $i_0$) we have that $\cech^{\lrcorner}(f)(\Delta^{op},\{1\})$ is the right Kan extension of $\ast\to Y$ along the inclusion $\{[0]\to[-1]\}\hookrightarrow \Delta_+^\op$ followed by its restriction to $\Delta^{\op}$. In other words, $\cech^{\lrcorner}(f)(\Delta^{\op},\{1\})$ is the \v Cech nerve of the pointing $\ast\to Y$. It follows from Proposition \ref{prop:loops suspension calculations} (and the fact that Kan extensions paste along full subcategory inclusions) that $\cech^{\lrcorner}(f)([1],\{1\})\simeq\Omega Y$. The last statement follows from the bottom row of the diagram of Lemma \ref{lem:koszul duality diagram for actions} being an equivalence.
	\end{proof}

	\begin{lem}\label{lem:action groupoid is a groupoid}
		Let $G\in\Grp(\XX)$ and $\Act(X,G)\in\RMon_G(\XX)$. Then $\Barc(X,G)$ is a groupoid object of $\XX$ in the sense of \cite[Definition 6.1.2.7]{htt}.
	\end{lem}

	\begin{proof}
		For clarity in this proof we simply write $G^n$ instead of $G_1^n$ or, equivalently, $G_n$. Suppose we have a commutative square in which each morphism is a face map of $\Barc(X,G)$: 
		\[\begin{tikzcd}
			{G^n\times X} & {G^m\times X} \\
			{G^r\times X} & X
			\arrow["{d_j}", from=1-1, to=1-2]
			\arrow["{d_i}"', from=1-1, to=2-1]
			\arrow["{d_k}", from=1-2, to=2-2]
			\arrow["{d_l}"', from=2-1, to=2-2]
		\end{tikzcd}\]
		This diagram projects onto the corresponding diagram in $G$
		\[\begin{tikzcd}
			{G^n} & {G^m} \\
			{G^r} & \ast
			\arrow["{d_j}", from=1-1, to=1-2]
			\arrow["{d_i}"', from=1-1, to=2-1]
			\arrow["\lrcorner"{anchor=center, pos=0.125}, draw=none, from=1-1, to=2-2]
			\arrow["{d_k}", from=1-2, to=2-2]
			\arrow["{d_l}"', from=2-1, to=2-2]
		\end{tikzcd}\]
		which is a pullback, since $G$ is a group(oid) object as in Definition \ref{defn:groupoid objects}. The first diagram forms the ``top'' and the second diagram forms the ``bottom'' of the following commutative cube:
		\[\begin{tikzcd}
			{G^n\times X} && {G^m\times X} \\
			& {G^r\times X} && X \\
			{G^n} && {G^m} \\
			& {G^r} && \ast
			\arrow[from=1-1, to=1-3]
			\arrow[from=1-1, to=2-2]
			\arrow[from=1-1, to=3-1]
			\arrow[from=1-3, to=2-4]
			\arrow["{~}"{description}, from=1-3, to=3-3]
			\arrow["{~}"{description}, from=3-1, to=3-3]
			\arrow[from=2-2, to=2-4]
			\arrow[from=2-2, to=4-2]
			\arrow[from=2-4, to=4-4]
			\arrow[from=3-1, to=4-2]
			\arrow[from=3-3, to=4-4]
			\arrow[from=4-2, to=4-4]
		\end{tikzcd}\]
		Because $\Act(X,G)$ is a Cartesian natural transformation, as in Remark \ref{rmk:pullback square at bottom of action}, and $G$ is a group object, every face in the above cube is a pullback by assumption except for the top one. However this implies that the top face must also be a pullback.
	\end{proof}

	\begin{cor}
		Let $\Act(X,G)$ be a right action object in $\RMon_{gp}(\XX)$. Then the restriction of $\Act(X,G)$ to $\Delta^{\op}\times\{0\}$ is a groupoid resolution of $X\to X/G$ in the sense of \cite[p. 556]{htt}.
	\end{cor}

	\begin{proof}
		By Lemma \ref{lem:action groupoid is a groupoid}, we have that $\Barc(X,G)$ is a groupoid object of $\XX$. Moreover, by Lemma \ref{lem:koszul duality diagram for actions} we have that $X/G$ is the colimit of $\Barc(X,G)$.
	\end{proof}

	\begin{cor}
		The morphism $X\to X/G$ in $\Act^+(X,G)$ is an effective epimorphism.
	\end{cor}

	\begin{proof}
		By \cite[Theorem 6.1.0.6]{htt} we know that every groupoid object of $\XX$ is effective, as in \cite[Definition 6.1.2.14]{htt}, which means every groupoid object is equivalent to the \v Cech nerve of the augmentation map to its colimit. Hence $\BB_\bullet(X,G)$ is equivalent to the \v Cech nerve of $X\to X/G$. Therefore $X\to X/G$ satisfies the conditions of Definition \ref{defn: cech nerve and effective epis}.
	\end{proof}

	\subsection{Free Group Objects}

	\begin{lem}\label{lem:free group is cech sigma}
		The pointed forgetful functor $s_0^\ast\colon \Grp(\XX)\to\XX_\ast$ of Definition \ref{defn:group forgetful functor} admits a left adjoint given by the composite $\cech\Sigma$.
	\end{lem}

	\begin{proof}
		By construction there is an equivalence of functors $s_0^\ast\cech\simeq\Omega^{\geq 1}\colon \XX^{\geq 1}_\ast\to\XX_\ast$. By Lemma \ref{lem:omegasigma adjunction} we have that $s_0^\ast\cech$ is right adjoint to $\Sigma\colon \XX_\ast\to\XX_\ast^{\geq 1}$. Because $\BB^+$ and $\cech$ are equivalences, they are both left adjoints and both right adjoints. Therefore we may form the composite adjunction:
		\[\begin{tikzcd}
			\XX_\ast &&& \XX_\ast^{\geq 1} && \Grp(\XX)
			\arrow[""{name=0, anchor=center, inner sep=0}, "\Sigma"', curve={height=12pt}, from=1-1, to=1-4]
			\arrow[""{name=1, anchor=center, inner sep=0}, "{\Omega^{\geq 1}}"', curve={height=12pt}, from=1-4, to=1-1]
			\arrow[""{name=2, anchor=center, inner sep=0}, "\cech"', curve={height=12pt}, from=1-4, to=1-6]
			\arrow[""{name=3, anchor=center, inner sep=0}, "\BB^+"', curve={height=12pt}, from=1-6, to=1-4]
			\arrow["\dashv"{anchor=center, rotate=90}, draw=none, from=0, to=1]
			\arrow["\dashv"{anchor=center, rotate=90}, draw=none, from=2, to=3]
		\end{tikzcd}\]
		However, $\Omega^{\geq 1}\circ \BB^+\simeq s_0^\ast\circ \cech\circ\BB^+\simeq s_0^\ast$. Therefore $\cech\Sigma$ is left adjoint to the forgetful functor $\Grp(\XX)\to \XX_\ast$.
	\end{proof}

	\begin{lem}\label{lem:B shifts homotopy}
		For any $G\in\Grp(\XX)$ there is an isomorphism $\pi_i(s_0^\ast G)\cong \pi_{i+1}(\BB^+ G)$.
	\end{lem}

	\begin{proof}
		From the preceding lemmas, we have a string of adjunction equivalences:
		\[\XX_\ast(S^i,s_0^\ast G)\simeq \Grp(\XX)(\cech\Sigma S^i, G)\simeq \XX_\ast^{\geq 1}( S^{i+1},\BB^+ G)\simeq\XX_\ast(S^{i+1},\BB^+ G)\]
		Applying the truncation functor gives the result.
	\end{proof}

	\begin{defn}\label{defn:cechplus}
		Write $\check{\Omega}\colon \XX_\ast\to\XX^{\Delta^{\op}}$ for the functor obtained by right Kan extending along the inclusion $\{[0]\to[-1]\}\hookrightarrow\Delta_+^{\op}$ and then restricting along $\Delta^{\op}\hookrightarrow\Delta_+^{\op}$.
	\end{defn}

	\begin{lem}
		The functor $\check{\Omega}$ factors through $\Grp(\XX)$.
	\end{lem}

	\begin{proof}
		From \cite[Proposition 6.1.2.11]{htt} we have that $\check{\Omega}$ factors through groupoid objects. As we are right Kan extending diagrams of the form $\ast\to Y$ in $\XX$, where $[0]$ is taken to $\ast$, it is immediate that this factors through the full category of group objects, per \cite[Definition 7.2.2.1]{htt}.
	\end{proof}

	\begin{lem}
		The functor $\check{\Omega}\colon \XX_\ast\to\Grp(\XX)$ is right adjoint to $B^+$.
	\end{lem}

	\begin{proof}
		Note that the composite functor $\XX^{\Delta^{\op}}\to\XX^{\Delta^{\op}_+}\to\XX^{\Delta^1}$ obtained by left Kan extending and then restricting (cf.~(2) and (3) of Definition \ref{defn:the simplicial functors}) has a composite right adjoint given by \textit{right} Kan extending along the inclusion $\{[0]\to[-1]\}\hookrightarrow \Delta^{\op}_+$ and then restricting to $\Delta^{\op}\subseteq \Delta^{\op}_+$. Restriction of the first composite functor to $\Grp(\XX)$ is precisely $B^+$ (whose essential image lies in $\XX_\ast$). Restricting the right adjoint to the essential image of $B^+$ is exactly $\check{\Omega}$. Suitably restricting and corestricting an adjunction still gives an adjunction, as the triangle inequalities for the unit and counit are clearly still satisfied.
	\end{proof}

	\begin{lem}\label{lem:BC is a right adjoint}
		The composite $\BB^+\check{\Omega}\colon \XX_\ast\to\XX_\ast^{\geq 1}$ is a right adjoint to the inclusion $\XX_\ast^{\geq 1}\hookrightarrow\XX_\ast$.
	\end{lem}

	\begin{proof}
		Let $\epsilon\colon \BB^+\check{\Omega}\Rightarrow id_{\XX_\ast}$ denote the counit transformation of $\BB^+\dashv\check{\Omega}$. By \cite[Lemma 7.2.2.11]{htt} this natural transformation is an equivalence when restricted to connected objects. It follows that the conditions of (the opposite of) \cite[Proposition 5.2.7.4 (3)]{htt} are satisfied. This implies that $\BB^+\check{\Omega}$ is a colocalization functor which, by definition, is a functor that is right adjoint to an inclusion.
	\end{proof}

	The following is a standard fact for colocalizations.

	\begin{cor}\label{cor: limits computed by including then BC}
		Limits in $\XX_\ast^{\geq 1}$ are computed by applying $\BB^+\check{\Omega}$ to the limit computed in $\XX_\ast$.
	\end{cor}

	\begin{proof}
		Let $F\colon K\to \XX_\ast^{\geq 1}$ be a diagram. Then we compute $\lim(F)\simeq\lim(\BB^+\check{\Omega}iF)\simeq \BB^+\cech\lim(iF)$ where $i$ denotes the inclusion functor and in the last step we use that $\BB^+\check{\Omega}$ is a right adjoint.
	\end{proof}

	\subsection{Base Change of Action Objects}\label{sec:base change}

In this section, we make several technical arguments that will give us control over the restriction/extension of scalars adjunction associated to a morphism of group objects $H\to G$ in $\Grp(\XX)$. From the representation theory perspective, this corresponds to considering induced and restricted representations. We ensure here that these functors behave ``as expected.'' An important part of this is to check that inducing and restricting group actions along a group morphism $H\to G$ ``is the same as'' thinking of this as a morphism of (coherently) associative algebra objects in $\XX$ and extending and restricting scalars along the adjunction between $\RMod_{H}(\XX)$ and $\RMod_{G}(\XX)$. We show that these agree in Proposition \ref{prop:group change is base change}. Arguably the most important result of this section is the result immediately following that, Corollary \ref{cor:induction to the point is quotient}. This result ensures that, given a fiber sequence of groups $H\to G\to K$ in $\Grp(\XX)$, taking $K$ as the purported ``quotient'' of $G$ by $H$ is valid, since we can write it as the group object obtained by ``extending scalars'' of the induced $H$-action on $G$ along the terminal morphism $H\to\ast$. Another way to say it is that the underlying object $K\in\XX$ is indeed the orbits of the $H$-action on $G$.

These goals seem to require a number of explicit constructions from \cite{ha}, in particular Sections 2.3.3, 2.4.2, 4.1.2, and 4.2.2. These may be unfamiliar to the casual reader. For such a reader, in addition to the results cited in the preceding paragraph, we recommend taking note of Example \ref{exam:puppe sequence example} and Definition \ref{defn: puppe sequence defn}, as the notion of a ``P\"uppe sequence'' will be used frequently in the sequel.

	\begin{defn}
		Let $\CC$ be a Cartesian monoidal $\infty$-category (i.e.~one in which the symmetric monoidal structure is given by the Cartesian product). Write $\Cut\colon \Alg_{\mathcal{A}ssoc}(\CC)\to \Mon(\CC)$ and $\CutR\colon \RMod(\CC)\to \RMon(\CC)$ for the equivalence from associative algebra objects to monoid objects (\cite[Proposition 2.4.2.5]{ha})and the equivalence from right module objects of $\CC$ to right action objects (\cite[Proposition 4.2.2.9]{ha}), respectively.
	\end{defn}

	\begin{rmk}
		By restricting to objects in $\RMod(\CC)$ with a fixed algebra $A$, we obtain an equivalence $\CutR_A\colon \RMod_A(\CC)\to\RMon_{\Cut (A)}(\CC)$.  We will abuse notation by writing $\RMon_A(\CC)$ instead of $\RMon_{\Cut (A)}(\CC)$ when we refer to the category of right action objects over the monoid associated to an associative algebra.
	\end{rmk}

%

	\begin{lem}\label{lem:F preserves group actions}
		Let $F\colon\XX\to\YY$ be a monoidal (i.e.~product preserving) functor of $\infty$-topoi  and write $F_\ast\colon \RMon(\XX)\to\YY^{\Delta^{\op}\times\Delta^1}$ for postcomposition with $F$. Then $F_\ast$ factors through $\RMon(\YY)$.  Moreover, $F_\ast$ restricts to a functor $\RMon_{gp}(\XX)\to\RMon_{gp}(\YY)$.
	\end{lem}

	\begin{proof}
		Note that both monoids and right action objects are defined via products and a terminal object, so it is immediate that $F_\ast$ factors through $\RMon(\YY)$. For the case of groups we need to show that $F_\ast$ preserves group objects when restricted to $\Grp(\XX)$. By \cite[Remark 5.2.6.5]{ha} it suffices to show that $(\Cut)^{-1}(FG)$ is a grouplike $\Ass$-monoid in $\YY$ for each group object $G\in\Grp(\XX)$. If we write $\mu_G\colon G\times G\to G$ for the inner face map giving the monoid structure map for $G$ (called $m$ in \cite[Definition 5.2.6.2]{htt}),  then $G$ is grouplike if $(\mu_G,\pi_1)\colon G\times G\to G\times G$ and $(\pi_2,\mu_G)\colon G\times G\to G\times G$ are both equivalences, as in \cite[Definition 5.2.6.2]{ha}. Since $F$ preserves products (hence projection maps) and $F(\mu_G)=\mu_{FG}$ by definition, the corresponding morphisms for $FG$ must still be equivalences, so $FG$ is a group object.
	\end{proof}

	We now give our definition of the induction and restriction functors for right action objects in $\infty$-topoi. Note that a right action object is a certain kind of homotopy coherent diagram in an $\infty$-category. As a result, it is difficult to take a morphism of group objects $\phi\colon G\to H$, with $G$ acting on an object $Y$, and simply ``write down'' the induced action of $H$ on $Y$. For this reason, our induction and restriction functors pass through slice categories, which are easier to describe and easier to pass back and forth between. Proposition \ref{prop:group change is base change} will then show that the below induction and restriction functors are equivalent to the usual base-change and restriction functors we get from thinking of group objects as associative algebra objects.

	\begin{defn}
		Let $\XX$ be an $\infty$-topos and let $\phi\colon H\to G$ be a morphism of group objects in $\XX$. Define the functor $\Res_H^G$  to be the composite
		\[\RMon_G(\XX)\xrightarrow{\mathrm{Act}^{\lrcorner}_G}\XX_{/\BB G}\xrightarrow{(\BB\phi)^\ast}\XX_{/\BB H}\xrightarrow{\cech^{\lrcorner}_H} \RMon_H(\XX)\] where $(\BB\phi)^\ast$ denotes pullback along $\BB\phi\colon\BB H\to \BB G$. On the other hand, define $\Ind_H^G$ to be the composite functor
		\[\RMon_H(\XX)\xrightarrow{\mathrm{Act}^{\lrcorner}_H}\XX_{/\BB H}\xrightarrow{(\BB\phi)_\ast}\XX_{/\BB G}\xrightarrow{\cech^{\lrcorner}_G}\RMon_G(\XX)\] where $(\BB\phi)_\ast$ denotes postcomposition with $\BB\phi$. It is immediate that $\Res_H^G$ is right adjoint to $\Ind_H^G$.
	\end{defn}

	\begin{exam}\label{exam:puppe sequence example}
		When $\phi\colon H\to G$ is a morphism of group objects, we will be particularly interested in the $H$-action on $\Res^G_H(\Act(G,G))$, where $G$ is acting on itself by translation (cf.~Example \ref{exam:translation action}).  Note that $\mathrm{Act}^{\lrcorner}_G(\Act(G,G))$ is the cospan with the base point on both sides, $\ast\to \BB G\leftarrow \ast$. Applying $(\BB\phi)^\ast$ and $\cech^{\lrcorner}_H$ gives a right action object that looks like the following:
		\[\begin{tikzcd}
			\vdots & \vdots \\
			{G\times H\times H} & {H\times H} \\
			{G\times H} & H \\
			G & \ast
			\arrow[shift right, from=1-1, to=2-1]
			\arrow[shift left, from=1-1, to=2-1]
			\arrow[shift right=3, from=1-1, to=2-1]
			\arrow[shift left=3, from=1-1, to=2-1]
			\arrow[shift right, from=1-2, to=2-2]
			\arrow[shift left, from=1-2, to=2-2]
			\arrow[shift left=3, from=1-2, to=2-2]
			\arrow[shift right=3, from=1-2, to=2-2]
			\arrow[from=2-1, to=2-2]
			\arrow[from=2-1, to=3-1]
			\arrow[shift left=2, from=2-1, to=3-1]
			\arrow[shift right=2, from=2-1, to=3-1]
			\arrow[from=2-2, to=3-2]
			\arrow[shift left=2, from=2-2, to=3-2]
			\arrow[shift right=2, from=2-2, to=3-2]
			\arrow[from=3-1, to=3-2]
			\arrow[shift right, from=3-1, to=4-1]
			\arrow[shift left, from=3-1, to=4-1]
			\arrow[shift left, from=3-2, to=4-2]
			\arrow[shift right, from=3-2, to=4-2]
			\arrow[from=4-1, to=4-2]
		\end{tikzcd}\]
		Recall however that $\cech^{\lrcorner}_H$ factors through taking the right Kan extension along the inclusion $\KK\hookrightarrow \Delta_+^{\op}\times\Delta^1$ which is a diagram of the following form:
		\[\begin{tikzcd}
			\vdots & \vdots \\
			{G\times H\times H} & {H\times H} \\
			{G\times H} & H \\
			G & \ast \\
			{G/H} & {\BB H}
			\arrow[shift right, from=1-1, to=2-1]
			\arrow[shift left, from=1-1, to=2-1]
			\arrow[shift right=3, from=1-1, to=2-1]
			\arrow[shift left=3, from=1-1, to=2-1]
			\arrow[shift right, from=1-2, to=2-2]
			\arrow[shift left, from=1-2, to=2-2]
			\arrow[shift left=3, from=1-2, to=2-2]
			\arrow[shift right=3, from=1-2, to=2-2]
			\arrow[from=2-1, to=2-2]
			\arrow[from=2-1, to=3-1]
			\arrow[shift left=2, from=2-1, to=3-1]
			\arrow[shift right=2, from=2-1, to=3-1]
			\arrow[from=2-2, to=3-2]
			\arrow[shift left=2, from=2-2, to=3-2]
			\arrow[shift right=2, from=2-2, to=3-2]
			\arrow[from=3-1, to=3-2]
			\arrow[shift right, from=3-1, to=4-1]
			\arrow[shift left, from=3-1, to=4-1]
			\arrow[shift left, from=3-2, to=4-2]
			\arrow[shift right, from=3-2, to=4-2]
			\arrow[from=4-1, to=4-2]
			\arrow[from=4-1, to=5-1]
			\arrow["\lrcorner"{anchor=center, pos=0.125}, draw=none, from=4-1, to=5-2]
			\arrow[from=4-2, to=5-2]
			\arrow[from=5-1, to=5-2]
		\end{tikzcd}\]
		But the morphism $G/H\to \BB H$, by construction, is precisely the fiber of $\BB\phi\colon\BB H\to \BB G$. Therefore by composing fiber sequences and extending we obtain a kind of Puppe sequence
		\[H\xrightarrow{\phi_1}G\to G/H\to\BB H\xrightarrow{\BB \phi}\BB G\]
	\end{exam}

	\begin{defn}\label{defn: puppe sequence defn}
		When $\phi\colon H\to G$ is a morphism of group objects in an $\infty$-topos, we will call the composite fiber sequence of Example \ref{exam:puppe sequence example} the \textit{Puppe sequence associated to $\phi$}. We will call the object $G/H$ the \textit{quotient of $G$ by $H$ with respect to $\phi$}.
	\end{defn}

	\begin{prop}\label{prop:group change is base change}
		Let $\phi\colon H\to G$ be a morphism of group objects in an $\infty$-topos $\XX$. Then the composite \[\RMod_H(\XX)\xrightarrow{\CutR_H}\RMon_H(\XX)\xrightarrow{\Ind_H^G}\RMon_G(\XX)\xrightarrow{(\CutR_G)^{-1}}\RMod_G(\XX)\]
		is equivalent to the extension of scalars functor $\phi_\ast\colon \RMod_H(\XX)\to\RMod_G(\XX)$. Consequently, the composite
		\[\RMod_G(\XX)\xrightarrow{\CutR_G}\RMon_G(\XX)\xrightarrow{\Res_H^G}\RMon_H(\XX)\xrightarrow{(\CutR_H)^{-1}}\RMod_H(\XX)\]
		is equivalent to the restriction of scalars functor $\phi^\ast\colon \RMod_G(\XX)\to\RMod_H(\XX)$.\end{prop}

	\begin{proof}
		By \cite[Theorem 4.8.4.1]{ha}, we have that colimit preserving functors $\RMod_H(\XX)\to\RMod_G(\XX)$ are equivalent to tensoring with $(H,G)$-bimodules. This equivalence implies that a specific colimit preserving functor $F\colon \RMod_H(\XX)\to\RMod_G(\XX)$ is equivalent to $-\otimes_HF(H)$.  The first composite above evidently takes $H$ with the translation action on itself to $G$ with the translation action on itself and so is equivalent to the standard base change functor. The second statement follows from uniqueness of adjoints.
	\end{proof}

	\begin{cor}\label{cor:induction to the point is quotient}
		Let $\Act(X,G)\in\RMon_G(\XX)$ be a right action of a group $G\in\Grp(\XX)$ on $X\in\XX$ and let  $t\colon G\to \ast$ be the terminal group morphism. Then there is an equivalence $X/G\simeq t_\ast(X)=X\otimes_G\ast$, i.e.~the quotient of $X$ by $G$ is equivalent to applying the base change functor $-\otimes_G\ast\colon\RMod_G(\XX)\to\RMod_\ast(\XX)$ to $X$.
	\end{cor}

	\begin{proof}
		One can check, by correctly invoking Propositions 2.4.2.5, 3.4.2.1, 4.2.2.9 of \cite{ha} along with Corollaries  4.2.2.16 and 4.5.1.5 of ibid., that the following diagram commutes, where the vertical functors are the forgetful ones and all functors are equivalences:
		\[\begin{tikzcd}
			\RMod_{\ast}(\XX)\ar[r,"Cut^R_\ast"] \ar[d] & \RMon_{\ast}(\XX)\ar[d]\\ \XX\ar[r,"id_{\XX}"] & \XX
		\end{tikzcd}\]
		Hence the first displayed composite of Proposition \ref{prop:group change is base change} can be rewritten as
		\[\RMod_{G}(\XX)\xrightarrow{\CutR_\ast}\RMon_G(\XX)\xrightarrow{-/G} \XX\xrightarrow{id_{\XX}}\XX.\]
	\end{proof}

	Before stating our next main result, we need that base change is compatible with monoidal functors.

	\begin{lem}\label{lem:pullbackmonoidalcommute}
		Let $\CC$ and $\DD$ be monoidal $\infty$-categories, $F\colon \CC\to \DD$ a monoidal functor and $\phi\colon A\to B$ a morphism of associative algebras in $\CC$. Then the following diagram commutes:
		\[\begin{tikzcd}
			{\RMod_B(\CC)} & {\RMod_A(\CC)} \\
			{\RMod_{FB}(\DD)} & {\RMod_{FA}(\DD)}
			\arrow["{\phi^\ast}", from=1-1, to=1-2]
			\arrow["{\RMod(F)_B}"', from=1-1, to=2-1]
			\arrow["{\RMod(F)_A}", from=1-2, to=2-2]
			\arrow["{(F\phi)^\ast}"', from=2-1, to=2-2]
		\end{tikzcd}\]
		where $\RMod(F)_B$ and $\RMod(F)_A$ refer to the restriction of~~$\RMod(F)$ (i.e.~postcomposition with $F$ as a functor $\mathcal{RM}^\otimes\to\CC^{\otimes}$) to $\RMod_B(\CC)$ and $\RMod_A(\CC)$ respectively.
	\end{lem}

	\begin{proof}
		The proof of the lemma requires substantial use of the theory of Cartesian fibrations and Cartesian morphisms as developed in \cite{htt}. For the sake of brevity, we will do so without further explanation. From \cite[Corollary 4.2.3.2]{ha}, along with the definition of monoidal functor, we have a commutative square whose vertical legs, which project away the module coordinate, are Cartesian fibrations and whose upper horizontal morphism preserves Cartesian morphisms:
		\[\begin{tikzcd}
			{\RMod(\CC)} & {\RMod(\DD)} \\
			{\Alg(\CC)} & {\Alg(\DD)}
			\arrow[from=1-1, to=1-2]
			\arrow[from=1-1, to=2-1, "p"]
			\arrow[from=1-2, to=2-2, "q"]
			\arrow[from=2-1, to=2-2]
		\end{tikzcd}\]
		Moreover, the morphisms in $\RMod(\CC)$ which are $p$-Cartesian are precisely those which induce an equivalence on the underlying module (in particular, restriction of scalar morphisms).
		By taking the pullback of the lower right cospan in the above square, we obtain a commutative diagram

		\[\begin{tikzcd}
			{\RMod(\CC)} \\
			& {\mathcal{E}} & {\RMod(\DD)} \\
			\\
			& {\Alg(\CC)} & {\Alg(\DD)}
			\arrow[from=1-1, to=2-2]
			\arrow[curve={height=-6pt}, from=1-1, to=2-3]
			\arrow["p"', curve={height=6pt}, from=1-1, to=4-2]
			\arrow[from=2-2, to=2-3]
			\arrow["r", from=2-2, to=4-2]
			\arrow["\lrcorner"{anchor=center, pos=0.125}, draw=none, from=2-2, to=4-3]
			\arrow["q", from=2-3, to=4-3]
			\arrow[from=4-2, to=4-3]
		\end{tikzcd}\]

		By \cite[Remark 2.4.1.12]{htt} and \cite[Proposition 2.4.2.3 (2)]{htt} we have that $r$ is still a Cartesian fibration and that a morphism of $\mathcal{E}$ is $r$-Cartesian if and only if its image in $\RMod(\DD)$ is $q$-Cartesian. Applying the straightening functor to the leftmost triangle of the above diagram (which is a morphism in the $\infty$-category of Cartesian fibrations over $\Alg(\CC)$), we see that for any $\phi\in\Alg(\CC)$ the resulting naturality square is exactly the desired diagram. Here we are using the fact that fibers are preserved by pullbacks so that $r^{-1}(A)\simeq q^{-1}(FA)\simeq \RMod_{FA}(\DD)$ for every $A\in\Alg(\CC)$.
	\end{proof}

	\begin{prop}\label{prop:res commutes with monoidal functors}
		Let $F\colon\XX\to\YY$ be a monoidal (i.e.~product preserving) functor of $\infty$-topoi. Then the following diagram commutes:
		\[\begin{tikzcd}
			{\RMon_G(\XX)} & {\RMon_H(\XX)} \\
			{\RMon_{FG}(\YY)} & {\RMon_{FH}(\YY)}
			\arrow["{\Res_H^G}", from=1-1, to=1-2]
			\arrow["{F_\ast}"', from=1-1, to=2-1]
			\arrow["{F_\ast}", from=1-2, to=2-2]
			\arrow["{\Res_{FH}^{FG}}"', from=2-1, to=2-2]
		\end{tikzcd}\]
		where $F_\ast$ denotes postcomposition with $F$.
	\end{prop}

	\begin{proof}
		First note that the vertical morphisms in the diagram are well-defined by Lemma \ref{lem:F preserves group actions}. Now consider the extended diagram:
		\[\begin{tikzcd}
			{\RMod_G(\XX)} & {\RMon_G(\XX)} & {\RMon_H(\XX)} & {\RMod_H(\XX)} \\
			{\RMod_{FG}(\YY)} & {\RMon_{FG}(\YY)} & {\RMon_{FH}(\YY)} & {\RMod_{FH}(\YY)}
			\arrow["{\CutR_G}", from=1-1, to=1-2]
			\arrow["{\RMod(F)_G}"', from=1-1, to=2-1]
			\arrow["{\Res_H^G}", from=1-2, to=1-3]
			\arrow["{F_\ast}"', from=1-2, to=2-2]
			\arrow["{(\CutR_H)^{-1}}", from=1-3, to=1-4]
			\arrow["{F_\ast}", from=1-3, to=2-3]
			\arrow["{\RMod(F)_H}", from=1-4, to=2-4]
			\arrow["{\CutR_{FG}}", from=2-1, to=2-2]
			\arrow["{\Res_{FH}^{FG}}", from=2-2, to=2-3]
			\arrow["{(\CutR_{FH})^{-1}}", from=2-3, to=2-4]
		\end{tikzcd}\]
		The leftmost and rightmost squares commute because both vertical morphisms are simply postcomposition with $F$. From Proposition \ref{prop:group change is base change}, the uppermost and lowermost ``squares'' of the following diagram (the semicircles on top and bottom) also commute:
		\[\begin{tikzcd}
			{\RMod_G(\XX)} & {\RMon_G(\XX)} & {\RMon_H(\XX)} & {\RMod_H(\XX)} \\
			{\RMod_{FG}(\YY)} & {\RMon_{FG}(\YY)} & {\RMon_{FH}(\YY)} & {\RMod_{FH}(\YY)}
			\arrow["{\CutR_G}", from=1-1, to=1-2]
			\arrow["{\phi^\ast}", curve={height=-45pt}, from=1-1, to=1-4]
			\arrow["{\RMod(F)}"', from=1-1, to=2-1]
			\arrow["{\Res_H^G}", from=1-2, to=1-3]
			\arrow["{F_\ast}"', from=1-2, to=2-2]
			\arrow["{(\CutR_H)^{-1}}", from=1-3, to=1-4]
			\arrow["{F_\ast}", from=1-3, to=2-3]
			\arrow["{\RMod(F)}", from=1-4, to=2-4]
			\arrow["{\CutR_{FG}}", from=2-1, to=2-2]
			\arrow["{(F\phi)^\ast}"', curve={height=45pt}, from=2-1, to=2-4]
			\arrow["{\Res_{FH}^{FG}}", from=2-2, to=2-3]
			\arrow["{(\CutR_{FH})^{-1}}", from=2-3, to=2-4]
		\end{tikzcd}\]
		Finally, the outermost square commutes by Lemma \ref{lem:pullbackmonoidalcommute}.
		A standard diagram chase gives the result.
	\end{proof}

	Suppose that $F\colon \XX\to\YY$ is a monoidal functor between $\infty$-topoi and $G\in\Grp(\XX)$. We have from Lemma \ref{lem:F preserves group actions} that $F$ preserves $G$-actions. On the other hand, as a monoidal functor, $F$ preserves group morphisms $\phi\colon H\to G$. Therefore $\phi$ induces an $H$-action on $G$, and $F\phi$ induces an $FH$-action on $FG$. Therefore there are, a priori, two ways to take a group map $\phi\colon H\to G$ and produce an $FH$-action on $FG$. The following corollary, which follows immediately from Proposition \ref{prop:res commutes with monoidal functors}, is essentially the statement that these two actions agree.

	\begin{cor}\label{cor:actF is Fact}
		Let $\phi\colon G\to H$ be a morphism of group objects in $\XX$ and $F\colon \XX\to \YY$ a monoidal functor between $\infty$-topoi. If we write $\Act^\phi(G,H)$ for right $H$-action on $G$ induced by $\phi$ and $\Act^{F\phi}(FG,FH)$ for the $FH$-action on $FG$ induced by $F\phi$ then there is an equivalence:
		\[F\circ\Act^\phi(G,H)\simeq \Act^{F\phi}(FG,FH)\]
	\end{cor}

	\subsection{The Space Case}

	Given an $\infty$-topos $\XX$ and a group object $G\in\Grp(\XX)$, we have two equivalent ways of thinking about $G$-actions: as right action objects in $\RMon_G(\XX)$; or as right modules in $\RMod_G(\XX)$ for $G$ thought of as a (grouplike) $\Ass$-algebra. In the case that $\XX=\Spaces$, the $\infty$-topos of $\infty$-groupoids, we have yet another way. Using Corollary \ref{cor:G modules are BG comodules} we have that the category of $G$-actions $\RMon_G(\Spaces)$ is equivalent to the slice category $\Spaces_{/\BB G}$. Then the next result follows from the straightening/unstraightening equivalence of \cite{htt}.

	\begin{prop}\label{prop:straightening G-actions}
		There is an equivalence of $\infty$-categories $\RMon_G(\Spaces)\xrightarrow{\mathrm{Act}^{\lrcorner}_G}\Spaces_{/\BB G}\xrightarrow{St} \Fun(\BB G,\Spaces)$.
	\end{prop}

	\begin{rmk}
		Note that if $\Act(X,G)\in\RMon(\Spaces)$ we have that $\BB G$ is a pointed $\infty$-groupoid and the associated functor $\BB G\to \Spaces$ takes the base point to $X\in\Spaces$. Moreover, every morphism in $\BB G$ is invertible, so this functor factors through $\BB\mathrm{Aut}_{\Spaces}(X)$. Indeed, any action of any group object $H$ on $X$ corresponds to a morphism $\BB H\to\BB \mathrm{Aut}_{\Spaces}(X)$. In other words, the space of actions on $X$ by group objects of $\Spaces$ is equivalent to the space of group maps into $\mathrm{Aut}_{\Spaces}(X)$. Of course it's true that $\Act(X,G)\in\RMon(\XX)$ corresponds to a group map $G\to\mathrm{Aut}_{\XX}(X)$ for any $\infty$-topos $\XX$, for instance by the results of \cite[Section 4.7.1]{ha}, but the case of $\Spaces$ is particularly easy to see.
	\end{rmk}

	Recall that the application of the unstraightening functor $Un$ to a functor $F\colon Z\to \Spaces$ can be given by taking the pullback of the universal left fibration $\mcal U\to\Spaces$ along $F$ (see, for instance, \cite[5.2.8]{cisinski-HCHA}). We can use this to prove that two natural notions of ``quotient'' by a group action are compatible with the equivalence of Proposition \ref{prop:straightening G-actions}.

	\begin{prop}
		Let $\Act(X,G)\in\RMon(\Spaces)$ have associated functor $F_X\colon\BB G\to\Spaces$. Then $X/G$ is equivalent to the colimit of $F_X$.
	\end{prop}

	\begin{proof}
		From Remark \ref{rmk:pullback square at bottom of action} we have that $X/G$ fits into a fiber sequence $X\to X/G\to \BB G$. Proposition \ref{prop:straightening G-actions} implies that the unstraightening of the  functor $F_X\colon\BB G\to \Spaces$ associated to $\Act(X,G)$ corresponds to the given map $X/G\to \BB G$. In other words, we have a composite pullback diagram
		\[\begin{tikzcd}
			X & {X/G} & {\mathcal{U}} \\
			\ast & {\BB G} & \Spaces
			\arrow[from=1-1, to=1-2]
			\arrow[from=1-1, to=2-1]
			\arrow["\lrcorner"{anchor=center, pos=0.125}, draw=none, from=1-1, to=2-2]
			\arrow[from=1-2, to=1-3]
			\arrow[from=1-2, to=2-2]
			\arrow["\lrcorner"{anchor=center, pos=0.125}, draw=none, from=1-2, to=2-3]
			\arrow[from=1-3, to=2-3]
			\arrow[from=2-1, to=2-2]
			\arrow[from=2-2, to=2-3, "F_X"]
		\end{tikzcd}\]

		From \cite[Corollary 3.3.4.6]{htt}, the colimit in $\Spaces$ of the bottom right horizontal functor is equivalent to $X/G$.
	\end{proof}

	\subsection{Monomorphisms and Epimorphisms of Group Objects}\label{sec:monos and epis of groups}

	In the category of discrete groups every morphism can be factored as a surjection followed by an injection. This lifts the usual surjection/injection factorization system on the 1-category of sets. In this section we extend this factorization system to $\Grp(\XX)$ for any $\infty$-topos $\XX$. Recall the following result from \cite{htt}.

	\begin{prop}[{\cite[Remark 5.2.8.16]{htt}}]\label{prop:epimono factorization system}
		If~~$n\Conn$ denotes the class of $n$-connective morphisms, e.g.~effective epimorphisms in the case that $n=0$, in $\XX$ and $n\Trunc$ denotes the $(n-1)$-truncated morphisms of $\XX$, then $( n\Conn,n\Trunc)$ is a factorization system on $\XX$ in the sense of \cite[Section 5.2.8]{htt}.
	\end{prop}

	Proposition \ref{prop:epimono factorization system} also appears in \cite{anel-blakers-massey} although some translation in terminology is necessary.

	\begin{prop}\label{prop: pointed factorization}
		Let $\XX$ be an $\infty$-topos equipped with the factorization system $( n\Conn,n\Trunc)$ of Proposition \ref{prop:epimono factorization system}. Then there is a factorization system $( n\Conn_\ast,n\Trunc_\ast)$ on $\XX_\ast$ where $ n\Conn_\ast$ and $n\Trunc_\ast$ are the classes of morphisms in $\XX_\ast$ which are in $ n\Conn$, respectively $n\Trunc$, after forgetting to $\XX$.
	\end{prop}

	\begin{proof}
		We check the conditions of \cite[Definition 5.2.8.8]{htt}. The two classes of morphisms are closed under retracts because they are closed under retracts in $\XX$. The usual formulae for mapping spaces in slice and coslice $\infty$-categories and standard manipulations with pullback squares   imply that morphisms of $ n\Conn_\ast$ are still left orthogonal to morphisms of $n\Trunc_\ast$. Now let $f\colon X\to Y$ be a morphism of $\XX_\ast$ which factors as $X\xrightarrow{f_L} Z\xrightarrow{f_R} Y$ with respect to $( n\Conn,n\Trunc)$. By choosing the pointing of $Z$ to be the composite $\ast\to X\xrightarrow{f_L} Z$, we obtain a factorization in $\XX_\ast$.
	\end{proof}

	\begin{cor}
		The factorization system $( n\Conn_\ast,n\Trunc_\ast)$ of Proposition \ref{prop: pointed factorization} restricts to a factorization system $( n\Conn_\ast^{\geq 1},n\Trunc_\ast^{\geq 1})$ on $\XX_\ast^{\geq 1}$.
	\end{cor}

	\begin{proof}
		Suppose that $f\colon X\to Y$ is a morphism in $\XX_\ast^{\geq 1}$. It suffices to show that if $f$ factors in $\XX_\ast$ as $X\xrightarrow{f_L} Z\xrightarrow{f_R} Y$ then $Z$ is connected. By again applying Lemma \ref{lem:section is -1 connected} to $\ast\to X\to \ast$ and using that $X\to\ast$ is $1$-connective we have that $\ast\to X$ is $0$-connective, i.e.~an effective epimorphism in $\XX$. We also, of course, have that $X\xrightarrow{f_L}Z$ is $0$-connective (for instance, by \cite[Proposition 6.5.1.16 (1)]{htt}). By composing effective epimorphisms, e.g.~using \cite[5.2.8.6 (4)]{htt}, we have that $\ast\to Z$ must be an effective epimorphism from which it follows, again by Lemma \ref{lem:section is -1 connected}, that $Z\to \ast$ is $1$-connective, hence $Z$ is connected.
	\end{proof}

	\begin{lem}\label{lem:connective fiber connective map}
		Suppose $f\colon X\to Y$ is a morphism of $\XX$ with $Y$ pointed and connected. Then the fiber $\fib(f)$ is $n$-connective if and only if $f$ is $n$-connective.
	\end{lem}

	\begin{proof}
		We have a pullback diagram
		\[\begin{tikzcd}
			{\fib(f)} & X \\
			\ast & Y
			\arrow[from=1-1, to=1-2]
			\arrow[from=1-1, to=2-1]
			\arrow["\lrcorner"{anchor=center, pos=0.125}, draw=none, from=1-1, to=2-2]
			\arrow["f", from=1-2, to=2-2]
			\arrow[from=2-1, to=2-2]
		\end{tikzcd}\]
		in which, since $Y$ is connected, the bottom horizontal arrow is an effective epimorphism. Since, by definition, $\fib(f)\to\ast$ is $n$-connective exactly when $\fib(f)$ is $n$-connective, the result then follows from \cite[Proposition 6.5.1.16 (6)]{htt}.
	\end{proof}

	\begin{defn}\label{defn:n image for spaces}
		Let $f\colon X\to Y$ be a morphism in either $\XX$, $\XX_\ast$ or $\XX_\ast^{\geq 1}$ with a factorization $X\to Z\to Y$ in the $( n\Conn,n\Trunc)$, $( n\Conn_\ast,n\Trunc_\ast)$ or $( n\Conn_\ast^{\geq 1},n\Trunc_\ast^{\geq 1})$ factorization systems respectively. Then we refer to $Z$ as the \textit{$n$-image of $f$} and write it $\im_n(f)$.
	\end{defn}

	\begin{rmk}
		Note that $\im_n(f)$ is uniquely determined up to a contractible space of choices by \cite[Proposition 5.2.8.17]{htt}.
	\end{rmk}

	\begin{defn}
		Let $( \Conn_{n-1}^{\Grp},\Trunc_{n-1}^{\Grp})$ denote the factorization system on $\Grp(\XX)$ induced by transferring the factorization system $( n\Conn_\ast^{\geq 1},n\Trunc_\ast^{\geq 1})$ along the equivalence $\cech\colon \XX_\ast^{\geq 1}\xrightarrow{\simeq}\Grp(\XX)$.
	\end{defn}

	\begin{prop}\label{prop:n effective group maps are n connective}
		A morphism $f\colon H\to G$ of~~$\Grp(\XX)$ is in $ \Conn_{n}^{\Grp}$ if and only if it is $n$-connective as a morphism on underlying objects of $\XX$.
	\end{prop}

	\begin{proof}
		By Corollary \ref{cor:omega gives loop space action}, the underlying morphism of $f$ is $\Omega\BB f\colon \Omega\BB H\to\Omega\BB G$. Moreover, $f\in\Conn_n^{\Grp}$ if and only if the map $\BB H\to\BB G$ is $(n+1)$-connective. Therefore it suffices to show that if $X\to Y$ is a morphism of $\XX^{\geq 1}_\ast$ then it is $(n+1)$-connective if and only if $\Omega X\to\Omega Y$ is $n$-connective in $\XX$.

		We inductively apply \cite[Proposition 6.5.1.18]{htt}. For the base case, take $n=-1$. Suppose that $f\colon X\to Y$ is a morphism of $\XX_{\ast}^{\geq 1}$. Simply by virtue of being a morphism, $\Omega f\colon \Omega X\to \Omega Y$ is always $(-1)$-connected. Therefore it is only necessary to show that $f$ is $0$-connective. This follows immediately from recalling that for any pointed connected object the pointing itself is an effective epimorphism and applying \cite[Proposition 6.5.1.16]{htt} to the following commutative diagram
		\[\begin{tikzcd}
			& X \\
			\ast && Y
			\arrow[from=1-2, to=2-3]
			\arrow[from=2-1, to=1-2]
			\arrow[from=2-1, to=2-3]
		\end{tikzcd}\]

		Now suppose that the statement holds for $0\leq k\leq n$. Suppose that $f\colon X\to Y$ is a morphism of pointed connected objects. Then $f$ is $(n+1)$-connective if and only if the diagonal map $\delta\colon X\to X\times_Y X$ is $n$-connective. But by the inductive hypothesis $\delta$ is $n$-connective if and only if $\Omega\delta\colon \Omega X\to \Omega(X\times_Y X)\simeq\Omega X\times_{\Omega Y}\Omega X$ is $(n-1)$-connective. This is the case if and only if $\Omega f\colon \Omega X\to \Omega Y$ is $n$-connective, completing the proof.
	\end{proof}

	\begin{prop}\label{prop:n truncated group maps are n truncated}
		A morphism $\phi\colon H\to G$ is in $\Trunc_{n}^{\Grp}$ if and only if its underlying map in $\XX$ is $(n-1)$-truncated.
	\end{prop}

	\begin{proof}

		Similarly to the preceding proposition, it suffices to show that a morphism $f\colon X\to Y$ in $\XX_\ast^{\geq 1}$ is $n$-truncated if and only if $\Omega f\colon \Omega X\to \Omega Y$ is $(n-1)$-truncated. We can also inductively use \cite[Lemma 5.5.6.15]{htt} in a similar way to reduce to the case of $n=-1$.

		First we show that a morphism $f\colon X\to Y$ in $\XX_\ast^{\geq 1}$ is $(-1)$-truncated if and only if it is an equivalence. We have a commutative diagram in which both pointings are effective epimorphisms: 
		\[\begin{tikzcd}
			\ast && Y \\
			& X
			\arrow[two heads, from=1-1, to=1-3]
			\arrow[two heads, from=1-1, to=2-2]
			\arrow[ "f"', from=2-2, to=1-3]
		\end{tikzcd}\]
		Because factorizations are unique by \cite[Proposition 5.2.8.17]{htt}, and because $f$ is $(-1)$-truncated, it must be the case that the lower composite in the above diagram is the unique factorization of the pointing $\ast\to Y$ in the factorization system $(0\Conn,0\Trunc)$ (or its lifts to $\XX_\ast$ or $\XX_\ast^{\geq 1}$). By \cite[Corollary 6.2.3.12]{htt} we have that $f$ is also an effective epimorphism. Therefore $f$ is an equivalence (by, for instance, \cite[Corollary 1.37]{rasekh-truncations}). Of course if $f$ is an equivalence, i.e.~$(-2)$-truncated, then it is also $(-1)$-truncated.

		Now we check that $f\colon X\to Y$ in $\XX_\ast^{\geq 1}$ is an equivalence if and only if $\Omega f$ is an equivalence. This follows immediately from \cite[Corollary 5.2.6.18]{ha}, which shows that $\Omega$ is conservative in any $\infty$-topos. Therefore $f$ is $(-1)$-truncated if and only if $\Omega f$ is an equivalence, i.e.~$(-2)$-truncated.
	\end{proof}

	\begin{defn}\label{defn:n image for groups}
		Let $f\colon H\to G$ be a morphism of $\Grp(\XX)$ with factorization $H\xrightarrow{f_L} K\xrightarrow{f_R} G$ induced by $( \Conn_n^{\Grp},\Trunc_n^{\Grp})$. Then we will refer to $K$ as the \textit{$n$-image} of $f$.
	\end{defn}

	Note that by Propositions \ref{prop:n effective group maps are n connective} and \ref{prop:n truncated group maps are n truncated} there is no ambiguity between Definitions \ref{defn:n image for spaces} and \ref{defn:n image for groups}.

	\begin{defn}
		When $n=0$ then the left class of morphisms in any of the above factorization systems will be referred to as effective epimorphisms and the right class will be called monomorphisms.
	\end{defn}
	
	\begin{cor}\label{cor:effective epis right cancel in groups}
Let $A\xrightarrow{f}B\xrightarrow{g}C$ be a composite of morphisms in $\Grp(\XX)$. If $g\circ f$ is an effective
epimorphism, then $g$ is an effective epimorphism.
\end{cor}

\begin{proof}
Let $U_{\Grp}\colon\Grp(\XX)\to\XX$ denote the forgetful functor.
By Proposition
\ref{prop:n effective group maps are n connective},
the composite
\[
U_{\Grp}(A)\xrightarrow{U_{\Grp}(f)}
U_{\Grp}(B)\xrightarrow{U_{\Grp}(g)}
U_{\Grp}(C)
\]
is an effective epimorphism in $\XX$. Hence
\cite[Corollary 6.2.3.12(2)]{htt}, applied in the
$\infty$-topos $\XX$, implies that $U_{\Grp}(g)$ is an
effective epimorphism. Another application of Proposition
\ref{prop:n effective group maps are n connective}
shows that $g$ is an effective epimorphism in $\Grp(\XX)$.
\end{proof}

	\subsection{An Orbit Stabilizer Theorem}

	The classical Orbit-Stabilizer Theorem for discrete groups states that if a group $G$ acts on a set $X$, then (i) the set $Stab_G(x)=\{g\in G:gx=x\}$ is a subgroup of $G$ and (ii) there is a canonical isomorphism between the quotient $G/Stab_G(x)$ and the orbits $Orb_G(x)=\{y\in X:gx=y~\text{for some}~g\in G\}$. We prove this theorem for group objects in $\infty$-topoi here with an important caveat. The first is that asking for $Stab_G(X)$ to be a ``subgroup'' of $G$ is a much weaker condition in our setting. We only exhibit a canonical morphism of groups $Stab_G(x)\to G$. Our definitions below agree with those of \cite[Section 2.3]{gripaios-generalized-symmetries}.

	\begin{lem}
		Let $\Act(X,G)$ be a right action object in $\XX$ for $G\in\Grp(\XX)$ with corresponding fiber sequence $X\xrightarrow{q} X/G\xrightarrow{p} \BB G$. Let $x\colon \ast\to X$ be a point of $X$ and write $z\colon \ast\to X/G$ for its composition with $q$. Then the fiber of $q$ at $z$ is equivalent to $G$.
	\end{lem}
	\begin{proof}
		The choice of $x$ (and thus $z$) and the fact that $BG$ is connected allows us to lift the composite $X\to X/G\to \BB G$ to the $\infty$-category $\XX_\ast$. Extending once to the left in the standard way gives a further fiber sequence $G\simeq\Omega \BB G\to X\to X/G$.
	\end{proof}

	\begin{defn}
		Given $\Act(X,G)$, a right action object in $\XX$ for $G\in\Grp(\XX)$, and a point $x\in X$, let $i\colon G\to X$ denote the fiber map extending $X\xrightarrow{q}X/G\xrightarrow{p}\BB G$ in the preceding lemma. Define \textit{the orbits of $x\in X$ with respect to $p\colon X/G\to\BB G$}, denoted $Orb_p(x)$, to be $\im_0(i)$. Define the \textit{stabilizer of $x\in X$ with respect to $p$}, denoted $Stab_p(x)$, to be $\check{\Omega}(\ast\xrightarrow{z} X/G)\in\Grp(\XX)$. Note that the underlying object of $Stab_p(x)$ is the loop space object $\Omega_z(X/G)$.
	\end{defn}

	The analogy with the notion of orbits and stabilizers in discrete group theory is as follows. Given an element $x\in X$ in a set acted upon by a group $G$, the orbits of $x$ under the group action are precisely the preimage of $xG$ under the projection $X\to X/G$. Note that the orbit set of $x$ under the $G$ action always admits a function from $G$ given by taking $g$ to $xg$. This is map $i$ that we are describing above. Continuing this analogy then, the stabilizer of $x\in X$ should be the fiber of $x$ with respect to $i$. In other words, it is the ``elements'' $g\in G$ such that $xg=x$. We can paste together pullback squares as below
	\[
	\begin{tikzcd}
		{Stab_p(x)} & \ast \\
		G & X \\
		\ast & {X/G}
		\arrow[from=1-1, to=1-2]
		\arrow[from=1-1, to=2-1]
		\arrow["\lrcorner"{anchor=center, pos=0.125}, draw=none, from=1-1, to=2-2]
		\arrow[from=1-2, to=2-2, "x"]
		\arrow["i", from=2-1, to=2-2]
		\arrow[from=2-1, to=3-1]
		\arrow["\lrcorner"{anchor=center, pos=0.125}, draw=none, from=2-1, to=3-2]
		\arrow["q", from=2-2, to=3-2]
		\arrow[from=3-1, to=3-2, "z"]
	\end{tikzcd}
\] to see that this is precisely the space of loops in $X/G$ based at $q(x)=z$.

	\begin{thm}\label{thm:orbit stabilizer}
		Given  a right action object $\Act(X,G)$ in $\XX$ with defining morphism $p\colon X/G\to\BB G$  and a morphism $x\colon\ast\to X$ there is a canonical equivalence $G/Stab_p(x)\simeq Orb_p(x)$.
	\end{thm}

	\begin{proof}

		First we are explicit about the construction of the morphism $\phi\colon Stab_p(x)\to G$. By definition it is $\check{\Omega}$ applied to $p\colon X/G\to \BB G$, but for the case of exposition it will be useful to think of this as a diagram $\Delta^{\op}\times\Delta^1\to\XX$:
		\[\begin{tikzcd}
			\vdots & \vdots \\
			{\Omega_z(X/G)\times\Omega_z(X/G)} & {G\times G} \\
			{\Omega_z(X/G)} & G \\
			\ast & \ast \\
			{X/G} & {\BB G}
			\arrow[shift right, from=1-1, to=2-1]
			\arrow[shift left, from=1-1, to=2-1]
			\arrow[shift right=3, from=1-1, to=2-1]
			\arrow[shift left=3, from=1-1, to=2-1]
			\arrow[shift right, from=1-2, to=2-2]
			\arrow[shift left, from=1-2, to=2-2]
			\arrow[shift left=3, from=1-2, to=2-2]
			\arrow[shift right=3, from=1-2, to=2-2]
			\arrow["{\phi_2}", from=2-1, to=2-2]
			\arrow[from=2-1, to=3-1]
			\arrow[shift left=2, from=2-1, to=3-1]
			\arrow[shift right=2, from=2-1, to=3-1]
			\arrow[from=2-2, to=3-2]
			\arrow[shift left=2, from=2-2, to=3-2]
			\arrow[shift right=2, from=2-2, to=3-2]
			\arrow["{\phi_1}", from=3-1, to=3-2]
			\arrow[shift right, from=3-1, to=4-1]
			\arrow[shift left, from=3-1, to=4-1]
			\arrow[shift left, from=3-2, to=4-2]
			\arrow[shift right, from=3-2, to=4-2]
			\arrow[from=4-1, to=4-2, "\phi_0"]
			\arrow["z"', from=4-1, to=5-1]
			\arrow[from=4-2, to=5-2]
			\arrow["p", from=5-1, to=5-2]
		\end{tikzcd}\]

		By definition, $\BB\phi\colon\BB Stab_p(x)\to \BB G$ is obtained by taking the colimit of the non-augmented part of the above diagram. This factors the above diagram so that the following commutes:
		\[\begin{tikzcd}
			\vdots & \vdots \\
			{\Omega_z(X/G)} & G \\
			\ast & \ast \\
			{\BB Stab_p(x)} & {\BB G} \\
			{X/G} & {\BB G}
			\arrow[from=1-1, to=2-1]
			\arrow[shift left=2, from=1-1, to=2-1]
			\arrow[shift right=2, from=1-1, to=2-1]
			\arrow[from=1-2, to=2-2]
			\arrow[shift left=2, from=1-2, to=2-2]
			\arrow[shift right=2, from=1-2, to=2-2]
			\arrow["{\phi_1}", from=2-1, to=2-2]
			\arrow[shift right, from=2-1, to=3-1]
			\arrow[shift left, from=2-1, to=3-1]
			\arrow[shift left, from=2-2, to=3-2]
			\arrow[shift right, from=2-2, to=3-2]
			\arrow["{\phi_0}", from=3-1, to=3-2]
			\arrow[from=3-1, to=4-1]
			\arrow[from=3-2, to=4-2]
			\arrow["{\BB\phi}", from=4-1, to=4-2]
			\arrow[from=4-1, to=5-1]
			\arrow[from=4-2, to=5-2,"\simeq"]
			\arrow["p", from=5-1, to=5-2]
		\end{tikzcd}\]

		Because $\XX$ is an $\infty$-topos we have that $\ast\to\BB Stab_p(x)$ is an effective epimorphism, i.e.~$\BB Stab_p(x)$ is connected. On the other hand, \cite[Proposition 6.2.3.4]{htt} implies that $\BB Stab_p(x)\to X/G$ is $(-1)$-truncated. Recall that we construct the quotient of $G$ by $Stab_p(x)$ via $\phi$ by taking the pullback of $\ast\to\BB G$ along $\BB\phi\colon\BB Stab_p(x)\to \BB G$.

		Therefore we have a composite of pullback squares
		\[\begin{tikzcd}
			G & {G/Stab_p(x)} & X & \ast \\
			\ast & {\BB Stab_p(x)} & {X/G} & {\BB G}
			\arrow[from=1-1, to=1-2]
			\arrow[two heads, from=1-1, to=2-1]
			\arrow["\lrcorner"{anchor=center, pos=0.125}, draw=none, from=1-1, to=2-2]
			\arrow[from=1-2, to=1-3]
			\arrow[two heads, from=1-2, to=2-2]
			\arrow["\lrcorner"{anchor=center, pos=0.125}, draw=none, from=1-2, to=2-3]
			\arrow[from=1-3, to=1-4]
			\arrow[two heads, from=1-3, to=2-3]
			\arrow["\lrcorner"{anchor=center, pos=0.125}, draw=none, from=1-3, to=2-4]
			\arrow[two heads, from=1-4, to=2-4]
			\arrow[from=2-1, to=2-2]
			\arrow[from=2-2, to=2-3]
			\arrow["{\BB\phi}"', curve={height=12pt}, from=2-2, to=2-4]
			\arrow[from=2-3, to=2-4]
		\end{tikzcd}\]
		in which each vertical morphism is an effective epimorphism because $\BB G$ is connected and effective epimorphisms are stable under pullback. Similarly, because $\ast\to\BB Stab_p(x)$ is an effective epimorphism so is $G\to G/Stab_p(x)$. Finally, we have already seen that $\BB Stab_p(x)\to X/G$ is $(-1)$-truncated and therefore so is $G/Stab_p(x)\to X$ by \cite[Remark 5.5.6.12]{htt}. By uniqueness of factorizations (see \cite[Proposition 5.2.8.17]{htt}) we must have that $G\to G/Stab_p(x)\to X$ is the $0$-connective/$(-1)$-truncated factorization of $G\to X$. In other words, $G/Stab_p(x)$ is the $0$-image of $G\to X$ in the $( \Conn_0,\Trunc_0)$ factorization system on $\XX$. So there is an equivalence making the following diagram commute which is unique up to a contractible space of choices.
		\[\begin{tikzcd}
			& {Orb_p(x)} \\
			G && X \\
			& {G/Stab_p(x)}
			\arrow[from=1-2, to=2-3]
			\arrow[from=2-1, to=1-2]
			\arrow[from=2-1, to=3-2]
			\arrow["\simeq"', from=3-2, to=1-2]
			\arrow[from=3-2, to=2-3]
		\end{tikzcd}\]
	\end{proof}

	\section{Normal Maps of Group Objects}\label{sec:normal maps}

	Going forward, we will write $G$ for: a group object $G\in\Grp(\XX)$; the underlying pointed object of that group object $s_0^\ast G\in\XX_\ast$; the underlying object $G_1\in\XX$. Which category $G$ is living in will be clear from context.

	Suppose $G$ is a group object in $\Set$ and $H\leq G$ is a subgroup. By composing with the translation action of $G$ on itself, $G$ inherits a right $H$-action and we may form the set of cosets, or orbits, $G/H$. The notion of \textit{normality} is an answer to the question: when does $G/H$ admit a group structure compatible with the quotient map (of sets) $G\to G/H$? That question can be answered in the affirmative precisely when $H$ is a normal subgroup. There are many equivalent definitions of ``normal,'' some of which make sense in the setting of higher algebra and some of which do not.

	Our definition of normality below, essentially the same as that of Prasma \cite{prasmahmtpynorm}, \textit{defines} normality of a ``subgroup'' $H$ in $G$ to be the property that the quotient of $G$ by the induced right $H$-action admits a group structure. However, there are immediately two significant differences between our case and the classical case.

	First, notice that ``being a subgroup'' is not a well-behaved concept in homotopy theory. For instance, we can clearly take $S^1$ to be a subgroup of $\mathbb{C}$ in classical group theory, but up to homotopy the inclusion becomes the terminal map $S^1\to \ast$. Therefore, again following Prasma (who in turn is following Farjoun and Segev \cite{farjounSegev-homotopynormal}), we consider normality for arbitrary morphisms of group objects $H\to G$.

	Secondly, we will soon see that if $H\to G$ is a morphism of group objects in an $\infty$-topos (indeed, simply take $\XX=\Spaces$), then there may be more than one group structure on $G/H$ so that the quotient map $G\to G/H$ is a group morphism. This is in contrast to the discrete case, where the group structure on $G/H$, if it exists, will be unique. To deal with this added complexity, we make a change which is common in higher algebra and replace the \textit{property} of being normal with \textit{structure} or \textit{data}. This leads us to the following definition of \textit{normality data} for a morphism of group objects.

	\begin{defn}\label{defn:normality data}
		Let $\XX$ be an $\infty$-topos and let $\phi\colon H\to G$ be a morphism in $\Grp(\XX)$. We say a morphism $\pi\colon\BB G\to X$ in $\XX_\ast^{\geq 1}$ is a \textit{normality datum for $\phi$} if the fiber of $\pi$ is equivalent to $\BB\phi\colon \BB H\to \BB G$. If $\phi$ admits at least one normality datum, we will say that $\phi$ is a \textit{normal map of group objects}.
	\end{defn}

	\begin{rmk}
		The idea behind Definition \ref{defn:normality data} is that the object $X$ should be thought of as $\BB (G/H)$, which determines a group structure on $G/H$ via the equivalence $G/H\simeq \Omega\BB (G/H)$. However, as there may be more than one compatible group structure on $G/H$ (or none at all), the notation $\BB G/H$ does not a priori make sense! Therefore we simply require the existence of some object $X$ and a map $\BB G \to X$ whose fiber is $\BB\phi$. Another way to think about it is that the map $\phi$ induces a fiber sequence \[H\to G\to G/H\to \BB H\to \BB G\] as in Example \ref{exam:puppe sequence example} and we are asking for an extension of this sequence one step to the right.
	\end{rmk}

	\begin{exam}
		Let $F\to E\to B$ be a fiber sequence of connected pointed objects in an $\infty$-topos. Then $\Omega F\to \Omega E$ is normal with normality datum $E\to B$ i.e.~the group structure on the quotient object $\Omega E/\Omega F$ is precisely that of $\Omega B$.
	\end{exam}

	\begin{exam}\label{exam:double loop spaces}
		Normality data for group morphisms $G\to \ast$ are precisely double loop space structures on $G$. First, suppose that $G$ has a double loop space structure (note that it is common to say ``$G$ \textit{is} a double loop space'' but of course the same space can have multiple double deloopings). Then it admits a double delooping $\BB^2 G$ which allows us to produce the P\"uppe sequence:
		\[G\to \ast\to\BB G\to\BB G\to \ast\to\BB^2G\] Therefore the inclusion of the canonical basepoint $\ast\simeq\BB\ast\to\BB^2 G$ is a normality datum for $G\to \ast$. We, as usual, think of $\BB G$ as $\ast/G$, with respect to the $\EE_1$-structure of $G$. Note that the double delooping (as well as, of course, the single delooping) \textit{depends} on the $\EE_2$-structure and that different $\EE_2$-structures will give different double deloopings, hence different normality data.

		Now suppose there is a normality datum $\ast\simeq\BB \ast\to X$ for the terminal group morphisms $G\to \ast$. In other words, we have a fiber sequence \[\Omega^2 X\to\ast\to \Omega X\to \Omega X \ast\to X\] in which $\Omega X$ is identified with $\BB G$. Hence $G$ obtains a double loop space structure from the given normality datum.
	\end{exam}

	The following proposition makes it clear that, indeed, a normality datum induces a group structure on the quotient $G/H$.

	\begin{prop}
		Let $\phi\colon H\to G$ be a morphism of $\Grp(\XX)$ with a normality datum $\pi\colon \BB G\to X$. Then the underlying object in $\XX$ of ~$\cech(X)$ is equivalent to the quotient $G/H$ of Definition \ref{defn: puppe sequence defn}.
	\end{prop}

	\begin{proof}
		Recall from Example \ref{exam:puppe sequence example} that the right action of $H$ on $G$ induced by $\phi$ is obtained by taking the translation action of $G$ on itself, corresponding to the pointing $\ast\to\BB G$, and pulling it back along $\BB\phi\colon \BB H\to \BB G$. In other words, $G/H$, which is the geometric realization of the right $H$-action on $G$ induced by $\phi$, is equivalent to the fiber of $\BB\phi$. Because there is a fiber sequence $\BB H\xrightarrow{\BB\phi}\BB G\xrightarrow{\pi} X$ we have that $G/H\simeq \Omega X$. On the other hand, $\cech(X)$ is obtained, by definition, by taking the \v Cech nerve of the pointing $\ast\to X$. The arguments in the proof of Corollary \ref{cor:omega gives loop space action} make it clear that the underlying object of $\cech(X)$, in the sense of Definition \ref{defn:group forgetful functor}, is equivalent to $\Omega X$.
	\end{proof}

	\begin{rmk}
		There is potentially some ambiguity in Definition \ref{defn:normality data}, in that we do not specify in which category we are forming the fiber of $\pi$. In general, the fiber in $\XX_\ast^{\geq 1}$ of a morphism of pointed connected objects need not be equivalent to the fiber taken in $\XX_\ast$. The former is obtained from the latter by deleting everything which is not in the same ``connected component'' of the base point (see Corollary \ref{cor: limits computed by including then BC}). However, by demanding that the fiber of $\pi$ be equivalent to $\BB H$, which is necessarily connected, we have that the two fibers are in agreement.
	\end{rmk}

	We reiterate that a morphism of group objects can be normal in more than one way. This is illustrated by the following two examples, the first of which is due to Kiran Luecke. 	See also Example \ref{example:third iso counterexample S1 actions} for another case in which the same group map, in this case the trivial map $\mathbb{Z}\to\Omega S^2$, can have more than one normality datum.

	\begin{exam}\label{exam: kirans EM space nonunique normality data example}
		In this example we show that a group map (indeed, a map of \textit{commutative} $\infty$-groups) can potentially admit more than one normality datum (i.e.~being normal is not a property for $\infty$-groups, as it is for discrete groups). We consider the product $\infty$-group of Eilenberg-MacLane spaces $K(\ZZ,1)\times K(\ZZ,3)$ and the terminal map to the trivial group $K(\ZZ,1)\times K(\ZZ,3)\to\ast$. We will show that there are two homotopically distinct fiber sequences of group maps \[K(\ZZ,1)\times K(\ZZ,3)\to\ast\to K~~~\text{and}~~~K(\ZZ,1)\times K(\ZZ,3)\to\ast\to J\]
		or, equivalently, two homotopically distinct fiber sequences of $\infty$-groupoids
		\[K(\ZZ,2)\times K(\ZZ,4)\to\ast\xrightarrow{\pi} \BB K~~~\text{and}~~~K(\ZZ,2)\times K(\ZZ,4)\to\ast\xrightarrow{\pi'} \BB J\]
		In the notation of Definition \ref{defn:normality data}, we are taking $H=K(\ZZ,1)\times K(\ZZ,3)$ and $\phi$ the trivial map to $G=\ast$. We will now produce two distinct normality data $\pi$ and $\pi'$.

		To this end, let $f\colon K(\ZZ,2)\to K(\ZZ,4)$ generate the degree 4 integral cohomology of $K(\ZZ,2)$. Let $F$ be the fiber of $f$ and note that because $f$ is not null it cannot be the case that $F$ is equivalent to  $K(\ZZ,4)\times K(\ZZ,2)$, i.e.~$f$ does not classify the trivial bundle. Taking loops of this fiber sequence gives a sequence $\Omega F\to K(\ZZ,1)\xrightarrow{\Omega f} K(\ZZ,3)$, but $\Omega f$ must be null because $K(\ZZ,1)\simeq S^1$ has no integral third cohomology. Thus we obtain $\Omega F\simeq K(\ZZ,1)\times K(\ZZ,3)$. Therefore the fiber sequence $\Omega F\to \ast\to F$ can be rewritten as $K(\ZZ,1)\times K(\ZZ,3)\to\ast\to F$. On the other hand, we trivially have a fiber sequence $K(\ZZ,1)\times K(\ZZ,3)\to\ast\to K(\ZZ,2)\times K(\ZZ,4)$. But we've already seen that $F$ is not equivalent to $K(\ZZ,2)\times K(\ZZ,4)$, so we have two inequivalent normality data for the map $K(\ZZ,1)\times K(\ZZ,3)\to\ast$. In other words, we have produced two inequivalent $\infty$-group structures on the space \[\ast/(K(\ZZ,1)\times K(\ZZ,3))\simeq \BB(K(\ZZ,1)\times K(\ZZ,3))\simeq K(\ZZ,2)\times K(\ZZ,4)\] of orbits of the trivial action of $K(\ZZ,1)\times K(\ZZ,3)$ on the point.
	\end{exam}

	\begin{exam}
		The preceding example is of course just a special case of the phenomenon in which some $G\in\Grp(\XX)$ admits two distinct $\EE_2$-structures lifting its $\EE_1$-structure. Each of these $\EE_2$-structures will produce distinct deloopings, say $BG$ and $BG'$. Then $\ast\to BG$ and $\ast\to BG'$ are distinct normality data for the same map, namely $G\to \ast$.
	\end{exam}

	\begin{rmk}\label{rmk:the wrong def of normality}
		Suppose that $\phi\colon H\to G$ is a morphism of classical groups. We may take the cokernel in the category of groups to obtain another group morphism $G\to \mathrm{cok}\phi$. When $\phi$ is injective, $G$ is normal if and only if the kernel of $G\to \mathrm{cok}\phi$ is isomorphic to $H$ (if $\phi$ is not injective, the question doesn't make sense, classically). In general however, the kernel of this projection is the \textit{normal closure} of $\im(\phi)\subseteq G$ and the cokernel $\mathrm{cok}(\phi)$ is isomorphic to the quotient of $G$ by this normal closure.  Now consider the category $\Grp(\XX)$ of group objects in an $\infty$-topos $\XX$. The notion of ``injectivity'' in this setting is less clear and one might guess that a map of group objects $\phi\colon H\to G$ should be ``normal'' whenever the dotted arrow below is an equivalence (where the cofiber and the fiber are both taken in $\Grp(\XX)$):
		\[\begin{tikzcd}
			H \\
			& {\fib(\cof(\phi))} & G \\
			& \ast & \cof(\phi)
			\arrow[dashed, from=1-1, to=2-2]
			\arrow[curve={height=-12pt}, from=1-1, to=2-3,"\phi"]
			\arrow[curve={height=12pt}, from=1-1, to=3-2]
			\arrow[from=2-2, to=2-3]
			\arrow[from=2-2, to=3-2]
			\arrow["\lrcorner"{anchor=center, pos=0.125}, draw=none, from=2-2, to=3-3]
			\arrow["{\cof(\phi)}", from=2-3, to=3-3]
			\arrow[from=3-2, to=3-3]
		\end{tikzcd}\]
		In other words, one could argue that the map $\phi\colon H\to G$ should be called ``normal'' if $H$ is the fiber of the cofiber of $\phi$. We will see in Section \ref{sec:normal closure monadic} that $\fib(\cof(\phi))$ does behave something like a normal closure of $H$ and $\cof(\phi)$ is something like the quotient by it. However, in the setting of an $\infty$-topos, where fibers and cofibers are inherently \textit{homotopical}, taking this as our definition of normality is too strong. Indeed, it doesn't even recover classical normality for injective maps of discrete groups.

		To see this, consider the multiplication by $2$ map $\ZZ\xrightarrow{2}\ZZ$, which is an inclusion of a normal subgroup. This deloops to the index $2$ map $S^1\to S^1$ and we can complete to a normality datum by appending $S^1\hookrightarrow\mathbb{R}P^\infty\simeq B\mathbb{Z}/2$. This gives the P\"uppe sequence \[\ZZ\xrightarrow{2}\ZZ\to\ZZ/2\to S^1\xrightarrow{2}S^1\to\mathbb{R}P^\infty\]  It follows that, in terms of Definition \ref{defn:normality data}, $S^1\to\mathbb{R}P^\infty$ is a normality datum for $\ZZ\xrightarrow{2}\ZZ$.

		Now, recall Lemma \ref{lem:koszul duality diagram for actions} which implies the ``looping-delooping'' equivalence between $\Grp(\Spaces)$ and $\Spaces_\ast^{\geq 1}$. We may therefore compute the cofiber in $\Grp(\Spaces)$ of $\mathbb{Z}\xrightarrow{2}\ZZ$ by taking the cofiber (in $\Spaces_\ast^{\geq 1}$) of $S^1\xrightarrow{2} S^1$ and then taking loops. That cofiber is, of course, equivalent to $\mathbb{R}P^2$ and so the cofiber of the multiplication by 2 map in $\Grp(\Spaces)$ is $\Omega\mathbb{R}P^2$ (note that this \textit{does} give the ``right thing'' after applying $\pi_0$ since $\pi_0(\Omega\mathbb{R}P^2)\cong\pi_1(\mathbb{R}P^2)\cong \ZZ/2$).  The preceding diagram is now of the form
				\[\begin{tikzcd}
			\ZZ \\
			& {\fib(\ZZ\to\Omega\mathbb{R}P^2)} & \ZZ \\
			& \ast & \Omega\mathbb{R}P^2
			\arrow[dashed, from=1-1, to=2-2]
			\arrow[curve={height=-12pt}, from=1-1, to=2-3,"2"]
			\arrow[curve={height=12pt}, from=1-1, to=3-2]
			\arrow[from=2-2, to=2-3]
			\arrow[from=2-2, to=3-2]
			\arrow["\lrcorner"{anchor=center, pos=0.125}, draw=none, from=2-2, to=3-3]
			\arrow["{\cof(\phi)}", from=2-3, to=3-3]
			\arrow[from=3-2, to=3-3]
		\end{tikzcd}\]
		It is clear from the long exact sequence in homotopy groups that the map $\mathbb{Z}\to\fib(\ZZ\to\Omega\mathbb{R}P^2)$ cannot be an equivalence. For instance, its codomain has infinite cyclic fundamental group while its domain has trivial fundamental group. Therefore, were we to use this alternative definition, the inclusion of $2\ZZ$ into $\ZZ$ would not qualify as a normal subgroup. It is also possible to deduce that, given this definition of normality, Noether's First Isomorphism Theorem becomes false, though we leave this as an exercise to the interested reader.
	\end{rmk}

	\subsection{A Category of Normality Data}

	Classically, e.g.~in Galois theory, one often considers the lattice of normal subgroups of a discrete group $G$. This lattice is closely related to the set of quotient groups of $G$ and one passes from the latter to the former by taking kernels of quotient maps. We do something similar here, but replace the lattice of normal subgroups with a category comprising all normality data associated to a group object $G$. First we characterize normality data in general.

\begin{lem}\label{lem:normality data are 1connective}
	A morphism of pointed connected spaces $\pi\colon \BB G\to X$ is a normality datum for some morphism with codomain $G$ in $\Grp(\XX)$ if and only if it is $1$-connective.
\end{lem}

\begin{proof}
		Suppose that $\BB G\to X$  is a normality datum for some morphism $\phi\colon H\to G$ in $\Grp(\XX)$. Then there is a fiber sequence $\BB H\xrightarrow{\BB\phi} \BB G\to X$. This corresponds to a pullback square:
	\[\begin{tikzcd}
		{\BB H} & {\BB G} \\
		\ast & X
		\arrow["\BB\phi", from=1-1, to=1-2]
		\arrow[from=1-1, to=2-1]
		\arrow["\lrcorner"{anchor=center, pos=0.125}, draw=none, from=1-1, to=2-2]
		\arrow["\pi", from=1-2, to=2-2]
		\arrow[from=2-1, to=2-2]
	\end{tikzcd}\]
	Because $X$ is connected, the pointing $\ast\to X$ is an effective epimorphism. Because $\BB H$ is connected, i.e.~$1$-connective, the terminal map $\BB H$ is $1$-connective. Then from \cite[Proposition 6.5.1.16 (6)]{htt} we have that $\pi$ must also be $1$-connective.

	Now suppose that $\pi\colon \BB G\to X$ is a $1$-connective morphism of pointed connected obects and let $F\to \BB G$ be its fiber. Then again we have a pullback square
		\[\begin{tikzcd}
		{F} & {\BB G} \\
		\ast & X
		\arrow[ from=1-1, to=1-2]
		\arrow[from=1-1, to=2-1]
		\arrow["\lrcorner"{anchor=center, pos=0.125}, draw=none, from=1-1, to=2-2]
		\arrow["\pi", from=1-2, to=2-2]
		\arrow[from=2-1, to=2-2]
	\end{tikzcd}\]
	and it follows from loc.~cit.~that $F$ is connected, hence $F\simeq B\Omega F$ and $\pi$ is a normality datum for $\Omega F\to G$.
\end{proof}

%
%

	\begin{rmk}
		In the case that $\XX=\Spaces$, Lemma \ref{lem:normality data are 1connective} implies that the normality datum $\BB G\to X$ is a surjection on $\pi_1$. It follows from this, e.g.~by the long exact sequence in homotopy, that the associated quotient map $G\to \Omega X$ (where we think of $\Omega X$ as $G/H$ for some $H\to G$) is a surjection on $\pi_0$. This should not be surprising as one would expect ``quotient maps'' to be effective epimorphisms in general. This is generalized, and made more precise, by Proposition \ref{prop:normality data are effective epis from G}.
	\end{rmk}

	Lemma \ref{lem:normality data are 1connective} implies that the set of ``normal subgroups'' of $G$, i.e.~the set of normal maps to $G\in\Grp(\XX)$ is in bijection with the set of $1$-connective morphisms out of $\BB G$ in $\XX^{\geq 1}_\ast$. This justifies the following definition.

	\begin{defn}
		Let $\XX$ be an $\infty$-topos and $G\in\Grp(\XX)$.
		\begin{enumerate}
			\item Define the $\infty$-category $\Nrml(G)$ to be the full subcategory of the slice category $(\XX_\ast^{\geq 1})^{\backslash \BB G}$ spanned by the $1$-connective morphisms.
			\item Define the \textit{underlying group} functor $U_G\colon\Nrml(G)\to \Grp(\XX)_{/G}$ to be the composite of $\fib\colon \Nrml(G)\to (\XX_\ast^{\geq 1})_{/\BB G}$ with the equivalence $\cech_{/\BB G}\colon (\XX_\ast^{\geq 1})_{/\BB G}\to\Grp(\XX)_{/G}$.
			\item If $\phi\colon H\to G$ is a morphism of $\Grp(\XX)$, define the $\infty$-category of normality data for $\phi$, denoted $\Nrml_\phi(G)$, to be the following pullback:
			\[\begin{tikzcd}
				{\Nrml_\phi(G)} & {\Nrml(G)} \\
				{\{\phi\}} & {\Grp(\XX)_{/G}}
				\arrow[from=1-1, to=1-2]
				\arrow[from=1-1, to=2-1]
				\arrow["\lrcorner"{anchor=center, pos=0.125}, draw=none, from=1-1, to=2-2]
				\arrow[from=1-2, to=2-2, "U_G"]
				\arrow[from=2-1, to=2-2]
			\end{tikzcd}\]
		\end{enumerate}
	\end{defn}
%

	\begin{prop}\label{prop:normality data are effective epis from G}
		The fully faithful inclusion $\Nrml(G)\hookrightarrow (\XX_\ast^{\geq 1})^{\backslash \BB G}$ followed by $\cech^{\backslash \BB G}\colon (\XX_\ast^{\geq 1})^{\backslash\BB G}\to\Grp(\XX)^{\backslash G}$, which is an equivalence, has essential image the full subcategory spanned by group morphisms $G\to Q$ which are effective epimorphisms on underlying objects of $\XX$.
	\end{prop}

	\begin{proof}
		It suffices to show that a group morphism $\phi\colon G\to Q$ is an effective epimorphism if and only if $\BB G\to\BB Q$ is $1$-connective. Given that, as a morphism in $\XX$, $\phi\simeq \Omega\BB\phi$, this follows immediately from Proposition \ref{prop:n effective group maps are n connective}.
	\end{proof}

	This characterization of normality allows one to produce examples more readily.

	\begin{exam}
		Let $\phi\colon R\to S$ be a morphism of $\EE_1$-ring spectra, as in \cite[Section 7.1]{ha}, with fiber $\fib(\phi)\to R$ (which is a morphism of $\EE_1$-rings as well).  Suppose that $\pi_0(\phi)\colon\pi_0(R)\to\pi_0(S)$ is surjective and $\pi_0(R)$ is left Artinian. Applying the units functor $\mathrm{GL}_1\colon \Alg_{\EE_1}(\Spectra)\to \Grp(\Spaces)$, which is a right adjoint, gives a fiber sequence in $\Grp(\Spaces)$ \[\mathrm{GL}_1(\fib(\phi))\to\mathrm{GL}_1(R)\to\mathrm{GL}_1(S)\] in which the second map is an effective epimorphism by \cite[Lemma 3.4]{bartel-lenstra-isogenies}. It follows that the first map has a normality datum given by $\mathrm{BGL}_1(R)\to\mathrm{BGL}_1(S)$. It would be interesting to know some non-trivial normal maps to groups of units like $\mathrm{GL}_1(\sph)$.
	\end{exam}

	\begin{exam}
		For any object $G\in\Grp(\XX)$, the truncation morphism (recall that truncation in $\Grp(\XX)$ agrees with truncation in $\XX$ by \cite[Proposition 2.24]{beardsperoux}) $G\to \tau_{\leq n}G$ is evidently an effective epimorphism. It follows that its fiber, the $n^{th}$ object of the Whitehead tower, is a normal map with quotient the truncation: \[\tau^{>n}G\to G\to\tau_{\leq n}G \]
	\end{exam}

	\begin{cor}
		Let $\phi\colon H\to G$ be a morphism of $\Grp(\XX)$. Then the category of normality data for $\phi$ is equivalent to the pullback of the cospan 
		\[\begin{tikzcd}
			{\{\phi\}} & {\Grp(\XX)_{/G}} & {\Grp(\XX)^{\backslash G}_{epi}}
			\arrow[from=1-1, to=1-2]
			\arrow["\fib", from=1-3, to=1-2]
		\end{tikzcd}\]
		where $\Grp(\XX)^{\backslash G}_{epi}$ denotes the full subcategory of $\Grp(\XX)^{\backslash G}$ spanned by the epimorphisms.

	\end{cor}

	\begin{cor}\label{cor:every group is quotient of a free group}
		Every group object $G\in\Grp(\XX)$ is a quotient of a free group object.
	\end{cor}

	\begin{proof}
		Recall the free group object functor $\cech\Sigma\colon \XX_\ast\to \Grp(\XX)$ of Lemma \ref{lem:free group is cech sigma} which is left adjoint to the forgetful functor $s_0^\ast\colon \Grp(\XX)\to \XX_\ast$. This adjunction necessarily has a counit $\epsilon_G\colon \cech\Sigma s_0^\ast G\to G$. The triangle identities of the adjunction give a commutative triangle
		\[\begin{tikzcd}
			{s_0^\ast G} && {s_0^\ast G} \\
			& {s_0^\ast\cech\Sigma s_0^\ast G}
			\arrow["{id_{s_0^\ast G}}", from=1-1, to=1-3]
			\arrow["{\eta_{s_0^\ast G}}"', from=1-1, to=2-2]
			\arrow["{s_0^\ast(\epsilon_G)}"', from=2-2, to=1-3]
		\end{tikzcd}\]
		By \cite[Proposition 6.2.3.7 (1)]{htt} and \cite[Corollary 6.2.3.12 (2)]{htt} it must be the case that $s_0^\ast(\epsilon_G)$ is also an effective epimorphism. Hence $\BB\epsilon_G\colon \BB\cech\Sigma s_0^\ast G\to\BB G$ is a normality datum, i.e.~$G\simeq \cech\Sigma s_0^\ast G/\fib(\epsilon_G)$.
	\end{proof}

	We will need Lemma \ref{lem:NSub is closed under sifted colimits} in the next section but first we prove a preliminary technical result. We provide the proof for completeness, but it will likely be unhelpful to the casual reader.

	\begin{lem}\label{lem:sifted constant functors}
		Let $K$ be a sifted $\infty$-category and suppose that $X_K\colon K\to \XX$ is a constant functor for some $X\in\XX$. Then $\colim(X_K)\simeq X$.
	\end{lem}

	\begin{proof}
		From (the proof of) \cite[Tag 02QL]{kerodon}, we have that the terminal morphism $K\to\Delta^0$ is right cofinal, in the sense of \cite[Tag 02NQ]{kerodon}. It follows that pulling back along $K\to \Delta^0$ preserves colimits. Every constant functor out of $K$ is obtained by pulling back in this way, so the colimit of $X_K$ is equivalent to the colimit of $X\colon\Delta^0\to \XX$, which of course is $X$ itself.
	\end{proof}

	\begin{lem}\label{lem:connected is closed under sifted colimits} The $\infty$-category of connected objects in an $\infty$-topos $\XX$ is closed under sifted colimits.
	\end{lem}

	\begin{proof}
		Suppose that $K$ is a sifted $\infty$-category and $p\colon K\to \XX^{\geq 1}$ is a diagram. By \cite[Definition 6.5.1.11]{htt} it suffices to show that $\tau_{\leq 0}\colim(p)\simeq\ast$. Because $\tau_{\leq 0}$ is a left adjoint, we have that $\tau_{\leq 0}\colim(p)\simeq\colim(\tau_{\leq 0}\circ p)$. Since $\tau_{\leq 0}\circ p(\sigma)\simeq\ast$ for all $\sigma\in K$, the result follows from Lemma \ref{lem:sifted constant functors}.
	\end{proof}

	\begin{rmk}\label{rmk:colimits in undercats}
		It may help the reader to recall the way in which colimits are computed in undercategories. Suppose we have some cocomplete ($\infty$-)category $\CC$, an object $c\in \CC$ and a diagram $F\colon K\to \CC^{\backslash c}$. In general, the colimit of such a diagram is obtained by taking its colimit in $\CC$, but \textit{including} the cone point $c$. Because this diagram has $c$ as an initial object, every possible composition through the diagram from $c$ to the colimit is equivalent. Hence we obtain a unique morphism from $c$ to this colimit, making it an object of $\CC^{\backslash c}$. This morphism is the colimit of $F$ in $\CC^{\backslash c}$.

		 For instance, in pointed spaces $\Spaces_\ast=\Spaces^{\backslash \ast}$ the coproduct of $\ast\to X$ and $\ast\to Y$ is the colimit of the diagram $X\leftarrow \ast\to Y$, i.e.~the wedge sum $X\vee Y$. The pointing of the wedge sum is of course the point at which $X$ and $Y$ became attached by their respective base points.  Note, however, that the colimit in this case is \textit{not} computed as the colimit in $\CC$ (which would be $X\coprod Y$).

		However, if the indexing category $K$ is \textit{weakly contractible} (for instance, sifted), we can indeed compute the colimit of $F\colon K\to \CC^{\backslash c}$ in $\CC$. This is formally encoded by the statement that the forgetful functor $U\colon \CC^{\backslash c}\to \CC$ \textit{creates weakly contractible colimits}. In other words, if $K$ is weakly contractible, the forgetful functor preserves colimits but also has the property that if $UF\colon K\to \CC$ has a colimit then so does $F$ (which, by preservation, must be computed in $\CC$). This follows immediately from \cite[Proposition 4.4.2.9]{htt} and \cite[Corollary 2.1.2.2]{htt}. The first citation gives that left fibrations create weakly contractible colimits and the second implies that the forgetful functor $\CC^{\backslash c}\to \CC$ is a left fibration.
	\end{rmk}

	\begin{lem}\label{lem:NSub is closed under sifted colimits}
		The category $\Nrml(G)$ is closed under sifted colimits.
	\end{lem}
	\begin{proof}

		Let $K$ be a sifted $\infty$-category and let $F\colon K\to \Nrml(G)$ a functor. Because $K$ is sifted (and therefore weakly contractible) the discussion of \ref{rmk:colimits in undercats} implies that we may compute the colimit of $F$ in $\XX^{\geq 1}_\ast$ (and therefore $\XX_\ast$) after forgetting the cone point $\BB G$. Write $X$ for this colimit in $\XX_\ast^{\geq 1}$. It remains to show that the unique lift of this colimit back to $(\XX^{\geq 1}_\ast)^{\backslash \BB G}$, which is a morphism $\zeta\colon \BB G\to X$, is $1$-connective. By Lemma \ref{lem:connective fiber connective map}, it suffices to show that the fiber in $\XX_\ast$, $\fib(\zeta)$, is connected. This fiber is computed as the pullback of the following cospan in $\XX$ (note that we must be careful to take the fiber in $\XX$ or $\XX_\ast$ rather than $\XX_\ast^{\geq 1}$ as the fiber in the latter category will \textit{always} be connected):
		\[\begin{tikzcd}
			 & \BB G \\
			\ast & X
			\arrow[ from=1-2, to=2-2]
			\arrow[ from=2-1, to=2-2]
		\end{tikzcd}\]
		But by using Lemma \ref{lem:sifted constant functors} we have that this cospan is the colimit of natural transformations
		\[\begin{tikzcd}
			& \BB G_K \\
			\ast_K & UF
			\arrow[ from=1-2, to=2-2, Rightarrow]
			\arrow[ from=2-1, to=2-2, Rightarrow]
		\end{tikzcd}\]
		in which the top right and bottom left are constant functors and $UF$ is the composition of $F$ with the forgetful functor $\Nrml(\XX)\to\XX$. We are also using that the fact that $K$ is sifted again to ensure that this forgetful functor preserves the relevant colimits.

		Note that for each $\sigma\in K$ we have that the map $\zeta_\sigma\colon\BB G\to UF(\sigma)$ is $1$-connective, so each fiber $\fib(\zeta_\sigma)$ is connected. Then we use \cite[Lemma 5.5.6.17]{ha} to see that taking the colimit over $K$ commutes with taking fibers so that $\fib(\zeta)\simeq\colim_K(\fib(\zeta_\sigma))$. Finally, we use Lemma \ref{lem:connected is closed under sifted colimits} to deduce that $\fib(\zeta)$ must itself be connected, completing the proof.

	\end{proof}

	\subsection{Normality Data is Monadic}\label{sec:normal closure monadic}

	Recall that in classical group theory one can take the ``normal closure'' of a subgroup $H\leq G$. Explicitly, this is the smallest normal subgroup of $G$ containing $H$. If we write $ncl_G(H)$ for the normal closure of $H$ in $G$, then we have the following standard fact:

	\begin{prop}
		Let $H\leq G$ be a subgroup. Then the pushout of the cospan $0 \leftarrow H\to G$ in the category of discrete groups is isomorphic to $G/ncl_G(H)$.
	\end{prop}

	In the case of discrete groups, the above proposition allows one to characterize normal subgroups $H\trianglelefteq G$ as those with the property that the canonical morphism $H\to\ker(G\to G/ncl_G(H))$ is an isomorphism. We have already seen in Remark \ref{rmk:the wrong def of normality} that this definition is not well-behaved in higher group theory. Nonetheless, there is still a useful notion of normal closure for morphisms of group objects which produces the ``free normality datum'' generated by a morphism $H\to G$ in $\Grp(\XX)$.

	\begin{lem}\label{lem:cof is 1-connective}
		If~~$f\colon X\to Y$ is a morphism in $\XX_\ast^{\geq 1}$ then $\cof(f)\colon Y\to\cof(f)$ is $1$-connective in $\XX$. In particular, the cofiber functor $\cof\colon (\XX_\ast^{\geq 1})_{/\BB G}\to (\XX_\ast^{\geq 1})^{\backslash \BB G}$ factors through $\Nrml(G)$.
	\end{lem}

	\begin{proof}
		Recall that, by definition, the terminal morphism $X\to \ast$ is $1$-connective since $X$ is connected. By \cite[Corollary 6.5.1.17]{htt} we know that $1$-connective morphisms are preserved by pushout, so the result follows from considering the defining pushout diagram of the cofiber:
		\[\begin{tikzcd}
			X & Y \\
			\ast & {\cof(f)}
			\arrow["f", from=1-1, to=1-2]
			\arrow[from=1-1, to=2-1]
			\arrow[from=1-2, to=2-2]
			\arrow[from=2-1, to=2-2]
			\arrow["\lrcorner"{anchor=center, pos=0.125, rotate=180}, draw=none, from=2-2, to=1-1]
		\end{tikzcd}\]
	\end{proof}

	\begin{defn}
		Let $G\in\Grp(\XX)$. Then we define the \textit{normal closure in $G$} functor to be the following composite:
		\[\ncl_G\colon\Grp(\XX)_{/G}\xrightarrow{\simeq} (\XX_\ast^{\geq 1})_{/\BB G}\xrightarrow{\cof}\Nrml(G)\]
		By applying the underlying group functor $U_G\colon\Nrml(G)\to \Grp(\XX)_{/G}$ we obtain a morphism of group objects $\cech \fib(\cof(\BB\phi))\simeq\fib(\cof(\phi))\to G$.
	\end{defn}

	\begin{exam}\label{exam:free double loop space}
		Taking $G=\ast$, we see that the normal closure of the terminal group map $H\to \ast$ is the terminal morphism $\cech\Omega\Sigma\BB H\to \ast$ or, on underlying objects, $\Omega^2\Sigma\BB H\to\ast$. This is consistent with the fact that normal morphisms to $\ast$ in $\Grp(\Spaces)$ are the same as double loop spaces (see Example \ref{exam:double loop spaces}) and that $\Omega^2\Sigma\BB$ is the ``free double loop space on a loop space'' functor (not the be confused with the ``free double loop space on a pointed space'' functor which is given by $\Omega^2\Sigma^2$).
	\end{exam}

	\begin{lem}\label{lem:fibercofiberadjunction}
		Given a pointed $\infty$-category $\CC$ that admits pushouts and pullbacks, and an object $X\in \CC$, there is an adjunction 	\[\begin{tikzcd}
			{\CC_{/X}} && {\CC^{\backslash X}}
			\arrow[""{name=0, anchor=center, inner sep=0}, "\cof", curve={height=-6pt}, from=1-1, to=1-3]
			\arrow[""{name=1, anchor=center, inner sep=0}, "\fib", curve={height=-6pt}, from=1-3, to=1-1]
			\arrow["\dashv"{anchor=center, rotate=-90}, draw=none, from=0, to=1]
		\end{tikzcd}\] in which the left adjoint is given by taking the cofiber of a morphism $Y\to X$ and the right adjoint is given by taking the fiber of a morphism $X\to Y$.
	\end{lem}

	\begin{proof}
		Write $\mathrm{csp}$ for the cospan category $1\to 0\leftarrow 2$, $\mathrm{sp}$ for the span category $1\leftarrow 0\to 2$, and $\mathrm{sq}$ for the commutative square category \[\begin{tikzcd} 0\ar[r] \ar[d]& 1\ar[d]\\ 2\ar[r] &3 \end{tikzcd}\] the fully faithful inclusions $\mathrm{sp}\hookrightarrow\mathrm{sq}$ and $\mathrm{csp}\hookrightarrow\mathrm{sq}$ induce a composable pair of adjunctions of diagram categories in which all functors commute with the restriction functor to $\{1,2\}$:
		\[\begin{tikzcd}
			{\CC^{\mathrm{sp}}} && {\CC^{\mathrm{sq}}} && {\CC^{\mathrm{csp}}} \\
			\\
			&& {\CC\times\CC}
			\arrow[""{name=0, anchor=center, inner sep=0}, "Lan", curve={height=-6pt}, hook, from=1-1, to=1-3]
			\arrow["{\{1,2\}^\ast}"{description}, from=1-1, to=3-3]
			\arrow[""{name=1, anchor=center, inner sep=0}, "res", curve={height=-6pt}, from=1-3, to=1-1]
			\arrow[""{name=2, anchor=center, inner sep=0}, "res", curve={height=-6pt}, from=1-3, to=1-5]
			\arrow["{\{1,2\}^\ast}"{description}, from=1-3, to=3-3]
			\arrow[""{name=3, anchor=center, inner sep=0}, "Ran", curve={height=-6pt}, hook', from=1-5, to=1-3]
			\arrow["{\{1,2\}^\ast}"{description}, from=1-5, to=3-3]
			\arrow["\dashv"{anchor=center, rotate=-90}, draw=none, from=0, to=1]
			\arrow["\dashv"{anchor=center, rotate=-90}, draw=none, from=2, to=3]
		\end{tikzcd}\] Note that the left Kan extension above is precisely the pushout, and the right Kan extension is the pullback.
		We will show that each of these adjunctions is a \textit{relative adjunction}, relative to $\CC\times \CC$, in the sense of \cite[Definition 7.3.2.2]{ha}. To do so, we must first check that the restriction maps $\{1,2\}^\ast$ are categorical fibrations. In general, however, restrictions along inclusions are always categorical fibrations (for instance by \cite[Proposition 41.3]{RezkIntroQuasicategories}, taking $X=\CC$, $Y=\ast$ and $K\to L$ to be any of the inclusions of $\{1,2\}$). It remains to check that the unit of both adjunctions restricts to the identity transformation on $\CC\times \CC$, but clearly all functors involved are the identity at the corners. Hence we can apply \cite[Propostion 7.3.2.5]{ha} to pull back these adjunction along the inclusion $\{\ast,X\}\hookrightarrow\CC\times \CC$. Composing the adjunctions, we obtain a single adjunction
		\[\begin{tikzcd}
			\CC^{\mathrm{csp}}_{\ast,X} && {\CC^{\mathrm{sp}}_{\ast,X}}
			\arrow[""{name=0, anchor=center, inner sep=0}, curve={height=-6pt}, from=1-1, to=1-3]
			\arrow[""{name=1, anchor=center, inner sep=0}, curve={height=-6pt}, from=1-3, to=1-1]
			\arrow["\dashv"{anchor=center, rotate=-90}, draw=none, from=0, to=1]
		\end{tikzcd}\]
		The subscript on each category indicats that the left adjoint only takes pushouts of spans of the form $X\leftarrow Y\to \ast$ and the right adjoint only takes pullbacks of cospans of the form $X\to Y\leftarrow \ast$. Now compose on each side with the obvious equivalences ${\CC^{\mathrm{csp}}_{\ast,X}}\simeq \CC^{\backslash X}$ and ${\CC^{\mathrm{sp}}_{\ast,X}}\simeq \CC_{/X}$ (using that $\ast$ is a zero object in $\CC$).
	\end{proof}

	\begin{thm}\label{thm:normal closure functor}
		There is a monadic adjunction between $\Grp(\XX)_{/G}$ and $\Nrml(G)$ whose left adjoint is $\ncl_G$ and whose right is the underlying group functor $U_G$.
	\end{thm}

	\begin{proof}
		First we show the existence of the adjunction. We begin with the adjunction 
		\[\begin{tikzcd}
			{(\XX_\ast)_{/\BB G}} && {(\XX_\ast)^{\backslash \BB G}}
			\arrow[""{name=0, anchor=center, inner sep=0}, "\cof", curve={height=-6pt}, from=1-1, to=1-3]
			\arrow[""{name=1, anchor=center, inner sep=0}, "\fib", curve={height=-6pt}, from=1-3, to=1-1]
			\arrow["\dashv"{anchor=center, rotate=-90}, draw=none, from=0, to=1]
		\end{tikzcd}\] described in Lemma \ref{lem:fibercofiberadjunction}. Restricting to $(\XX_\ast^{\geq 1})_{/\BB G} $ on the left, we have by Lemma \ref{lem:cof is 1-connective} that the functor $\cof$ lands in $\Nrml(G)$. On the other hand, given $\BB G\to X$ in $\Nrml(G)\subseteq (\XX_\ast)^{\backslash \BB G}$, which must be $1$-connective by Lemma \ref{lem:normality data are 1connective}, we have that its fiber is $1$-connective (hence connected) by \cite[Proposition 6.5.1.16(6)]{htt}. Hence we have a pair of functors $\cof\colon (\XX_\ast^{\geq 1})_{/\BB G}\to \Nrml(G)$ and $\fib\colon \Nrml(G)\to (\XX_\ast^{\geq 1})_{/\BB G}$.  Using the definition of adjunctions given in \cite[Definition 2.1.1]{rvelements} it is not hard to check that the unit and counit transformations of the first adjunction descend to unit and counit transformations for these restricted functors and still satisfy the triangle identities (since we have restricted to full subcategories). Therefore we have another adjunction 
		\[\begin{tikzcd}
			{(\XX_\ast^{\geq 1})_{/\BB G}} && {\Nrml(G)}
			\arrow[""{name=0, anchor=center, inner sep=0}, "\cof", curve={height=-6pt}, from=1-1, to=1-3]
			\arrow[""{name=1, anchor=center, inner sep=0}, "\fib", curve={height=-6pt}, from=1-3, to=1-1]
			\arrow["\dashv"{anchor=center, rotate=-90}, draw=none, from=0, to=1]
		\end{tikzcd}\]

		We now check that this adjunction is monadic via the Barr-Beck Theorem of \cite[Theorem 4.7.3.5]{ha}. Checking conservativity requires showing that if $\fib(\alpha)$ in the following diagram
		\[\begin{tikzcd}
			{\fib(\phi)} & BG & X \\
			\\
			{\fib(\psi)} & BG & Y
			\arrow[from=1-1, to=1-2]
			\arrow["{\fib(\alpha)}"', from=1-1, to=3-1]
			\arrow["\phi", from=1-2, to=1-3]
			\arrow[shift right, no head, from=1-2, to=3-2]
			\arrow[no head, from=1-2, to=3-2]
			\arrow["\alpha"', from=1-3, to=3-3]
			\arrow[from=3-1, to=3-2]
			\arrow["\psi"', from=3-2, to=3-3]
		\end{tikzcd}\]
		is an equivalence then $\alpha$ is. First, note that, by assumption, $X$ and $Y$ are both connected. Now extend the columns of the above diagram by one step:
		\[\begin{tikzcd}
			{\Omega X} & {\fib(\phi)} & BG & X \\
			\\
			{\Omega Y} & {\fib(\psi)} & BG & Y
			\arrow[from=1-1, to=1-2]
			\arrow["{\Omega\alpha}"', from=1-1, to=3-1]
			\arrow[from=1-2, to=1-3]
			\arrow["{\fib(\alpha)}"', from=1-2, to=3-2]
			\arrow["\phi", from=1-3, to=1-4]
			\arrow[shift right, no head, from=1-3, to=3-3]
			\arrow[no head, from=1-3, to=3-3]
			\arrow["\alpha"', from=1-4, to=3-4]
			\arrow[from=3-1, to=3-2]
			\arrow[from=3-2, to=3-3]
			\arrow["\psi"', from=3-3, to=3-4]
		\end{tikzcd}\]
		Since $\fib(\alpha)$ is an equivalence, so is $\Omega\alpha$. But since $X$ and $Y$ are connected and pointed and $\alpha$ is a pointed map, $\Omega\alpha$ is also the image of $\alpha$ under the equivalence of $\infty$-categories $\check{C}\colon \XX_{\ast}^{\geq 1}\to \Grp(\XX)$. Since equivalences are always conservative functors, $\alpha$ must be an equivalence (of pointed connected objects) in $\XX$. Therefore the functor $\fib$ is conservative.

		Suppose we are given a simplicial object of $\Nrml(G)$. By Lemma \ref{lem:NSub is closed under sifted colimits} we know that this diagram admits a colimit so it only remains to show that $\fib$ preserves this colimit, but this was already shown in the proof of Lemma \ref{lem:NSub is closed under sifted colimits} (specifically, the citation of \cite[Lemma 5.5.6.17]{ha}, using the fact that $\Delta^{\op}$ is sifted). The result follows from composing with the equivalence $\Grp(\XX)_{/G}\simeq (\XX_\ast^{\geq 1})_{/\BB G}$.
	\end{proof}

	Another way to state Theorem \ref{thm:normal closure functor} is to say that normal maps to $G$ are the algebras for the normal closure monad. By taking monadic resolutions this gives an alternative proof of Corollary \ref{cor:every group is quotient of a free group}. Noting that normality of the terminal morphism $G\to \ast$ in $\Grp(\XX)$ corresponds to a grouplike $\EE_2$-algebra structure on $G$ (see Example \ref{exam:double loop spaces}), this also gives a description of grouplike $\EE_2$-algebras in $\XX$ as algebras for the $\cech\Omega\Sigma\BB$ monad on $\Grp(\XX)$.

	\begin{rmk}\label{rmk:normal closure is not idempotent}
		Note that the normal closure of a normal map $\phi\colon H\to G$ is \textit{not} generally equivalent to $\phi$. This corresponds to the fact that taking normal closure corresponds to freely adding normality data which necessarily makes $\phi$ ``bigger.'' Alternatively, note that because normality data is not generally unique, any normal closure functor cannot ``know'' which normality datum to append to a given map, whether or not that map is normal. As such, it is forced to add new data.
	\end{rmk}

	\subsection{Normality is Preserved by Monoidal Functors}

	In \cite[Theorem B]{prasmahmtpynorm}, Prasma shows that a homotopy monoidal endofunctor of the category of topological spaces preserves normal maps of loop spaces. In this section we will give a slight generalization of his result by showing that any monoidal functor of $\infty$-topoi preserves normal maps of group objects. Note that monoidal maps do not preserve \textit{quotients}. To clarify, let $\phi\colon H\to G$ be normal in $\Grp(\XX)$ and let $F\colon \XX\to \YY$ be a monoidal functor of $\infty$-topoi. We will show that $F(\phi)\colon F(H)\to F(G)$ remains a normal morphism of group objects. The fact that $F$ is monoidal and $G/H$ has a group structure imply that $F(G/H)$ has a group structure as well. However, what is \textit{not} generally true is that if $F$ is monoidal then the group $F(G/H)$ is equivalent to the group $F(G)/F(H)$ (where the latter has a group structure coming from $F$ preserving normality data). This is true almost by definition of course if $F$ preserves pullbacks or geometric realizations, but we wish to recover Prasma's more general theorem.

	\begin{lem}\label{lem:normal action object lifts to groups}
		Suppose that a morphism $\phi\colon H\to G$ in $\Grp(\XX)$ admits a normality datum $\pi\colon \BB G\to\BB (G/H)$. Then $\BB_\bullet(G,H)\colon\Delta^{\op}\to\XX$, the bar construction for the $H$-action on $G$ induced by $\phi$, factors through $\Grp(\XX)$.
	\end{lem}

	\begin{proof}
		By assumption, there is a fiber sequence in $\XX^{\geq 1}_\ast$: \[\BB H\xrightarrow{\BB \phi} \BB G\xrightarrow{\pi}\BB (G/H)\] Now take the \v Cech nerve of $\pi$, followed by $\cech$ (applied levelwise) to get a simplicial object of $\Grp(\XX)$. Noticing that $\cech\colon\XX_\ast^{\geq 1}\to \Grp(\XX)$ and $U_{\Grp}\colon \Grp(\XX)\to\XX$ are both right adjoints which preserve right Kan extensions we have that $U_{\Grp}\circ \cech$ applied to the \v Cech nerve of $\pi$ is the \v Cech nerve of $G\simeq \Omega\BB G\to \Omega \BB (G/H)\simeq G/H$. From this we have that the \v Cech nerve of $G\to G/H$ lifts to a simplicial group object.

		Now from the discussion following \cite[Corollary 6.2.3.5]{htt} we recall that groupoid resolutions of effective epimorphisms are unique. We know by Lemma \ref{lem:action groupoid is a groupoid} that $\BB_\bullet(G,H)$ is a groupoid resolution of $G\to G/H$ and therefore is equivalent, as a simplicial object of $\XX$, to the \v Cech nerve of $G\to G/H$. It follows that it factors through $\Grp(\XX)$.
	\end{proof}

	\begin{thm}
		If $\phi\colon H\to G$ is a normal map in $\Grp(\XX)$ and $F\colon \XX\to \YY$ is a monoidal functor of $\infty$-topoi then $F(\phi)\colon F(H)\to F(G)$ is normal in $\Grp(\YY)$.
	\end{thm}

	\begin{proof}
		From Corollary \ref{cor:actF is Fact}, we have that $F$ applied to the action object of $H$ on $G$ (induced by $\phi$) is the action object of $F(H)$ on $F(G)$ (induced by $F(\phi)$). By Lemma \ref{lem:normal action object lifts to groups}, we have that $\BB_\bullet(H,G)\in\Grp(\XX)^{\Delta^{\op}}$ and therefore $\BB_{\bullet} (FG,FH)\in\Grp(\YY)^{\Delta^{\op}}$. Note also that Lemma \ref{lem:action groupoid is a groupoid} implies that $\BB_\bullet(FG,FH)$ is a groupoid resolution, by definition, of $F(G)/F(H)$.

		Now take the colimit of $\BB_\bullet(FG,FH)$ in $\Grp(\YY)$. The forgetful functor $U_{\Grp}\colon\Grp(\YY)\to\YY$ preserves sifted colimits (as shown, e.g., in the proof of \cite[Corollary 5.2.6.18]{ha}) so there is an effective epimorphism of group objects $F(G)\to F(G)/F(H)$. It follows from Proposition \ref{prop:normality data are effective epis from G} that $F(\phi)\colon F(H)\to F(G)$ is normal, with normality datum $\BB F(G)\to \BB (F(G)/F(H))$.
	\end{proof}

	Since a map of group objects $H\to G$ being normal is the same as the existence of a fiber sequence $\BB H\to \BB G\to X$ in $\XX$, we also have that monoidal functors of $\infty$-topoi preserve the property of \textit{being a fiber}, at least for connected objects. This is similar to \cite[Theorem C]{prasmahmtpynorm} which recovers a result of Dwyer and Farjoun \cite[Section 3]{dwyer-farjoun-localization} regarding the preservation of principal fibrations by localization functors.

	\begin{cor}
		If~~$F\colon \XX\to \YY$ is a monoidal functor of $\infty$-topoi and $X\to Y\to Z$ is a fiber sequence of connected objects in $\XX$ then $FX\to FY$ is the fiber of a morphism $FY\to W$ in $\YY$.
	\end{cor}

	\section{Noether's Isomorphism Theorems in $\infty$-topoi}\label{sec:Noether Iso Theorems}

	In this section we show that Noether's First and Second Isomorphism Theorems hold for group objects in $\infty$-topoi and that her Third Isomorphism Theorem does not hold (indeed, the third can barely be stated).

	\subsection{The First Isomorphism Theorem}

	Noether's First Isomorphism Theorem for discrete groups can be stated as follows:
	\begin{thm*}
		Let $\phi\colon H\to G$ be a morphism of discrete groups. Then:
		\begin{enumerate}
			\item The image $im(\phi)\subseteq G$ is a subgroup of $G$.
			\item The kernel $\ker(\phi)\subseteq H$ is a normal subgroup of $H$.
			\item There is an isomorphism of groups $H/\ker(\phi)\cong im(\phi)$.
		\end{enumerate}
	\end{thm*}

	Our goal is to formulate and prove this statement for group objects of $\infty$-topoi and normality data as described in Section \ref{sec:normal maps}. Because we have replaced the property of \textit{being normal} with \textit{normality data}, this requires some care. Let us consider the preceding theorem point by point.

	For the first, recall from Section \ref{sec:monos and epis of groups} that we have a notion of ``image'' in any $\infty$-topos that lifts to $\Grp(\XX)$. In fact, any morphism has a so-called $n$-image for any $n\geq 0$. We work with the $0$-image here since it most naturally recovers the discrete notion after truncation. Intuitively, the $0$-image of a morphism $X\to Y$ in an $\infty$-topos $\XX$ is the ``set of connected components'' in the image of $\pi_0(X)\to\pi_0(Y)$. We may therefore phrase, and prove, the first point of the First Isomorphism Theorem as follows:

	\begin{prop}\label{prop:factorization part of first iso thm}
	Let $\phi\colon H\to G$ be a morphism in $\Grp(\XX)$ and write $U\colon\Grp(\XX)\to\XX$ for the forgetful functor. Then the factorization in $\XX$ \[UH\to\im_0(U\phi)\to UG\] lifts to a factorization in $\Grp(\XX)$ such that the latter morphism is a monomorphism.
	\end{prop}

	\begin{proof}
		This is simply a restatement of the results of Section \ref{sec:monos and epis of groups}. Therein, a factorization system is constructed on $\Grp(\XX)$ by which every morphism can be factored into an epimorphism followed by a monomorphism. Propositions \ref{prop:n effective group maps are n connective} and \ref{prop:n truncated group maps are n truncated} give that this factorization system is identical to the underlying factorization system on $\XX$.
	\end{proof}

	Note that the above argument relies on a \textit{choice} of what ``image'' even means in this setting. However, if we accept the generally agreed-upon notion of what a ``monomorphism'' should be in an $\infty$-topos then the content of the proposition is: the standard epi/mono factorization system in $\XX$ lifts to $\Grp(\XX)$.

	For the second and third points of the First Isomorphism Theorem, it will be helpful to briefly review what ``normal'' and ``quotient'' mean in the setting of group objects in $\infty$-topoi. The essential desideratum of a ``normal'' subgroup $N\to H$ is that taking a quotient by it, $H/N$, produces another group and a quotient group homomorphism $H\to H/N$. Our approach, and that of Prasma, Segev and Farjoun \cite{prasmahmtpynorm, farjounSegev-homotopynormal}, is to use this desideratum to \textit{define} normality. We say that a normal map of group objects $N\to H$ is precisely one that appears as the fiber of a group map $H\to G$ for some $G$ (i.e.~normal maps are exactly the ``kernels'' of group maps). Therefore, one way to describe a ``normality datum'' for a morphism $\iota\colon N\to H$ is as the data of a group map $H\to G$ whose fiber is $\iota$. Delooping leads to our equivalent definition (which tends to be more convenient) of a normality datum for $\iota\colon N\to H$ as a morphism of pointed connected objects $\pi\colon \BB H\to X$ for some $X\in\XX^{\geq 1}_\ast$. This leads to a P\"uppe sequence \[\cdots\to N\xrightarrow{\iota} H\to \Omega X\to \BB N\xrightarrow{\BB\iota}\BB H\to X\] in which $H\to \Omega X$ becomes the group morphism of which $\iota$ is the fiber. As we show in Section \ref{sec:normal maps}, such a datum for $\iota\colon N\to H$ is by no means unique and therefore, for any group map, there is a (possibly empty) moduli space of ways to realize it as a normal map.

	Note that the leftmost fiber sequence above $N\xrightarrow{\iota} H\to \Omega X$ is a fiber sequence but is of course not a cofiber sequence. We would, however, like to think of $\Omega X$ as a quotient of $H$ by the $N$-action on it induced by $\iota$. This was made precise in Section \ref{sec:base change}, where it was shown that whenever there is a fiber sequence of group objects $H\to G\to K$ (i.e.~a composite of group maps whose underlying composite in $\XX$ is a fiber sequence) the base object $K$ is indeed equivalent to the orbits of the $H$ action on $H$ induced by $H\to G$.

	Now we will see that the second and third points of the above First Isomorphism Theorem become combined into one statement in our setting. The second point of the theorem indicates that when we have a morphism of group objects $\phi\colon H\to G$, we must find a normality datum for $\fib(\phi)\to H$, i.e.~a morphism $\pi\colon\BB H\to X$ so that $(\fib(\pi)\to\BB H)\simeq \BB(\fib(\phi)\to H)$. However, note that a choice of normality datum is equivalent to a choice of quotient $H/\fib(\phi)$. Indeed, it is precisely a group structure on (the underlying object of) the orbits object $H/\fib(\phi)$. Therefore, requiring that $H/\fib(\phi)\simeq \im_0(\phi)$ forces the object $X$ in the above normality datum to be equivalent to $\BB\im_0(\phi)$, the homotopy type of which is predetermined by the group structure of $\im_0(\phi)$. Hence the only thing to prove is that there is a morphism $\BB H\to \BB\im_0(\phi)$ whose fiber is $\BB\fib(\phi)\to\BB H$. Note, however, that we already \textit{have} such a morphism, namely the delooping of the epimorphism $H\to \im_0(\phi)$ arising in Proposition \ref{prop:factorization part of first iso thm}. This is the content of the following theorem.

	\begin{thm}\label{thm:first iso theorem}
		Let $\phi\colon H\to G$ be a morphism in $\Grp(\XX)$. Then $\BB H\to\BB\im_0(\phi)$ is a normality datum for $\fib(\phi)\to H$.
	\end{thm}

	\begin{proof}
		By construction the morphism $H\to\im_0(\phi)$ is an effective epimorphism in $\Grp(\XX)$, so by Proposition \ref{prop:normality data are effective epis from G} it only remains to show that $\fib(H\to\im_0(\phi))\simeq\fib(\phi)$.

		By Proposition \ref{prop:n truncated group maps are n truncated} we know that $\im_0(\phi)\to G$ is $(-1)$-truncated.  It follows from \cite[Lemma 5.5.6.15]{htt} that the diagonal $\im_0(\phi)\to\im_0(\phi)\times_{G}\im_0(\phi)$ is $(-2)$-truncated, i.e.~an equivalence. Now consider the commutative diagram
		\[\begin{tikzcd}
			{\im_0(\phi)} \\
			& {\im_0(\phi)\times_G\im_0(\phi)} & {\im_0(\phi)} \\
			& {\im_0(\phi)} & G
			\arrow["\simeq"{description}, from=1-1, to=2-2]
			\arrow["\simeq"{description}, curve={height=-12pt}, from=1-1, to=2-3]
			\arrow["\simeq"{description}, curve={height=18pt}, from=1-1, to=3-2]
			\arrow[from=2-2, to=2-3]
			\arrow[from=2-2, to=3-2]
			\arrow["\lrcorner"{anchor=center, pos=0.125}, draw=none, from=2-2, to=3-3]
			\arrow[from=2-3, to=3-3]
			\arrow[from=3-2, to=3-3]
		\end{tikzcd}\] in which every morphism is $(-1)$-truncated and the indicated morphisms are equivalences. It follows that both projections $\im_0(\phi)\times_G\im_0(\phi)\to\im_0(\phi)$ are equivalences. Because the horizontal fibers of the above pullback square must be equivalent we have that $\fib(\im_0(\phi)\to G)\simeq\ast$. Therefore we have the composite pullback square
		\[\begin{tikzcd}
			{\fib(\phi)} & \ast & \ast \\
			H & {\im_0(\phi)} & G
			\arrow[from=1-1, to=1-2]
			\arrow[from=1-1, to=2-1]
			\arrow["\lrcorner"{anchor=center, pos=0.125}, draw=none, from=1-1, to=2-2]
			\arrow[from=1-2, to=1-3]
			\arrow[from=1-2, to=2-2]
			\arrow["\lrcorner"{anchor=center, pos=0.125}, draw=none, from=1-2, to=2-3]
			\arrow[from=1-3, to=2-3]
			\arrow[from=2-1, to=2-2]
			\arrow[from=2-2, to=2-3]
		\end{tikzcd}\]
		implying that $\fib(\phi)\simeq\fib(H\to\im_0(\phi))$.
	\end{proof}

	\begin{rmk}
		It is worth noting that in the case $\XX=\Spaces$, the fact that $\im_0(\phi)\to G$ is $(-1)$-truncated implies that it is an injection on $\pi_0$. Therefore the above theorem recovers the classical First Isomorphism Theorem after taking $\pi_0$.
	\end{rmk}

	\subsection{The Second Isomorphism Theorem}

	Noether's Second Isomorphism Theorem for discrete groups can be stated as follows:

	\begin{thm*}
		Let $H, K \leq G$ be subgroups with $K \trianglelefteq G$ normal in $G$. Then,
		\begin{enumerate}
			\item $H \cap K \trianglelefteq H$ is normal;
			\item $K \trianglelefteq HK$ is normal;
			\item There exists an isomorphism of groups
			\[
			H/(H \cap K) \cong HK/K.
			\]
		\end{enumerate}
	\end{thm*}
To prove this statement for group objects in an $\infty$-topos, we first need to make sense of the notions of \textit{intersection} and \textit{product} for morphisms $K\to G$ and $H\to G$ in $\Grp(\XX)$. The notion of intersection is simple to describe and can be formulated in $\Grp(\XX)$ as follows:
	\begin{defn}
		Let $\phi\colon H\to G$ and $\psi\colon K\to G$ be morphisms in $\Grp(\XX)$. Then we define the \textit{intersection} of $\phi$ and $\psi$, denoted $\phi\cap\psi\colon H\cap K\to G$ to be the following pullback:
			\[\begin{tikzcd}
				{H\cap K} & K \\
				H & G
				\arrow[from=1-1, to=1-2]
				\arrow[from=1-1, to=2-1]
				\arrow["\lrcorner"{anchor=center, pos=0.125}, draw=none, from=1-1, to=2-2]
				\arrow["\psi", from=1-2, to=2-2]
				\arrow["\phi"', from=2-1, to=2-2]
			\end{tikzcd}\]
		\end{defn}

		\begin{defn}
			The \textit{product} of $\phi$ and $\psi$, which we will denote by $\phi\psi\colon HK\hookrightarrow G$, is the $0$-image of the dotted universal map to $G$ in the following pushout diagram.
			\[\begin{tikzcd}
				{H\cap K} & K \\
				H & H\coprod\limits_{H\cap K}K \\
				&& G
				\arrow[from=1-1, to=1-2]
				\arrow[from=1-1, to=2-1]
				\arrow[from=1-2, to=2-2]
				\arrow[curve={height=-6pt}, from=1-2, to=3-3]
				\arrow[from=2-1, to=2-2]
				\arrow[curve={height=6pt}, from=2-1, to=3-3]
				\arrow["\lrcorner"{anchor=center, pos=0.125, rotate=180}, draw=none, from=2-2, to=1-1]
				\arrow[ dashed,from=2-2, to=3-3]
			\end{tikzcd}\]
			
			It arises naturally in the epi/mono factorization $H\coprod_{H\cap K}K\twoheadrightarrow HK\hookrightarrow G$.
	\end{defn}
	
	\begin{rmk}
	In the case that $G$ is a discrete group, $K$ and $H$ are subgroups and $K$ is normal, our definitions recover the classical intersection and product.
	\end{rmk}

	Now we check the first and second statements of the Second Isomorphism Theorem.
	
	\begin{lem}
		Suppose $\phi\colon H\to G$ and $\psi\colon K\to G$  are morphisms of group objects and $\pi\colon \BB G\to X$ is a normality datum for $\psi$. Then there is a canonical normality datum for the morphism $H\cap K\to H$ in the defining pullback square 			\[\begin{tikzcd}
			{H\cap K} & K \\
			H & G
			\arrow[from=1-1, to=1-2]
			\arrow[from=1-1, to=2-1]
			\arrow["\lrcorner"{anchor=center, pos=0.125}, draw=none, from=1-1, to=2-2]
			\arrow["\psi", from=1-2, to=2-2]
			\arrow["\phi"', from=2-1, to=2-2]
		\end{tikzcd}\]
	\end{lem}
	
	\begin{proof}
		The existence of the normality datum $\pi\colon \BB G\to X$ implies that there is a pullback square
					\[\begin{tikzcd}
			K & G \\
			\ast & \Omega X
			\arrow[from=1-1, to=1-2, "\psi"]
			\arrow[from=1-1, to=2-1]
			\arrow["\lrcorner"{anchor=center, pos=0.125}, draw=none, from=1-1, to=2-2]
			\arrow["\Omega\pi", from=1-2, to=2-2]
			\arrow[from=2-1, to=2-2]
		\end{tikzcd}\]
		where we are thinking of $\Omega X$ as the group object $G/K$. We now paste the pullback square in the lemma statement to obtain the double pullback 
				\[\begin{tikzcd}
			H\cap K & K & \ast \\
			H & G & \Omega X
			\arrow[from=1-1, to=1-2]
			\arrow[from=1-1, to=2-1]
			\arrow["\lrcorner"{anchor=center, pos=0.125}, draw=none, from=1-1, to=2-2]
			\arrow[from=1-2, to=1-3]
			\arrow[from=1-2, to=2-2, "\psi", swap]
			\arrow["\lrcorner"{anchor=center, pos=0.125}, draw=none, from=1-2, to=2-3]
			\arrow[from=1-3, to=2-3]
			\arrow[from=2-1, to=2-2, "\phi", swap]
			\arrow[from=2-2, to=2-3, "\Omega\pi", swap]
		\end{tikzcd}\]
		Theorem \ref{thm:first iso theorem} then implies that there is a normality datum for the morphism $H\cap K\to H$, namely $\BB  H\to \BB \im_0(\Omega\pi\circ\phi)$.
	\end{proof}
	
		Note that, in the case of classical discrete groups, the image of $H$ in $G/K$ is \textit{exactly} $HK/K$.  Therefore, we could arguably claim to have proven the Second Isomorphism Theorem so long as we \textit{define} $HK/K$ to be $\im_0(H\to\Omega X)$. However, we have a putative definition of $HK$, so for the sake of thoroughness we will show that this is consistent with the above interpretation.
		
	\begin{lem}\label{lem:split monos in topoi}
	Suppose that $f\colon X\hookrightarrow Y$ is a monomorphism in an $\infty$-topos $\XX$ and $\sigma\colon Y\to X$ is a morphism such that $f\circ\sigma\simeq id_Y$. Then $f$ is an equivalence.
	\end{lem}
	
	\begin{proof}
		The result follows from the case of $\XX=\Spaces$. Being a monomorphism, $f$ is $(-1)$-truncated and therefore the map of spaces $f_\ast\colon\XX(Z,X)\xrightarrow{}\XX(Z,Y)$ is $(-1)$-truncated and admits a section $\sigma_\ast\colon \XX(Z,Y)\to \XX(Z,X)$, for any $Z\in\XX$. It suffices to prove that $f_\ast$ is an equivalence. It is already an inclusion of a summand and therefore the result follows from applying $\pi_0$ and noticing that the result also holds for injections of sets.
	\end{proof}

\begin{lem}\label{lem:canonicalnormalityin2ISO}
	Suppose $\phi\colon H\to G$ and $\psi\colon K\to G$  are morphisms of group objects and $\pi\colon \BB G\to X$ is a normality datum for $\psi$. Then there is a canonical normality datum for the morphism $K\to HK$ in the following commutative diagram:
	\[\begin{tikzcd}
		{H\cap K} & K && \\
		H & {H\ast_{H\cap K}K} \\
		&& HK \\
		&&& G
		\arrow[from=1-1, to=1-2]
		\arrow[from=1-1, to=2-1]
		\arrow[from=1-2, to=2-2]
		\arrow[curve={height=-6pt}, from=1-2, to=3-3, "\theta" description]
		\arrow[curve={height=-12pt}, from=1-2, to=4-4, "\psi"]
		\arrow[from=2-1, to=2-2]
		\arrow[curve={height=12pt}, from=2-1, to=4-4, "\phi", swap]
		\arrow["\lrcorner"{anchor=center, pos=0.125, rotate=180}, draw=none, from=2-2, to=1-1]
		\arrow[two heads, from=2-2, to=3-3]
		\arrow[hook, from=3-3, to=4-4, "\iota" description]
	\end{tikzcd}\]
\end{lem}

\begin{proof}
	In the classical case we have that $HK/K$ is the image of $H$ in $G/K=\Omega X$. This suggests that $K$ should be the fiber of a composite map $HK\to G\to \Omega X$, i.e.~that there is a fiber sequence $K\to HK\to\Omega X$. We now show that this is precisely the case. First, define the group object $F$ via the solid composite pullback diagram
		\[\begin{tikzcd}
		F & K & \ast \\
		HK & G & \Omega X
		\arrow[from=1-1, to=1-2, hook]
		\arrow[from=1-1, to=2-1]
		\arrow["\lrcorner"{anchor=center, pos=0.125}, draw=none, from=1-1, to=2-2]
		\arrow[from=1-2, to=1-3]
		\arrow[from=1-2, to=2-2, "\psi", swap]
		\arrow["\lrcorner"{anchor=center, pos=0.125}, draw=none, from=1-2, to=2-3]
		\arrow[from=1-3, to=2-3]
		\arrow[from=2-1, to=2-2, "\iota", hook]
		\arrow[from=2-2, to=2-3, "\Omega\pi"]
	\end{tikzcd}\]
	Now note that $F\to K$ must be a monomorphism (since $HK\to G$ is) and that the universal property of $F$ induces a commutative diagram
	\[\begin{tikzcd}
		K && \\
		& F & K \\
		& HK & G
		\arrow["\sigma"{description}, dashed, from=1-1, to=2-2]
		\arrow["{id_K}", curve={height=-6pt}, from=1-1, to=2-3]
		\arrow["\theta"', curve={height=6pt}, from=1-1, to=3-2]
		\arrow[hook, from=2-2, to=2-3]
		\arrow[from=2-2, to=3-2]
		\arrow["\lrcorner"{anchor=center, pos=0.125}, draw=none, from=2-2, to=3-3]
		\arrow["\psi"', from=2-3, to=3-3]
		\arrow["\iota", hook, from=3-2, to=3-3]
	\end{tikzcd}\]
	in which $\sigma$ is a section of the monomorphism $F\hookrightarrow K$. It then follows from Lemma \ref{lem:split monos in topoi} that $F\hookrightarrow K$ is an equivalence. Hence $K$ is the fiber of $HK\to \Omega X$. Again, Theorem \ref{thm:first iso theorem} gives the normality datum $\BB HK\to \BB\im_0(\Omega\pi\circ\iota)$.
\end{proof}

\begin{prop}
	Under the hypotheses and notation above, there is a canonical equivalence: \[\im_0(\Omega\pi\circ\iota)\simeq \im_0(\Omega\pi\circ\phi)\]
\end{prop}

\begin{proof}
	Start by considering the following commutative extension of the diagram from Lemma \ref{lem:canonicalnormalityin2ISO} (in which more morphisms have been named):
\[\begin{tikzcd}
	& K &&&&& \\
	{H\cap K} && {H\ast_{H\cap K}K} & HK && G & {G/K\simeq\Omega X} \\
	& H &&& {\im_0(\Omega\pi\circ\phi)}
	\arrow[from=1-2, to=2-3]
	\arrow["\psi", curve={height=-12pt}, from=1-2, to=2-6]
	\arrow["{\rho_K}", from=2-1, to=1-2]
	\arrow["{\rho_H}"', from=2-1, to=3-2]
	\arrow["\lrcorner"{anchor=center, pos=0.125, rotate=-135}, draw=none, from=2-3, to=2-1]
	\arrow["\nu", two heads, from=2-3, to=2-4]
	\arrow["\iota", hook, from=2-4, to=2-6]
	\arrow["{\Omega\pi}", from=2-6, to=2-7]
	\arrow[from=3-2, to=2-3]
	\arrow["\phi", curve={height=12pt}, from=3-2, to=2-6]
	\arrow["\ell"', curve={height=6pt}, two heads, from=3-2, to=3-5]
	\arrow["r"', curve={height=6pt}, hook, from=3-5, to=2-7]
\end{tikzcd}\]
Here, $\ell$ and $r$ denote the epimorphism and monomorphism, respectively, in the factorization of $\Omega\pi\circ\phi$. Note that, since $\psi$ is the fiber of $\Omega\pi$, there is an equivalence $\Omega\pi\circ\psi\simeq\ast$. This fact, together with the commutative diagram, implies the following string of equivalences:
\[r\circ\ell\circ\rho_H\simeq \Omega\pi\circ\phi\circ\rho_H\simeq \Omega\pi\circ\psi\circ\rho_K\simeq\ast\] Therefore we have a commutative square 
\[\begin{tikzcd}
	{H\cap K} & {\im_0(\Omega\pi\circ\phi)} \\
	\ast & {\Omega X}
	\arrow["{\ell\circ\rho_H}", from=1-1, to=1-2]
	\arrow[two heads, from=1-1, to=2-1]
	\arrow["r", hook, from=1-2, to=2-2]
	\arrow[from=2-1, to=2-2]
\end{tikzcd}\] in which we notice that the trivial map is always an epimorphism so long as its domain is non-empty. There is therefore a unique factorization of $\ell\circ\rho_H$ through $\ast$, i.e.~it is null. 

The universal property of $H\ast_{H\cap K}K$ therefore induces a morphism $\alpha\colon H\ast_{H\cap K}K\to \im_0(\Omega\pi\circ\phi)$ fitting into the following commutative diagram: 
\[\begin{tikzcd}
	& K &&&& \\
	{H\cap K} && {H\ast_{H\cap K}K} && {\im_0(\Omega\pi\circ\phi)} & {G/K\simeq\Omega X} \\
	& H
	\arrow[from=1-2, to=2-3]
	\arrow["\ast", curve={height=-6pt}, from=1-2, to=2-5]
	\arrow["{\rho_K}", from=2-1, to=1-2]
	\arrow["{\rho_H}"', from=2-1, to=3-2]
	\arrow["\lrcorner"{anchor=center, pos=0.125, rotate=-135}, draw=none, from=2-3, to=2-1]
	\arrow["\alpha"{description}, dashed, from=2-3, to=2-5]
	\arrow["r"', hook, from=2-5, to=2-6]
	\arrow[from=3-2, to=2-3]
	\arrow["\ell"', curve={height=6pt}, two heads, from=3-2, to=2-5]
\end{tikzcd}\]

The morphism $\alpha$ now fits into another commutative square 
\[\begin{tikzcd}
	{H\ast_{H\cap K}K} & {\im_0(\Omega\pi\circ\phi)} \\
	HK & {\Omega X}
	\arrow["\alpha", from=1-1, to=1-2]
	\arrow["\nu"', two heads, from=1-1, to=2-1]
	\arrow[hook, from=1-2, to=2-2]
	\arrow["\beta"{description}, dashed, from=2-1, to=1-2]
	\arrow["{\Omega\pi\circ\iota}"', from=2-1, to=2-2]
\end{tikzcd}\]
in which $\beta$ is determined, up to a contractible space of choices, by orthogonality of the left epimorphism and the right monomorphism. We then restrict to $H$ on the left-hand side, writing $\tilde{\phi}$ for the composite map $H\to H\ast_{H\cap K}K\to HK$:
\[\begin{tikzcd}
	H & {H\ast_{H\cap K}K} & {\im_0(\Omega\pi\circ\phi)} \\
	& HK & {\Omega X}
	\arrow[from=1-1, to=1-2]
	\arrow["\ell", curve={height=-24pt}, from=1-1, to=1-3, two heads]
	\arrow["{\tilde{\phi}}"', from=1-1, to=2-2]
	\arrow["\alpha", from=1-2, to=1-3]
	\arrow["\nu"', two heads, from=1-2, to=2-2]
	\arrow[hook, from=1-3, to=2-3]
	\arrow["\beta"{description}, dashed, from=2-2, to=1-3]
	\arrow["{\Omega\pi\circ\iota}"', from=2-2, to=2-3]
\end{tikzcd}\]

Since $\beta\circ\tilde{\phi}\simeq \ell$ and $\ell$ is an epimorphism, we have that $\beta$ must also be an epimorphism by Corollary \ref{cor:effective epis right cancel in groups}. As such, the composite $HK\xrightarrow{\beta}\im_{0}(\Omega\pi\circ\phi)\xrightarrow{r}\Omega X$ must be the unique epi-mono factorization of $\Omega\pi\circ\iota$. Hence there is an equivalence $\im_0(\Omega\pi\circ\iota)\simeq \im_0(\Omega\pi\circ \phi)$. 
\end{proof}

Putting the above lemmas together then gives Noether's Second Isomorphism Theorem for group objects in $\infty$-topoi. It can be stated as follows.

\begin{thm}\label{thm:secondisotheorem}
Let $\phi\colon H\to G$ and $\psi\colon K\to G$ be morphisms of group objects in an $\infty$-topos $\XX$. Suppose that $\pi\colon \BB G\to X$ is a normality datum for $\psi$. Then,
\begin{enumerate}
	\item There is a canonical normality datum $\BB H\to\BB\im_0(\Omega\pi\circ\phi)$ for $H\cap K$ in $H$.
	\item There is a canonical normality datum $\BB HK\to \BB\im_0(\Omega\pi\circ\iota)$ for $K$ in $HK$.
	\item There is an equivalence, unique up to a contractible space of choices, between the quotients associated to the above normality data: \[H/H\cap K\simeq HK/K\]
\end{enumerate}	
\end{thm}

	\subsection{The Third Isomorphism Theorem}\label{sec:third iso}

	The classical Third Isomorphism Theorem can be stated as follows:

	\begin{thm*}
		Let $G$ be a group with normal subgroups $H$ and $K$ such that $K\subseteq H$. Then the quotient group $(G/K)/(H/K)$ is isomorphic to $G/H$.
	\end{thm*}

	Note, however, that this statement has some implicit assumptions which, while always true for discrete groups, need not be true for group objects in an arbitrary $\infty$-topos. First, if we want to write down the quotient $(G/K)/(H/K)$, we need that $K$ be normal in $H$. Otherwise, we do not have that $H/K$ is necessarily a group at all. Even assuming that we have $K$ normal in $H$, we still need that $H/K$ be a normal subgroup of $G/K$ to make the final statement.

	The first condition above is sometimes phrased as ``normality satisfies the intermediate subgroup condition.'' In other words, if $K\trianglelefteq G$ is a normal subgroup and $H\leq G$ is a subgroup containing $K$, then $K$ is normal in $H$. As soon as we replace normal subgroups with normality data this property fails. The following counterexample was suggested to us by Kiran Luecke.

	\begin{exam}\label{exam:hopf fibration third iso counter}

		Consider the Hopf fibration sequence
		\[S^1\to S^3\to S^2\to \mathbb{C}P^\infty\to BS^3\] where $S^3$ is equipped with the group structure coming from the equivalence $S^3\simeq SU(2)$. It is classical that, with respect to this, $S^1\to S^3$ is a map of group objects in $\Spaces$ (for instance as the inclusion of a certain class of diagonal matrices). Moreover, we know that $S^2$ \textit{does not} admit the structure of a group object in $\Spaces$ (if it did, it would admit an $H$-space structure, which it cannot by \cite{adams-Hspace_spheres}). Now note that we have a commutative triangle in $\Grp(\Spaces)$:
		\[\begin{tikzcd}
			{S^1} & {} & \ast \\
			& {S^3}
			\arrow[from=1-1, to=2-2]
			\arrow[from=1-1, to=1-3]
			\arrow[from=2-2, to=1-3]
		\end{tikzcd}\]
		where $\ast$ is contractible. Because $S^1$ is an infinite loop space, the trivial map $S^1\to \ast$ admits a normality datum $\ast\to K(\ZZ,3)$. However, a normality datum for $S^1\to S^3$ would endow $S^2$ with the structure of a group object, which is impossible.
	\end{exam}
	In fact there are many counterexamples even for \textit{discrete} $\infty$-groups. In other words, the issue is with replacing normal subgroups with normality data, not necessarily replacing groups with loop spaces. The following counterexample is due to David W\"arn.

	\begin{exam}\label{exam:Warn discrete third iso counter}
		Let $\phi\colon A\hookrightarrow G$ be an inclusion of a discrete subgroup such that $A$ is Abelian but not normal in $G$ (for instance, the inclusion of $\ZZ/2$ in the symmetric group $S_3$ as the cyclic transposition). Then we have a commutative triangle: 		\[\begin{tikzcd}
			{A} & {} & \ast \\
			& {G}
			\arrow["\phi"', from=1-1, to=2-2]
			\arrow[from=1-1, to=1-3]
			\arrow[from=2-2, to=1-3]
		\end{tikzcd}\]
		Because $A$ is Abelian, there is always a normality datum for $A\to \ast$ given by $\ast\to\BB^2 A$. Therefore, if the intermediate subgroup condition were satisfied for normality data, we would have that $A\hookrightarrow G$ admits a normality datum. One can check that an inclusion of a subgroup into a group admits a normality datum if and only if it is normal in the classical sense. This gives a host of simple counterexamples to lifting the intermediate subgroup property to normality data.
	\end{exam}


	To better understand this problem, let's first enumerate the data required to even \textit{state} the Third Isomorphism Theorem, as written above. Given a commutative triangle in $\Grp(\XX)$
	\[\begin{tikzcd}
		K && G \\
		& H
		\arrow["\phi", from=1-1, to=1-3]
		\arrow["\rho"', from=1-1, to=2-2]
		\arrow["\psi"', from=2-2, to=1-3]
	\end{tikzcd}\]  which we think of as the sequence of subgroup inclusions $K\leq H\leq G$, let us assume only what the classical theorem must assume; namely, we assume that we have normality data $\pi_\phi\colon BG\to B(G/K)$ and $\pi_{\psi}\colon BG\to B(G/H)$ for $\phi$ and $\psi$ respectively. Classically, this suffices for stating the Third Isomorphism Theorem (due to what is sometimes called the \textit{Fourth Isomorphism Theorem} or \textit{Lattice Theorem} cf.~Propositions 8.9 and 8.10 of \cite{aluffi:chapter0}), but we will need to assume more:
	\begin{enumerate}
		\item We need a normality datum $\pi_{\rho}\colon BH\to B(H/K)$ for $\rho$ so that $H/K$ can be lifted to an object of $\Grp(\XX)$.
		\item Assuming the existence of the above, we also need a morphism $\theta\colon H/K\to G/K$ and a normality datum $B(G/K)\to B((G/K)/(H/K))$ so that $(G/K)/(H/K)$ has a group object structure.
		\item Again, assuming both of the above data, we moreover need a $\infty$-group morphism $\omega\colon G/H\to (G/K)/(H/K)$ to serve as a candidate for the equivalence.
	\end{enumerate}

	In the case that we suppose the existence of all of the above data, the Third Isomorphism Theorem follows from checking that a simple diagram commutes:

	\begin{thm}\label{thm:weakthirdiso}
		Suppose we have a commutative triangle in $\Grp(\XX)$ 
		\[\begin{tikzcd}
			K && G \\
			& H
			\arrow["\phi", from=1-1, to=1-3]
			\arrow["\rho"', from=1-1, to=2-2]
			\arrow["\psi"', from=2-2, to=1-3]
		\end{tikzcd}\] along with the following data:
		\begin{enumerate}
			\item Normality data for $\phi$, $\psi$ and $\rho$:
			\begin{enumerate}
				\item $\pi_\phi\colon \BB G\to \BB (G/K)$,
				\item $\pi_{\psi}\colon \BB G\to \BB (G/H)$,
				\item $\pi_{\rho}\colon \BB H\to \BB (H/K)$.
			\end{enumerate}
			\item A morphism $\theta\colon H/K\to G/K$ in $\Grp(\XX)$.
			\item A normality datum $\pi_{\theta}\colon \BB (G/K)\to \BB ((G/K)/(H/K))$ for $\theta$.
			\item A morphism $\omega\colon G/H\to (G/K)/(H/K)$ in $\Grp(\XX)$.
		\end{enumerate}
		Then the morphism $\omega\colon G/H\xrightarrow{\simeq} (G/K)/(H/K)$ is an equivalence whenever the following diagram commutes:
		\[\begin{tikzcd}
			\BB H & {\BB (H/K)} \\
			\BB G & {\BB (G/K)} \\
			{\BB (G/H)} & {\BB ((G/K)/(H/K))}
			\arrow["{\pi_{\rho}}", from=1-1, to=1-2]
			\arrow["\BB \psi"', from=1-1, to=2-1]
			\arrow["\BB \theta", from=1-2, to=2-2]
			\arrow["{\pi_{\phi}}"', from=2-1, to=2-2]
			\arrow["{\pi_{\psi}}"', from=2-1, to=3-1]
			\arrow["{\pi_{\theta}}", from=2-2, to=3-2]
			\arrow["\BB \omega"', from=3-1, to=3-2]
		\end{tikzcd}\]
	\end{thm}

	\begin{proof}
		Suppose the given diagram commutes. Then we may take fibers along the top two horizontal arrows to obtain a commutative diagram in which the top two rows are fiber sequences:
		\[\begin{tikzcd}
			\BB K & \BB H & {\BB (H/K)} \\
			\BB K & \BB G & {\BB (G/K)} \\
			& {\BB (G/H)} & {\BB ((G/K)/(H/K))}
			\arrow[from=1-1, "\BB\rho", to=1-2]
			\arrow["{id_{\BB K}}"', from=1-1, to=2-1]
			\arrow["{\pi_{\rho}}", from=1-2, to=1-3]
			\arrow["\BB \psi"', from=1-2, to=2-2]
			\arrow["\BB \theta", from=1-3, to=2-3]
			\arrow[from=2-1, "\BB\phi", to=2-2,swap]
			\arrow["{\pi_{\phi}}"', from=2-2, to=2-3]
			\arrow["{\pi_{\psi}}"', from=2-2, to=3-2]
			\arrow["{\pi_{\theta}}", from=2-3, to=3-3]
			\arrow["\BB \omega"', from=3-2, to=3-3]
		\end{tikzcd}\]
		Further taking fibers along the columns gives a commutative diagram in which the top three rows are fiber sequences:
		\[\begin{tikzcd}
			\ast & {G/H} & {(G/K)/(H/K)} \\
			\BB K & \BB H & {\BB (H/K)} \\
			\BB K & \BB G & {\BB (G/K)} \\
			& {\BB (G/H)} & {\BB ((G/K)/(H/K))}
			\arrow[from=1-1, to=1-2]
			\arrow[from=1-1, to=2-1]
			\arrow["\simeq", from=1-2, to=1-3]
			\arrow[from=1-2, to=2-2]
			\arrow[from=1-3, to=2-3]
			\arrow[from=2-1, "\BB\rho", to=2-2]
			\arrow["{id_{\BB K}}"', from=2-1, to=3-1]
			\arrow["{\pi_{\rho}}", from=2-2, to=2-3]
			\arrow["\BB \psi"', from=2-2, to=3-2]
			\arrow["\BB \theta", from=2-3, to=3-3]
			\arrow[from=3-1, "\BB\phi", to=3-2, swap]
			\arrow["{\pi_{\phi}}"', from=3-2, to=3-3]
			\arrow["{\pi_{\psi}}"', from=3-2, to=4-2]
			\arrow["{\pi_{\theta}}", from=3-3, to=4-3]
			\arrow["\BB \omega"', from=4-2, to=4-3]
		\end{tikzcd}\]
		It follows that the upper right morphism is an equivalence. But this morphism is $\Omega\BB\omega=s_0^\ast\cech\BB\omega$ and $s_0^\ast$ is conservative (as shown in the proof of \cite[Corollary 5.2.6.18]{ha}). Therefore $\cech\BB\omega\simeq\omega$ is an equivalence of group objects.
	\end{proof}

	\begin{rmk}
		Note that weaker assumptions would give an equivalence in $\XX$ of the form $G/H\to (G/K)/(H/K)$. Specifically, suppose we had not assumed the existence of the morphism $\omega$ in the above theorem and had only assumed commutativity of the following diagram:
		\[\begin{tikzcd}
			\BB H & {\BB (H/K)} \\
			\BB G & {\BB (G/K)}
			\arrow["{\pi_{\rho}}", from=1-1, to=1-2]
			\arrow["\BB \psi"', from=1-1, to=2-1]
			\arrow["\BB \theta", from=1-2, to=2-2]
			\arrow["{\pi_{\phi}}"', from=2-1, to=2-2]
		\end{tikzcd}\]
		Then essentially the same argument would show that a morphism $G/H\to (G/K)/(H/K)$ existed and was an equivalence of objects in $\XX$. This map does not, however, necessarily lift to a morphism of $\Grp(\XX)$.
	\end{rmk}

	Clearly the statement and ``proof'' of the Third Isomorphism Theorem for $\Grp(\XX)$ requires the input of much more data than the classical version. We do not believe that any of that data comes ``for free'' from the initial assumptions of the classical case, but we do not have counterexamples for every possible configuration. However, the following example shows that even if we have a normal morphism which is a composite of normal morphisms $K\to H\to G$ (i.e.~every morphism in the commutative triangle is normal) it need not be the case that the Third Isomorphism Theorem holds. The following counterexample is due to Kiran Luecke.
	\begin{exam}\label{example:third iso counterexample S1 actions}
		Let $S^2_t$ denote the 2-sphere equipped with the trivial $S^1$-action and let $S^2_r$ denote the 2-sphere with the $S^1$-action that rotates it about an axis through the poles. We have (homotopy) quotients $S^2_t/S^1\simeq S^2\times \mathbb{C}P^\infty$ and $S^2_r/S^1\simeq \mathbb{C}P^\infty\vee \mathbb{C}P^\infty$. These give fiber sequences $S^1\to S^2_t\to S^2\times\mathbb{C}P^\infty$ and $S^1\to S^2_r\to \mathbb{C}P^\infty\vee \mathbb{C}P^\infty$. Therefore, we have a commutative diagram in which each straight composite is a fiber sequence: 
		\[\begin{tikzcd}
			&&& \ast \\
			{S^1} && {S_r^2} & {\mathbb{C}P^\infty\vee\mathbb{C}P^\infty} \\
			& {S^2_t} \\
			&& {S^2\times\mathbb{C}P^\infty}
			\arrow[from=2-1, to=2-3]
			\arrow[from=2-1, to=3-2]
			\arrow[from=2-3, to=1-4]
			\arrow[from=2-3, to=2-4]
			\arrow["id"', from=3-2, to=2-3]
			\arrow[from=3-2, to=4-3]
		\end{tikzcd}\]
		Now, take loops on this diagram to get a commutative triangle in $\Grp(\Spaces)$:
		\[\begin{tikzcd}
			&&& \ast \\
			{\mathbb{Z}} && {\Omega S_r^2} & {\Omega(\mathbb{C}P^\infty\vee\mathbb{C}P^\infty)} \\
			& {\Omega S^2_t} \\
			&& {\Omega S^2\times S^1}
			\arrow[from=2-1, to=2-3]
			\arrow[from=2-1, to=3-2]
			\arrow[from=2-3, to=1-4]
			\arrow[from=2-3, to=2-4]
			\arrow["id"', from=3-2, to=2-3]
			\arrow[from=3-2, to=4-3]
		\end{tikzcd}\]
		Because we have taken loops on fiber sequences, each of the group maps above is normal. For clarity we explain what each piece of the above pair of diagrams is in terms of the terminology of Theorem \ref{thm:weakthirdiso}. We have:
		\begin{itemize}
			\item $K=\ZZ$, $H=\Omega S^2_t$, $G=\Omega S^2_r$.
			\item $G/K=\Omega(\mathbb{C}P^\infty\vee \mathbb{C}P^\infty)$, $G/H=\ast$, $H/K=\Omega S^2\times S^1$.
		\end{itemize}
		For these groups to satisfy the Third Isomorphism Theorem we would first need a group map $\Omega S^2\times S^1\to \Omega(\mathbb{C}P^\infty\vee \mathbb{C}P^\infty)$. This would deloop to a map of spaces $S^2\times \mathbb{C}P^\infty\to \mathbb{C}P^\infty\vee \mathbb{C}P^\infty$, and we would need a normality datum $\pi\colon \mathbb{C}P^\infty\vee \mathbb{C}P^\infty\to X$ whose fiber is $S^2\times\mathbb{C}P^\infty$. Let us suppose such data existed. Then we would have $\Omega X$ as our model of $(G/K)/(H/K)$. But $G/H=\Omega S^2/\Omega S^2$ is contractible, so we would need $\Omega X\simeq \ast$, and thus $X\simeq \ast$. But if $X$ is contractible then the fiber sequence above yields $S^2\times\mathbb{C}P^\infty\simeq\mathbb{C}P^\infty\vee\mathbb{C}P^\infty$, which is not true (for instance, they have non-isomorphic cohomology rings).
	\end{exam}

	\begin{rmk}
		The astute reader may notice the similarity between the above diagrams and the diagram defining the so-called Octahedral Axiom for triangulated categories (see, for instance, \cite[1.1.7.1]{beilinsonbernsteindeligne-perversesheaves}). To make it clear, consider the second diagram flipped upside down:
		\[\begin{tikzcd}
			&& {\Omega S^2\times S^1} && \\
			\\
			& {\Omega S^2_t} && [xshift=-1em] {\Omega(\mathbb{C}P^\infty\vee\mathbb{C}P^\infty)} \\
			&& {\Omega S_r^2} \\
			{\mathbb{Z}} &&&& \ast
			\arrow[dotted, from=1-3, to=3-4]
			\arrow[from=3-2, to=1-3, end anchor=south west]
			\arrow["id", from=3-2, to=4-3]
			\arrow[dotted, from=3-4, to=5-5]
			\arrow[from=4-3, to=3-4]
			\arrow[from=4-3, to=5-5]
			\arrow[from=5-1, to=3-2]
			\arrow[from=5-1, to=4-3]
		\end{tikzcd}\]

		Indeed, it is not uncommon (see \cite[1.1.8.1]{beilinsonbernsteindeligne-perversesheaves}) to regard the Octahedral Axiom as a kind of ``third isomorphism theorem.'' Note, however, that triangulated categories are stable (loops and suspension are inverses). This suggests that some degree of ``stabilization'' might correct this failure. This is the content of Theorem \ref{thm:third iso for En spaces}.
	\end{rmk}
Before proving Theorem \ref{thm:third iso for En spaces}, we will need the following elementary lemma.
	\begin{lem}\label{lem:composite action fiber sequence}
		Let $K\to H\to G$ be a composite of morphisms of $\Grp(\XX)$. Then there is a fiber sequence $H/K\to G/K\to G/H$.
	\end{lem}

	\begin{proof}
		By definition there is a pullback square: 
		\[\begin{tikzcd}
			{G/H} & \ast \\
			\BB H & \BB G
			\arrow[from=1-1, to=1-2]
			\arrow[from=1-1, to=2-1]
			\arrow["\lrcorner"{anchor=center, pos=0.125}, draw=none, from=1-1, to=2-2]
			\arrow[from=1-2, to=2-2]
			\arrow[from=2-1, to=2-2]
		\end{tikzcd}\]
		By pasting, we can extend this to the following composite pullback diagram:
		\[\begin{tikzcd}
			{G/K} & {G/H} & \ast \\
			\BB K & \BB H & \BB G
			\arrow[from=1-1, to=1-2]
			\arrow[from=1-1, to=2-1]
			\arrow["\lrcorner"{anchor=center, pos=0.125}, draw=none, from=1-1, to=2-2]
			\arrow[from=1-2, to=1-3]
			\arrow[from=1-2, to=2-2]
			\arrow["\lrcorner"{anchor=center, pos=0.125}, draw=none, from=1-2, to=2-3]
			\arrow[from=1-3, to=2-3]
			\arrow[from=2-1, to=2-2]
			\arrow[from=2-2, to=2-3]
		\end{tikzcd}\]
		Finally, taking the fiber of the upper left horizontal morphism and using that there is a fiber sequence $H/K\to BK\to BH$, we obtain the desired fiber sequence in the uppermost square of the following diagram:
		\[\begin{tikzcd}
			{H/K} & \ast \\
			{G/K} & {G/H} & \ast \\
			\BB K & \BB H & \BB G
			\arrow[from=1-1, to=1-2]
			\arrow[from=1-1, to=2-1]
			\arrow["\lrcorner"{anchor=center, pos=0.125}, draw=none, from=1-1, to=2-2]
			\arrow[from=1-2, to=2-2]
			\arrow[from=2-1, to=2-2]
			\arrow[from=2-1, to=3-1]
			\arrow["\lrcorner"{anchor=center, pos=0.125}, draw=none, from=2-1, to=3-2]
			\arrow[from=2-2, to=2-3]
			\arrow[from=2-2, to=3-2]
			\arrow["\lrcorner"{anchor=center, pos=0.125}, draw=none, from=2-2, to=3-3]
			\arrow[from=2-3, to=3-3]
			\arrow[from=3-1, to=3-2]
			\arrow[from=3-2, to=3-3]
		\end{tikzcd}\]
	\end{proof}

	\begin{thm}\label{thm:third iso for En spaces}
		Suppose we have a commutative diagram of~~$\EE_n$-group objects for $n>1$ (in which all maps are morphisms of~~$\EE_n$-groups):
		\[\begin{tikzcd}
			K && G \\
			& H
			\arrow["\phi", from=1-1, to=1-3]
			\arrow["\rho"', from=1-1, to=2-2]
			\arrow["\psi"', from=2-2, to=1-3]
		\end{tikzcd}\]
		Then there are canonical choices for all data in the statement of Theorem \ref{thm:weakthirdiso} such that the Third Isomorphism Theorem holds.
	\end{thm}

	\begin{proof}
		Because $K$, $H$ and $G$ are all at least $\EE_2$-algebras, we may form the pullback diagrams of the proof of Lemma \ref{lem:composite action fiber sequence} entirely within $\Alg_{\EE_1}(\XX)$. It follows that there is a fiber sequence of group objects $H/K\xrightarrow{\theta} G/K\xrightarrow{\lambda} G/H$. We take the first morphism of this composite as the $\theta$ required by Theorem \ref{thm:weakthirdiso}. This has normality datum $\pi_\theta=\BB \lambda\colon \BB (G/K)\to \BB (G/H)$, the delooping of the second morphism in the last fiber sequence. It follows that we can take $\omega$ to be the identity morphism $G/H\to G/H$. It only remains to check that the required diagram of Theorem \ref{thm:weakthirdiso} commutes. We prove that the diagram commutes after applying the equivalence $\Omega$, which gives the result. This diagram is as follows:
		\[\begin{tikzcd}
			H & {H/K} \\
			G & {G/K} \\
			{G/H} & {G/H}
			\arrow["{\Omega\pi_\rho}", from=1-1, to=1-2]
			\arrow["\psi"', from=1-1, to=2-1]
			\arrow["\theta", from=1-2, to=2-2]
			\arrow["{\Omega\pi_{\phi}}", from=2-1, to=2-2]
			\arrow["{\Omega\pi_\psi}"', from=2-1, to=3-1]
			\arrow["\lambda", from=2-2, to=3-2]
			\arrow["id"', from=3-1, to=3-2]
		\end{tikzcd}\]
		To see that the top square commutes, note that we have a commutative square obtained by applying $\BB$ to the initial commutative triangle:
		\[\begin{tikzcd}
			BK & BH \\
			BK & BG
			\arrow["B\rho", from=1-1, to=1-2]
			\arrow["id"', from=1-1, to=2-1]
			\arrow["B\psi", from=1-2, to=2-2]
			\arrow["B\phi"', from=2-1, to=2-2]
		\end{tikzcd}\]

		Extending the horizontal arrows into P\"uppe sequences by repeatedly taking fibers yields the larger commutative diagram:
		\[\begin{tikzcd}
			H & {H/K} & BK & BH \\
			G & {G/K} & BK & BG
			\arrow[from=1-1, to=1-2]
			\arrow["\psi", from=1-1, to=2-1]
			\arrow[from=1-2, to=1-3]
			\arrow[from=1-2, to=2-2, "\theta"]
			\arrow["{B\rho}", from=1-3, to=1-4]
			\arrow["id", from=1-3, to=2-3]
			\arrow["{B\psi}", from=1-4, to=2-4]
			\arrow[from=2-1, to=2-2]
			\arrow[from=2-2, to=2-3]
			\arrow["{B\phi}"', from=2-3, to=2-4]
		\end{tikzcd}\]
		But the lefthand square of the above diagram is the top square of the diagram whose commutativity we wish to prove. The lower square of the diagram is proved to be commutative by applying the same argument to the following commutative square:
		\[\begin{tikzcd}
			BK & BG \\
			BH & BG
			\arrow["B\phi", from=1-1, to=1-2]
			\arrow["B\rho"', from=1-1, to=2-1]
			\arrow["id", from=1-2, to=2-2]
			\arrow["B\psi"', from=2-1, to=2-2]
		\end{tikzcd}\]

	\end{proof}

	\begin{rmk}
		The above theorem uses that any morphism of $\EE_n$-groups, for $n>1$, admits a canonical normality datum. We warn the reader however that even in the Abelian case, e.g.~$n=\infty$, this normality datum is not necessarily unique. One example of an infinite loop map with more than one normality datum is given by the terminal morphism $K(\ZZ,1)\times K(\ZZ,3)\to \ast$ of Example \ref{exam: kirans EM space nonunique normality data example}.
	\end{rmk}

	For $H$ a normal subgroup of a discrete group $G$, the Third Isomorphism Theorem \textit{assumes} that there is an order reversing correspondence between subgroups of $G$ containing $H$ and subgroups of $G/H$. We have seen above that this need not be the case in our more general context. Nonetheless, if $\phi\colon H\to G$ has a normality datum, we can classify normal maps to $G/H$ in terms of morphisms, in a slice category, out of the normality datum itself.

	\begin{thm}\label{thm:normality correspondence}
		Let $\phi\colon H\to G$ be a normal map of group objects in $\XX$ with normality datum $\pi\colon \BB G\to \BB (G/H)$. Then there is an equivalence of $\infty$-categories \[\Nrml(G)^{\backslash\pi}\simeq\Nrml(G/H).\]
	\end{thm}

	\begin{proof}
		Recall the equivalence of Proposition \ref{prop:normality data are effective epis from G} which identifies $\Nrml(G)$ with the full subcategory of $\Grp(\XX)^{\backslash G}$ spanned by the effective epimorphisms $G\to Q$. Write $\Grp(\XX)^{\backslash G}_{epi}$ for this $\infty$-category. Therefore we have an equivalence $\Nrml(G)^{\backslash \pi}\simeq (\Grp(\XX)^{\backslash G}_{epi})^{\backslash\cech\pi}$. This last category has as objects commutative triangles of the following form:
		\[\begin{tikzcd}
			G & {G/H} \\
			Q
			\arrow["{\cech\pi}", from=1-1, to=1-2]
			\arrow[two heads, from=1-1, to=2-1]
			\arrow[from=1-2, to=2-1]
		\end{tikzcd}\]
		where $G\twoheadrightarrow Q$ is an effective epimorphism. But this implies that $G/H\to Q$ must also be an effective epimorphism (by Corollary \ref{cor:effective epis right cancel in groups}). Therefore this category is equivalent to the full subcategory of $\Grp(\XX)^{\backslash \cech\pi}$ spanned by commutative triangles whose two edges to the cone point are effective epimorphisms. But by \cite[Tag 04Q4]{kerodon} (more concretely, we are using the fact that the inclusion $\{1\}\hookrightarrow\Delta^1$ is right anodyne along with \cite[Proposition 2.1.2.5]{htt}), there is an equivalence $\Grp(\XX)^{\backslash\cech\pi}\simeq \Grp(\XX)^{\backslash G/H}$. Therefore the category $(\Grp(\XX)^{\backslash G}_{epi})^{\backslash \cech\pi}$ is equivalent to $\Grp(\XX)^{\backslash G/H}_{epi}$ which is, by another application of Proposition \ref{prop:normality data are effective epis from G}, equivalent to $\Nrml(G/H)$.
	\end{proof}

	\section{Questions}

	We record two questions that arise naturally from the above considerations.

	\subsection{An Alternative Characterization of Normality}

	Normality can classically be determined by comparing right and left cosets. In this paper we focused exclusively on right action objects, but there is a dual notion of \textit{left action object} which behaves identically. It then follows readily from the work of Section \ref{sec: action groupoids} that right actions by $G$ are in bijection with left actions of $G$ (this is not the case for \textit{monoid} actions). This is in analogy with the classical case in which $G$-actions on a set are in bijection with $G^{\op}$-actions. We can therefore form the left quotient and right quotient of an action by a group object.

	In the case of a morphism of group objects $\phi\colon H\to G$ we can also read off these left and right actions by thinking of it as a morphism of $\EE_1$-algebras as in \cite[Corollary 3.4.1.7]{ha}, which implies that the data of $\phi$ is equivalent to that of an algebra structure on $G$ in the category of $H$-bimodules. Moreover, the quotient by the left action induced by $\phi$ and the quotient by the right action induced by $\phi$ must be equivalent as objects of $\XX$, since both are the fiber of $\BB\phi\colon\BB H\to \BB G$.

	Now recall that, in classical group theory, $H<G$ is a normal subgroup if and only if the left (or right) cosets of $H$ in $G$ are the same as the double cosets by $H$, i.e.~the sets of the form $HgH$. This suggests the following classification of normality data.

	\begin{ques}
		Let $\phi\colon H\to G$ be a morphism in $\Grp(\XX)$ so that \[G\otimes_H\ast\simeq G/H\simeq\colim(\phi)\simeq H\backslash G\simeq \ast\otimes_HG\] Noticing that $G\otimes_H \ast$ has a residual \textit{left} $H$-action and $\ast\otimes_H G$ has a residual \textit{right} $H$-action, we can take further quotients to obtain double coset objects. Is there a bijection between normality data for $\phi$ and equivalences in $\XX$ of the form $\ast\otimes_H(G\otimes_H\ast)\xrightarrow{\simeq}(\ast\otimes_H G)\otimes_H \ast$?
	\end{ques}

	\subsection{Galois Theory}

	In \cite{rog}, Rognes introduced Galois extensions of ring spectra for what he called \textit{stably dualizable groups}. Examples of stably dualizable groups include compact Lie groups and Eilenberg-MacLane spaces (after localizing at the Morava K-theories). In practice, Galois extensions of ring spectra in which the Galois group is not discrete, and typically finite, seem uncommon. Rognes does however give examples of Galois extensions for more general groups by taking $\mathbb{F}_p$-valued cochain spectra on principal $G$-bundles \cite[Proposition 5.6.3 (a)]{rog}. Rognes also describes a Galois correspondence in \cite[Theorem 7.2.3]{rog} and \cite[Theorem 11.2.2]{rog}, but under somewhat strict conditions. For instance, his result going from intermediate extensions to subgroups requires $G$ to be finite and discrete. Given that the Galois correspondence is extremely complicated (and rarely a bijection) even in the case of extensions of discrete rings, we cannot expect to obtain a perfect correspondence in this significantly more general setting, but one might still ask how far the analogy can be taken.

	\begin{ques}
		Let $G\in\Grp(\XX)$ be a group object in an $\infty$-topos. For a morphism of algebra objects $A\to B$ in $\mathrm{Stab}(\XX)$, the stabilization of $\XX$, with $G$ acting on $B$ by equivalences of $A$-bimodules, we can repeat Rognes' definition of $A\to B$ being a $G$-Galois extension. Given a normal map $H\to G$, perhaps one which is additionally a monomorphism, can we produce a composite $A\to C\to B$ where $C\to B$ is an $H$-Galois extension and $A\to C$ is a $G/H$-Galois extension? On the other hand, given a factorization $A\to C\to B$ of a $G$-Galois extension $A\to B$, when is it possible to produce a normal map $\BiMod{C}{C}(A,A)^{\times}\to G$?
	\end{ques}

	In the slightly more general case of so-called \textit{Hopf-Galois extensions}, in which one replaces groups with Hopf-algebras and $G$-actions with $H$-coactions, partial results can be obtained by combining \cite{beardsrelative,beardsleybialgebras, ergus-thesis}.

	\printbibliography
	\end{document}